\documentclass[11pt]{amsart}
\usepackage[T1]{fontenc}
\usepackage[latin1]{inputenc}
\usepackage{a4wide}
\usepackage{setspace}
\usepackage{xypic}
\usepackage{amssymb}
\usepackage{amsthm}
\usepackage[english]{babel}
\usepackage{latexsym}
\date{}
\usepackage{srcltx}
\usepackage{amscd}
\newtheorem{thm}{Theorem}[section]

\newtheorem{defn}[thm]{Definition}
\newtheorem{rem}[thm]{Remark}
\newtheorem{conj}[thm]{Conjecture}

\newtheorem{conjL'}[thm]{Conjecture $\textbf{L}(\mathcal{X}_{et},d)_{\geq0}$}

\newtheorem{prop}[thm]{Proposition}

\newtheorem{lem}[thm]{Lemma}
\newtheorem{cor}[thm]{Corollary}
\newtheorem{notation}[thm]{Notation}
\newenvironment{f-proof}[1][\sc Proof.]{\begin{trivlist}
\item[\hskip \labelsep {\bfseries #1}]}{\hfill{$\square$}\end{trivlist}}

\newcommand{\Prod}{\displaystyle\prod}

\newcommand{\bq}{\mathbb Q}
\newcommand{\bz}{\mathbb Z}
\newcommand{\br}{\mathbb R}
\newcommand{\bc}{{\mathbb C}}

\newcommand{\et}{\mathrm{et}}

\newcommand{\X}{\mathcal X}
\newcommand{\Y}{\mathcal Y}
\newcommand{\p}{\mathfrak p}
\newcommand{\mydet}{\mathrm{det}}
\DeclareMathOperator{\Spec}{Spec}
\DeclareMathOperator{\ord}{ord}

\begin{document}
\title[Zeta functions at $s=0$]{Zeta functions of regular arithmetic schemes at $s=0$}
\author{Baptiste Morin}

\maketitle

\begin{abstract}
In \cite{Lichtenbaum} Lichtenbaum conjectured the existence of a Weil-\'etale cohomology in order to describe the vanishing order and the special value of the Zeta function of an arithmetic scheme $\mathcal{X}$ at $s=0$ in terms of Euler-Poincar\'e characteristics. Assuming the (conjectured) finite generation of some \'etale motivic cohomology groups we construct such a cohomology theory for regular schemes proper over $\mathrm{Spec}(\mathbb{Z})$. In
particular, we obtain (unconditionally) the right Weil-\'etale cohomology for geometrically cellular schemes over number rings.  We state a conjecture expressing the vanishing order and the special value up to sign of the Zeta function $\zeta(\mathcal{X},s)$ at $s=0$ in terms of a perfect complex of abelian groups $R\Gamma_{W,c}(\mathcal{X},\mathbb{Z})$. Then we relate this conjecture to Soul\'e's conjecture and to the Tamagawa number conjecture of Bloch-Kato, and deduce its validity in simple cases.
\end{abstract}

\vspace{0.5cm}


\footnote{B. Morin\\
CNRS\\
Universit\'e Paul Sabatier, Institut de Math\'ematiques de Toulouse,\\
118 route de Narbonne, 31062 Toulouse cedex 9, France\\
e-mail: baptiste.morin@math.univ-toulouse.fr\\
}

\section{Introduction}

In \cite{Lichtenbaum} Lichtenbaum conjectured the existence of a Weil-\'etale cohomology in order to describe the vanishing order and the special value of the Zeta function of an arithmetic scheme at $s=0$ in terms of Euler-Poincar\'e characteristics. More precisely, we have the following
\begin{conj}\label{conjLichtenbaumIntro}(Lichtenbaum)
On the category of separated schemes of finite type $\X\rightarrow\Spec(\mathbb{Z})$, there exists a cohomology theory given by abelian groups $H_{W,c}^i(\mathcal{X},\mathbb{Z})$ and real vector spaces $H_{W}^i(\mathcal{X},\tilde{\mathbb{R}})$ and $H_{W,c}^i(\mathcal{X},\tilde{\mathbb{R}})$ such that the following holds.
\begin{enumerate}
\item The groups $H_{W,c}^i(\mathcal{X},\mathbb{Z})$ are finitely generated and zero for $i$ large.
\item The natural map from $\mathbb{Z}$ to $\tilde{\mathbb{R}}$-coefficients induces isomorphisms
$$H_{W,c}^i(\mathcal{X},\mathbb{Z})\otimes\mathbb{R}\simeq H_{W,c}^i(\mathcal{X},\tilde{\mathbb{R}}).$$
\item There exists a canonical class $\theta\in H_W^1(\mathcal{X},\tilde{\mathbb{R}})$ such that cup-product with $\theta$ turns the sequence
$$...\xrightarrow{\cup\theta} H_{W,c}^i(\mathcal{X},\tilde{\mathbb{R}})\xrightarrow{\cup\theta} H_{W,c}^{i+1}(\mathcal{X},\tilde{\mathbb{R}})\xrightarrow{\cup\theta}...$$
into a bounded acyclic complex of finite dimensional vector spaces.
\item The vanishing order of the zeta function $\zeta(\mathcal{X},s)$ at $s=0$ is given by the formula
$$\mbox{\emph{ord}}_{s=0}\zeta(\mathcal{X},s)=\sum_{i\geq0}(-1)^i\cdot i\cdot \mbox{\emph{rank}}_{\mathbb{Z}}H^i_{W,c}(\mathcal{X},\mathbb{Z})$$
\item The leading coefficient $\zeta^*(\mathcal{X},0)$ in the Taylor expansion of $\zeta(\mathcal{X},s)$ at $s=0$ is given up to sign by
$$\mathbb{Z}\cdot\lambda(\zeta^*(\mathcal{X},0)^{-1})=\bigotimes_{i\in\mathbb{Z}}\mbox{\emph{det}}_{\mathbb{Z}}
H^i_{W,c}(\mathcal{X},\mathbb{Z})^{(-1)^i}$$
where $\lambda:\mathbb{R}\stackrel{\sim}{\rightarrow} (\bigotimes_{i\in\mathbb{Z}}\mbox{\emph{det}}_{\mathbb{Z}}
H^i_{W,c}(\mathcal{X},\mathbb{Z})^{(-1)^i})\otimes \br$ is induced by \emph{(2)} and \emph{(3)}.

\end{enumerate}
\end{conj}
Such a cohomology theory for smooth varieties over finite fields was defined in \cite{Lichtenbaum-finite-field}, but similar attempts to construct such cohomology groups for flat arithmetic schemes failed.  In \cite{Lichtenbaum} Lichtenbaum gave the first construction for number rings. He defined a Weil-\'etale topology which bears the same relation to the usual \'etale topology as the Weil group does to the Galois group. Under a vanishing statement, he was able to show that his cohomology miraculously yields the value of Dedekind zeta functions at $s=0$. But this cohomology with coefficients in $\mathbb{Z}$ was then shown by Flach to be infinitely generated hence non-vanishing in even degrees $i\geq4$ \cite{MatFlach}. Consequently, Lichtenbaum's complex computing the cohomology with $\mathbb{Z}$-coefficients needs to be artificially truncated in the case of number rings, and is not helpful for flat schemes of dimension greater than $1$. However, the Weil-\'etale topology yields the expected cohomology with $\tilde{\mathbb{R}}$-coefficients, and this fact extends to higher dimensional arithmetic schemes \cite{Flach-moi}. 

The first goal of this paper is to define the right Weil-\'etale cohomology with $\mathbb{Z}$-coefficients for arithmetic schemes satisfying the following conjecture \ref{conjLintro} (see Theorem \ref{thm1intro} below), in order to state a precise version of Conjecture \ref{conjLichtenbaumIntro}. We consider a regular scheme $\mathcal{X}$ proper over $\mathrm{Spec}(\mathbb{Z})$ of pure Krull dimension $d$, and we denote by $\mathbb{Z}(d)$ Bloch's cycle complex.
\begin{conj}\label{conjLintro}(Lichtenbaum)
The \'etale motivic cohomology groups $H^i(\mathcal{X}_{et},\mathbb{Z}(d))$ are finitely generated for $0\leq i\leq 2d$.
\end{conj}
Using some purity for the cycle complex $\mathbb{Z}(d)$ on the \'etale site (which is the key result of \cite{Geisser-Duality}) we construct in the appendix a class $\mathcal{L}(\mathbb{Z})$ (see Definition \ref{def-C(Z)}) of separated schemes of finite type over $\mathrm{Spec}(\mathbb{Z})$ satisfying Conjecture \ref{conjLintro}. This class $\mathcal{L}(\mathbb{Z})$ contains any geometrically cellular scheme over a number ring (see Definition \ref{def-cellular}), and includes the class $A(\mathbb{F}_q)$ of smooth  projective varieties over $\mathbb{F}_q$ which can be constructed out of products of smooth projective curves by union, base extension, blow-ups and quasi-direct summands in the category of Chow motives ($A(\mathbb{F}_q)$ was introduced by Soul\'e in \cite{Soul84}).

In order to state our first main result, we need to fix some notation. For a scheme $\mathcal{X}$ separated and of finite type over $\Spec(\bz)$, we consider the quotient topological space $\mathcal{X}_{\infty}:=\mathcal{X}(\mathbb{C})/G_{\mathbb{R}}$ where $\mathcal{X}(\mathbb{C})$ is endowed with the complex topology and we set $\overline{\mathcal{X}}:=(\mathcal{X},\mathcal{X}_{\infty})$. We denote by $\overline{\X}_{et}$ the Artin-Verdier \'etale topos of $\overline{\X}$ which comes with a closed embedding $u_{\infty}:Sh(\X_{\infty})\rightarrow \overline{\X}_{et}$ in the sense of topos theory, where $Sh(\X_{\infty})$ is the category of sheaves on the topological space $\X_{\infty}$  (see \cite{Flach-moi} Section 4). The Weil-\'etale topos over the archimedean place $\mathcal{X}_{\infty,W}$ is associated to the trivial action of the topological group $\mathbb{R}$ on the topological space $\mathcal{X}_{\infty}$ (see \cite{Flach-moi} Section 6). Our first main result is the following 

\begin{thm}\label{thm1intro}
Let $\mathcal{X}$ be a regular scheme proper over $\Spec(\bz)$. Assume that $\mathcal{X}\in\mathcal{L}(\bz)$, or more generally that the connected components of $\X$ satisfy Conjecture \ref{conjLintro}. Then there exist complexes $\mbox{\emph{R}}\Gamma_{W}(\overline{\mathcal{X}},\mathbb{Z})$ and $\mbox{\emph{R}}\Gamma_{W,c}(\mathcal{X},\mathbb{Z})$ such that the following properties hold.
\begin{itemize}
\item If $\X$ has pure dimension $d$, there is an exact triangle
\begin{equation}\label{fundamentaltriangle}
\mbox{\emph{RHom}}(\tau_{\geq0}\mbox{\emph{R}}\Gamma(\mathcal{X},\mathbb{Q}(d)),\mathbb{Q}[-2d-2])\rightarrow \mbox{\emph{R}}\Gamma(\overline{\mathcal{X}}_{et},\mathbb{Z})\rightarrow \mbox{\emph{R}}\Gamma_W(\overline{\mathcal{X}},\mathbb{Z}).
\end{equation}
\item The complex $\mbox{\emph{R}}\Gamma_{W}(\overline{\mathcal{X}},\mathbb{Z})$ is contravariantly functorial.

\item There exists a unique morphism
$i^*_{\infty}:\mbox{\emph{R}}\Gamma_W(\overline{\mathcal{X}},\mathbb{Z})\rightarrow \mbox{\emph{R}}\Gamma(\mathcal{X}_{\infty,W},\mathbb{Z})$ which renders the following square commutative.
\[ \xymatrix{
 \mbox{\emph{R}}\Gamma(\overline{\mathcal{X}}_{et},\mathbb{Z})\ar[d]^{u^*_{\infty}}\ar[r]& \mbox{\emph{R}}\Gamma_W(\overline{\mathcal{X}},\mathbb{Z})\ar[d]^{i^*_{\infty}}\\
\mbox{\emph{R}}\Gamma(\mathcal{X}_{\infty},\mathbb{Z})\ar[r]&\mbox{\emph{R}}\Gamma(\mathcal{X}_{\infty,W},\mathbb{Z})
}
\]
We define $\mbox{\emph{R}}\Gamma_{W,c}(\mathcal{X},\mathbb{Z})$ so that there is an exact triangle
$$\mbox{\emph{R}}\Gamma_{W,c}(\mathcal{X},\mathbb{Z})\rightarrow\mbox{\emph{R}}\Gamma_{W}(\overline{\mathcal{X}},\mathbb{Z})\stackrel{i_{\infty}^*}{\rightarrow} \mbox{\emph{R}}\Gamma(\mathcal{X}_{\infty,W},\mathbb{Z}).$$

\item The cohomology groups $H_{W}^i(\overline{\mathcal{X}},\mathbb{Z})$ and $H_{W,c}^i(\mathcal{X},\mathbb{Z})$ are finitely generated for all $i$ and zero for $i$ large.
\item The cohomology groups $H_{W}^i(\overline{\mathcal{X}},\mathbb{Z})$ form an integral model for $l$-adic cohomology: for any prime number $l$ and any $i\in\mathbb{Z}$ there is a canonical isomorphism
    $$H_{W}^i(\overline{\mathcal{X}},\mathbb{Z})\otimes\mathbb{Z}_l\simeq H^i(\overline{\mathcal{X}}_{et},\mathbb{Z}_l).$$
\item  If $\mathcal{X}$ has characteristic $p$ then there is a canonical isomorphism in the derived category
$$\mbox{\emph{R}}\Gamma(\mathcal{X}_{W},\mathbb{Z})\stackrel{\sim}{\longrightarrow} \mbox{\emph{R}}\Gamma_W(\mathcal{X},\mathbb{Z})$$
where $\mbox{\emph{R}}\Gamma(\mathcal{X}_{W},\mathbb{Z})$ is the cohomology of the Weil-\'etale topos \cite{Lichtenbaum-finite-field} and $\mbox{\emph{R}}\Gamma_W(\mathcal{X},\mathbb{Z})$ is the complex defined in this paper. Moreover the exact triangle (\ref{fundamentaltriangle}) is isomorphic to Geisser's triangle (\cite{Geisser-Weiletale} Corollary 5.2 for $\mathcal{G}^{\cdot}=\bz$). 
\item If $\mathcal{X}=\mbox{\emph{Spec}}(\mathcal{O}_F)$ is the spectrum of a totally imaginary number ring then there is a canonical isomorphism in the derived category
$$\mbox{\emph{R}}\Gamma_{W}(\overline{\mathcal{X}},\mathbb{Z})\stackrel{\sim}{\longrightarrow}\tau_{\leq3}\mbox{\emph{R}}\Gamma(\overline{\mathcal{X}}_{W},\mathbb{Z})$$
where $\tau_{\leq3}\mbox{\emph{R}}\Gamma(\overline{\mathcal{X}}_{W},\mathbb{Z})$ is the truncation of Lichtenbaum's complex \cite{Lichtenbaum}
and $\mbox{\emph{R}}\Gamma_{W}(\overline{\mathcal{X}},\mathbb{Z})$ is the complex defined in this paper.  
\end{itemize}
\end{thm}
Theorem \ref{thm1intro} is proven in \ref{def-cohomology}, \ref{finitelygenerated-cohomology}, \ref{cor-functoriality}, \ref{cor-comp-l-adic}, \ref{thm-comparison-char-p}, \ref{thm-compare-nbrrings} and \ref{prop-iinfty}. Notice that the same formalism is used to treat flat arithmetic schemes and schemes over finite fields (see \cite{Lichtenbaum} Question 1 in the Introduction). The basic idea behind the proof of Theorem \ref{thm1intro} can be explained as follows. The Weil group is defined as an extension of the Galois group by the id\`ele class group corresponding to the fundamental class of class field theory (more precisely, as the limit of these group extensions). In this paper we use \'etale duality for arithmetic schemes rather than class field theory, in order to obtain a canonical "extension" of the \'etale $\mathbb{Z}$-cohomology by the dual of motivic $\mathbb{Q}(d)$-cohomology. More precisely, our Weil-\'etale complex is defined as the cone of a map
$$\alpha_{\X}:\mbox{RHom}(\tau_{\geq0}\mbox{R}\Gamma(\mathcal{X},\mathbb{Q}(d)),\mathbb{Q}[-2d-2])\stackrel{}{\longrightarrow} \mbox{R}\Gamma(\overline{\mathcal{X}}_{et},\mathbb{Z}),$$
where $\alpha_{\X}$ is constructed out of \'etale duality. Here the scheme $\X$ is of pure dimension $d$. This idea was suggested by works of Burns \cite{Burns}, Geisser \cite{Geisser-Weiletale} and the author \cite{On the WE}. The techniques involved in this paper rely on results due to Geisser and Levine on Bloch's cycle complex (see \cite{Geisser-MotvicCohDed}, \cite{Geisser-Duality}, \cite{Geisser-Levine00}, \cite{Geisser-Levine01} and \cite{Levine-localization}).

From now on, Conjecture 1.1 is understood as a list of expected properties for the complex $\mbox{R}\Gamma_{W,c}(\mathcal{X},\mathbb{Z})$ of Theorem \ref{thm1intro}. The second goal of this paper is to relate Conjecture \ref{conjLichtenbaumIntro} to Soul\'e's conjecture \cite{Soule} and to the Tamagawa number conjecture of Bloch-Kato (in the formulation of Fontaine and Perrin-Riou, see \cite{Fontaine92} and \cite{Fontaine-Perrin-Riou}). To this aim, we need to assume the following conjecture, which is a special case of a natural refinement for arithmetic schemes of the classical conjecture of Beilinson relating motivic cohomology to Deligne cohomology.  Let $\mathcal{X}$ be a regular, proper and flat arithmetic scheme of pure dimension $d$.
\begin{conj}\label{confBFintro}
(Beilinson) The Beilinson regulator $$H^{2d-1-i}(\mathcal{X},\mathbb{Q}(d))_{\mathbb{R}}\rightarrow H^{2d-1-i}_{\mathcal{D}}(\mathcal{X}_{/\mathbb{R}},\mathbb{R}(d))$$ is an isomorphism for $i\geq 1$ and there is an
exact sequence
$$0\rightarrow
H^{2d-1}(\mathcal{X},\mathbb{Q}(d))_{\mathbb{R}}\rightarrow H^{2d-1}_{\mathcal{D}}(\mathcal{X}_{/\mathbb{R}},\mathbb{R}(d))\rightarrow CH^0(\mathcal{X}_{\mathbb{Q}})^*_{\mathbb{R}}\rightarrow 0
$$
\end{conj}

Now we can state our second main result. If $\mathcal{X}$ is defined over a number ring $\mathcal{O}_F$ we set $\X_F=\X\otimes_{\mathcal{O}_F}F$ and $\X_{\overline{F}}=\X\otimes_{\mathcal{O}_F}\overline{F}$.
\begin{thm}\label{the2Intro} Assume that $\mathcal{X}$ satisfies Conjectures \ref{conjLintro} and \ref{confBFintro}.
\begin{itemize}
\item Statements \emph{(1)}, \emph{(2)} and \emph{(3)} of Conjecture \ref{conjLichtenbaumIntro} hold for $\X$.
\item Assume that $\X$ is projective over $\bz$. Then Statement \emph{(4)} of Conjecture \ref{conjLichtenbaumIntro}, i.e. the identity
$$\mbox{\emph{ord}}_{s=0}\zeta(\mathcal{X},s)=\sum_{i\geq0}(-1)^i\cdot i\cdot \mbox{\emph{rank}}_{\mathbb{Z}}H^i_{W,c}(\mathcal{X},\mathbb{Z})$$
is equivalent to Soul\'e's Conjecture \cite{Soule} for the vanishing order of $\zeta(\mathcal{X},s)$ at $s=0$.
\item Assume that $\mathcal{X}$ is smooth projective over a number ring $\mathcal{O}_F$, and that the representations $H^i(\X_{\overline{F},\et},\bq_l)$ of $G_F$ satisfies $H^1_f(F,H^i(\X_{\overline{F},\et},\bq_l))=0$. Then Statement \emph{(5)} of Conjecture \ref{conjLichtenbaumIntro}, i.e. the identity
$$\mathbb{Z}\cdot\lambda(\zeta^*(\mathcal{X},0)^{-1})=\bigotimes_{i\in\mathbb{Z}}\mbox{\emph{det}}_{\mathbb{Z}}
H^i_{W,c}(\mathcal{X},\mathbb{Z})^{(-1)^i}$$
is equivalent to the Tamagawa number conjecture \cite{Fontaine92} for the motive $\bigoplus_{i=0}^{2d-2}h^i(\mathcal{X}_{F})[-i]$.
\end{itemize}
\end{thm}
The proof of the last statement of Theorem \ref{the2Intro} was already given in \cite{Flach-moi}, assuming expected properties of Weil-\'etale cohomology. This proof is based on the fact that $\bigotimes_{i\in\mathbb{Z}}\mathrm{det}_{\mathbb{Z}}
H^i_{W,c}(\mathcal{X},\mathbb{Z})^{(-1)^i}$ provides the fundamental line (in the sense of \cite{Fontaine92}) with a canonical $\bz$-structure. This result shows that the Weil-\'etale point of view is compatible with the Tamagawa number conjecture of Bloch-Kato, answering a question of Lichtenbaum (see \cite{Lichtenbaum} Question 2 in the Introduction).

We obtain examples of flat arithmetic schemes satisfying Conjecture \ref{conjLichtenbaumIntro}. 

\begin{thm}
For any number field $F$, Conjecture \ref{conjLichtenbaumIntro} holds for $\X=\Spec(\mathcal{O}_F)$.

Let $\X$ be a smooth projective scheme over the number ring $\mathcal{O}_{\X}(\X)=\mathcal{O}_F$, where $F$ is an abelian number field.  Assume that $\X_F$ admits a smooth cellular decomposition (see Definition \ref{def-cellular}) and that $\X\in \mathcal{L}(\mathbb{Z})$ (see Definition \ref{def-C(Z)}). Then Conjecture \ref{Lichtenbaum Conjecture} holds for $\X$.
\end{thm}
We refer to \cite{Deninger08} for a proof of a dynamical system analogue of Conjecture \ref{conjLichtenbaumIntro}.

\vspace{0.5cm}

{\bf {Notations.}} We denote by $\mathbb{Z}$ the ring of integers, and by $\mathbb{Q}$, $\mathbb{Q}_p$, $\mathbb{R}$ and $\mathbb{C}$ the fields of rational, $p$-adic, real and complex numbers respectively. An arithmetic scheme is a scheme which is separated and of finite type over $\mathrm{Spec}(\mathbb{Z})$. An arithmetic scheme $\X$ is said to be proper (respectively flat) if the map $\X\rightarrow \mathrm{Spec}(\mathbb{Z})$ is proper (respectively flat).

For a field $k$, we choose a separable closure $\bar{k}/k$ and we denote by $G_k=\textrm{Gal}(\bar{k}/k)$ the absolute Galois group of $k$. For a  proper scheme $\mathcal{X}$ over $\Spec(\mathbb{Z})$ we denote by $\mathcal{X}_{\infty}:=\mathcal{X}(\mathbb{C})/G_{\mathbb{R}}$ the quotient topological space of $\mathcal{X}(\mathbb{C})$ by $G_{\mathbb{R}}$ where $\mathcal{X}(\mathbb{C})$ is given with the complex topology. We consider the Artin-Verdier \'etale topos $\overline{\mathcal{X}}_{et}$ given with the open-closed decomposition of topoi
$$\varphi:\mathcal{X}_{et}\rightarrow\overline{\mathcal{X}}_{et}\leftarrow Sh(\mathcal{X}_{\infty}):u_{\infty}$$
where $\mathcal{X}_{et}$ is the usual \'etale topos of the scheme $\mathcal{X}$ (i.e. the category of sheaves of sets on the small \'etale site of $\mathcal{X}$) and $Sh(\mathcal{X}_{\infty})$ is the category of sheaves of sets on the topological space $\mathcal{X}_{\infty}$ (see \cite{Flach-moi} Section 4). For any abelian sheaf $A$ on ${\mathcal{X}}_{et}$ and any $n>0$ the sheaf $R^n\varphi_*A$ is a 2-torsion sheaf concentrated on $\mathcal{X}(\mathbb{R})$. In particular, if $\mathcal{X}(\mathbb{R})=\emptyset$ then $\varphi_*\simeq R\varphi_*$.
For $\mathcal{X}(\mathbb{C})=\emptyset$ one has $\mathcal{X}_{et}=\overline{\mathcal{X}}_{et}$. For $T=\mathcal{X}_{et},\overline{\mathcal{X}}_{et}$ or $Sh(\mathcal{X}_{\infty})$, or more generally for any Grothendieck topos $T$, we denote by $\Gamma(T,-)$ the global section functor, by $R\Gamma(T,-)$ its total right derived functor and by $H^i(T,-):=H^i(R\Gamma(T,-))$ its cohomology.

Let $\mathbb{Z}(n):=z^n(-,2n-*)$ be Bloch's cycle complex (see \cite{Bloch86}, \cite{Geisser-MotvicCohDed}, \cite{Geisser-MotCoh-K-theory}, \cite{Levine-motivic-and-K-theory}  and \cite{Levine-localization}), which we consider as a complex of abelian sheaves on the small \'etale (or Zariski) site of $\mathcal{X}$. For an abelian group $A$ we define $A(n)$ to be $\mathbb{Z}(n)\otimes A$. Note that $\mathbb{Z}(n)$ is a complex of flat sheaves, hence tensor product and derived tensor product with $\mathbb{Z}(n)$ agree. We denote by $H^i(\mathcal{X}_{et},A(n))$ the \'etale hypercohomology of $A(n)$, and by $H^i(\mathcal{X},A(n)):=H^i(\mathcal{X}_{Zar},A(n))$ its Zariski hypercohomology. For $\mathcal{X}(\mathbb{R})=\emptyset$, we still denote by $A(n)=\varphi_*A(n)$ the push-forward of the cycle complex $A(n)$ on $\overline{\mathcal{X}}_{et}$, and by $H^i(\overline{\mathcal{X}}_{et},A(n))$ the \'etale hypercohomology of $A(n)$ (in this paper $A=\mathbb{Z},\mathbb{Z}/m\mathbb{Z},\mathbb{Q}$ or $\mathbb{Q}/\mathbb{Z}$). Notice that the complex of \'etale sheaves $\mathbb{Z}/m\mathbb{Z}(n)$ is not in general quasi-isomorphic to $\mu_m^{\otimes n}[0]$ (however, see \cite{Geisser-MotvicCohDed} Theorem 1.2(4)).

We denote by $\mathcal{D}$ the derived category of the category of abelian groups. More generally, we write $\mathcal{D}(R)$ for the derived category of the category of $R$-modules, where $R$ is a commutative ring. An exact (i.e. distinguished) triangle in $\mathcal{D}(R)$ is (somewhat abusively) depicted as follows
$$C'\rightarrow C\rightarrow C''$$
where $C$, $C'$ and $C''$ are objects of $\mathcal{D}(R)$. For an object $C$ of $\mathcal{D}(R)$ we write $C_{\leq n}$ (respectively $C_{\geq n}$) for the truncated complex so that $H^i(C_{\leq n})=H^i(C)$ for $i\leq n$ and $H^i(C_{\leq n})=0$ otherwise (respectively  so that $H^i(C_{\geq n})=H^i(C)$ for $i\geq n$ and $H^i(C_{\geq n})=0$ otherwise).

Let $R$ be a principal ideal domain. For a finitely generated free $R$-module $L$ of rank $r$, we set $\mathrm{det}_{R}(L)=\bigwedge_{R}^{r}L$ and $\mathrm{det}_{R}(L)^{-1}=\mathrm{Hom}_{R}(\mathrm{det}_{R}(L),R)$. For a finitely generated $R$-module $M$, one may choose a resolution given by an exact sequence $0\rightarrow L_{-1}\rightarrow L_0\rightarrow M\rightarrow 0,$ where $L_{-1}$ and $L_0$ are finitely generated free $R$-modules. Then
$$\mathrm{det}_{R}(M):=\mathrm{det}_{R}(L_0)\otimes_{R} \mathrm{det}_{R}(L_{-1})^{-1}$$
does not depend on the resolution. If $C\in\mathcal{D}(R)$ is a complex such that $H^i(C)$ is finitely generated for all $i\in\mathbb{Z}$ and $H^i(C)=0$ for almost all $i$, we denote $\mathrm{det}_{R}(C):=\bigotimes_{i\in\mathbb{Z}} \mathrm{det}_{R}H^i(C)^{(-1)^{i}}$.

For an abelian group $A$ we write $A_{tor}$ and $A_{div}$ for the maximal torsion and the maximal divisible subgroups of $A$ respectively, and we set $A_{codiv}=A/A_{div}$ and $A_{cotor}=A/A_{tor}$. We denote by ${_n}A$ and $A_n$ the kernel and the cokernel of the map $n:A\rightarrow A$ (multiplication by $n$). We write $A^D=\textrm{Hom}(A,\bq/\bz)$ for the Pontryagin dual of a finite abelian group $A$.

\tableofcontents

\section{Weil-\'etale cohomology}

\subsection{The conjectures $\textbf{L}(\mathcal{X}_{et},d)$ and $\textbf{L}(\mathcal{X}_{et},d)_{\geq0}$}
Let  $\mathcal{X}$ be a proper, regular and connected arithmetic scheme of dimension $d$. The following conjecture is an \'etale version of (a special case of) the motivic Bass Conjecture (see \cite{Kahn-K-theory} Conjecture 37). 

\begin{conj}\label{conj-clef}\emph{$\textbf{L}(\mathcal{X}_{et},d)$}  For $i\leq 2d$, the abelian group $H^i(\mathcal{X}_{et},\mathbb{Z}(d))$ is finitely generated.
\end{conj}

In order to define the Weil-\'etale cohomology we shall consider schemes satisfying a weak version of the previous conjecture.
\begin{conj}\label{conj-clef'} \emph{$\textbf{L}(\mathcal{X}_{et},d)_{\geq0}$} For $0\leq i\leq 2d$,
the abelian group $H^i(\mathcal{X}_{et},\mathbb{Z}(d))$ is finitely generated.
\end{conj}

\subsection{The morphism $\alpha_{\mathcal{X}}$}\label{sect-alpha}
Let $\mathcal{X}$ be a proper regular connected arithmetic scheme. The goal of this section is to construct (conditionally) a certain morphism $\alpha_{\X}$ in the derived category of abelian groups $\mathcal{D}$, in order to define in Section \ref{sect-RGW} the Weil-\'etale complex $\mathrm{R}\Gamma_W(\overline{\X},\mathbb{Z})$ as the cone of $\alpha_{\X}$.

\subsubsection{The case $\mathcal{X}(\mathbb{R})=\emptyset$}\label{sectNotX(R)} In this subsection, $\mathcal{X}$ denotes a proper, regular and connected arithmetic scheme of dimension $d$ such that $\mathcal{X}(\mathbb{R})$ is empty. In particular  $\varphi_*$ is exact (since $\X(\mathbb{R})=\emptyset$), hence $H^i(\overline{\X}_{et},\mathbb{Z}(d)):=H^i(\overline{\X}_{et},\varphi_*\mathbb{Z}(d))\simeq H^i(\X_{et},\mathbb{Z}(d))$.
\begin{lem}\label{firstLemma}
We have
\begin{eqnarray*}
H^{i}(\overline{\mathcal{X}}_{et},\mathbb{Z}(d))&=&0\mbox{ for $i>2d+2$}\\
&\simeq&\mathbb{Q}/\mathbb{Z} \mbox{ for $i=2d+2$}\\
&=&0\mbox{ for $i=2d+1$}\\
&\simeq&H^{i}(\mathcal{X},\mathbb{Q}(d))\mbox{ for $i<0$.}
\end{eqnarray*}
\end{lem}

\begin{proof}
For any positive integer $n$ and any $i\in\mathbb{Z}$, one has
\begin{eqnarray}
H^{i-1}(\overline{\mathcal{X}}_{et},\mathbb{Z}/n\mathbb{Z}(d))
&=&\mbox{Ext}^{i-1}_{\overline{\mathcal{X}},\mathbb{Z}/n\mathbb{Z}}(\mathbb{Z}/n\mathbb{Z},\mathbb{Z}/n\mathbb{Z}(d))\\
\label{oneiso}&\simeq&\mbox{Ext}^{i}_{\overline{\mathcal{X}}}(\mathbb{Z}/n\mathbb{Z},\mathbb{Z}(d))\\
\label{twoiso}&\simeq& H^{2d+2-i}(\overline{\mathcal{X}}_{et},\mathbb{Z}/n\mathbb{Z})^D
\end{eqnarray}
where (\ref{oneiso}) and (\ref{twoiso}) are given by (\cite{Geisser-Duality} Lemma 2.4) and (\cite{Geisser-Duality} Theorem 7.8) respectively, and $ H^{2d+2-i}(\overline{\mathcal{X}}_{et},\mathbb{Z}/n\mathbb{Z})^D$ denotes the Pontryagin dual of the finite group $H^{2d+2-i}(\overline{\mathcal{X}}_{et},\mathbb{Z}/n\mathbb{Z})$.

The complex $\mathbb{Z}(d)$ consists of flat sheaves, hence tensor product and derived tensor product with $\mathbb{Z}(d)$ agree. Therefore, applying $\mathbb{Z}(d)\otimes^L(-)$ to the exact sequence $0\rightarrow \mathbb{Z}\rightarrow\mathbb{Q}\rightarrow\mathbb{Q}/\mathbb{Z}\rightarrow 0$ yields an exact triangle:
\begin{equation}\label{oneexacttriangle}
\mathbb{Z}(d)\rightarrow\mathbb{Q}(d)\rightarrow\mathbb{Q}/\mathbb{Z}(d).
\end{equation}
Moreover, we have $H^{i}(\overline{\mathcal{X}}_{et},\mathbb{Q}(d))\simeq
H^{i}(\mathcal{X},\mathbb{Q}(d))$ for any $i$ (see \cite{Geisser-MotvicCohDed} Proposition 3.6) and $H^{i}(\mathcal{X},\mathbb{Q}(d))=0$ for any $i>2d$ (see \cite{Levine-motivic-and-K-theory} Lemma 11.1). Hence (\ref{oneexacttriangle})
gives
$$H^{i}(\overline{\mathcal{X}}_{et},\mathbb{Z}(d))\simeq H^{i-1}(\overline{\mathcal{X}}_{et},\mathbb{Q}/\mathbb{Z}(d))
\simeq\underrightarrow{\mbox{lim}}\,H^{i-1}(\overline{\mathcal{X}}_{et},\mathbb{Z}/n\mathbb{Z}(d))$$
for $i\geq2d+2$. The result for $i\geq2d+2$ then follows from (\ref{twoiso}).

By (\cite{SGA4} X.6.2), the scheme $\mathcal{X}$ has $l$-cohomological dimension $2d+1$ for any prime number $l$ (note that $\mathcal{X}(\mathbb{R})=\emptyset$), hence (\ref{twoiso}) shows that $H^{i}(\overline{\mathcal{X}}_{et},\mathbb{Q}/\mathbb{Z}(d))=0$ for $i<0$. The result for $i<0$ now follows from $H^{i}(\overline{\mathcal{X}}_{et},\mathbb{Q}(d))\simeq H^{i}(\mathcal{X},\mathbb{Q}(d))$ and from the exact triangle (\ref{oneexacttriangle}).

It remains to treat the case $i=2d+1$. Assume that $\mathcal{X}$ is flat over $\mathbb{Z}$. Then we have (see \cite{Kato-Saito86} Theorem 6.1(2))
\begin{equation}\label{flat-vanish}
H^{2d}(\mathcal{X},\mathbb{Q}(d))\simeq CH^{d}(\mathcal{X})\otimes\mathbb{Q}=0
\end{equation}
hence $H^{i}(\overline{\mathcal{X}}_{et},\mathbb{Q}(d))=0$ for $i\geq2d$. Here $CH^{d}(\mathcal{X})$ denotes the Chow group of cycles of codimension $d$. We obtain
\begin{eqnarray*}
H^{2d+1}(\overline{\mathcal{X}}_{et},\mathbb{Z}(d))&\simeq& H^{2d}(\overline{\mathcal{X}}_{et},\mathbb{Q}/\mathbb{Z}(d))\\
&\simeq&\underrightarrow{\mbox{lim}}\, (H^{1}(\overline{\mathcal{X}}_{et},\mathbb{Z}/n\mathbb{Z})^D)\\
&\simeq&(\underleftarrow{\mbox{lim}}\,\mbox{Hom}(\pi_1(\overline{\mathcal{X}}_{et}),\mathbb{Z}/n\mathbb{Z}))^D\\
&\simeq&\mbox{Hom}(\pi_1(\overline{\mathcal{X}}_{et})^{ab},\widehat{\mathbb{Z}})^D\\
&=&0
\end{eqnarray*}
where the first isomorphism (respectively the second) is given by (\ref{oneexacttriangle}) (respectively by (\ref{twoiso})), and the last isomorphism follows from the fact that the abelian fundamental group $\pi_1(\overline{\mathcal{X}}_{et})^{ab}$ is finite (see \cite{Kato-Saito86} Theorem 9.10 and \cite{Wiesend07} Corollary 3).

Assume now that $\mathcal{X}$ is a smooth proper scheme over a finite field (hence $\overline{\mathcal{X}}=\mathcal{X}$). One has
\begin{eqnarray}
\label{onebyGeisser}H^{2d}(\mathcal{X}_{et},\mathbb{Q}/\mathbb{Z}(d))&\simeq&\underrightarrow{\mbox{lim}}\, (H^{1}(\mathcal{X}_{et},\mathbb{Z}/n\mathbb{Z})^D)\\
&\simeq&\underrightarrow{\mbox{lim}}\, (\mathrm{Hom}(\pi_1(\mathcal{X}_{et}),\mathbb{Z}/n\mathbb{Z})^D)\\
\label{oenbyCFT}&\simeq&\underrightarrow{\mbox{lim}}\, (CH^d(\mathcal{X})\otimes \mathbb{Z}/n\mathbb{Z})\\
\label{last}&\simeq&CH^d(\mathcal{X})\otimes \mathbb{Q}/\mathbb{Z}
\end{eqnarray}
where (\ref{onebyGeisser}) is given by (\ref{twoiso}) while (\ref{oenbyCFT}) follows from class field theory (see \cite{Wiesend07} Corollary 3). The exact triangle (\ref{oneexacttriangle}) yields an exact sequence
\begin{equation}\label{regla}
H^{2d}(\mathcal{X},\mathbb{Q}(d))\stackrel{p}{\rightarrow} H^{2d}(\mathcal{X}_{et},\mathbb{Q}/\mathbb{Z}(d))\rightarrow H^{2d+1}(\mathcal{X}_{et},\mathbb{Z}(d))\rightarrow 0.
\end{equation}
Moreover, the isomorphism (\ref{last}) is the direct limit of the maps 
$$CH^{d}(\mathcal{X})_n=H^{2d}(\mathcal{X},\mathbb{Z}(d))_n\rightarrow H^{2d}(\mathcal{X},\mathbb{Z}/n\mathbb{Z}(d))\rightarrow H^{2d}(\mathcal{X}_{et},\mathbb{Z}/n\mathbb{Z}(d)).$$
It follows that the map $p$ in the sequence (\ref{regla}) can be identified with the morphism
$$CH^{d}(\mathcal{X})\otimes\mathbb{Q}\rightarrow CH^{d}(\mathcal{X})\otimes\mathbb{Q}/\mathbb{Z}$$
induced by the surjection $\mathbb{Q}\rightarrow \mathbb{Q}/\mathbb{Z}$. Hence $p$ is surjective, since $CH^{d}(\mathcal{X})$ is finitely generated of rank one (\cite{Kato-Saito86} Theorem 6.1), so that $H^{2d+1}(\mathcal{X}_{et},\mathbb{Z}(d))=0$. 

\end{proof}

\begin{notation}
We set $H^{j}(\overline{\mathcal{X}}_{et},\mathbb{Z}(d))_{\geq0}:=H^{j}(\mbox{\emph{R}}\Gamma(\overline{\mathcal{X}}_{et},\mathbb{Z}(d))_{\geq0})$.
\end{notation}

\begin{lem}\label{prop-duality-Z-coefs}
Assume that $\mathcal{X}$ satisfies $\textbf{\emph{L}}(\mathcal{X}_{et},d)_{\geq0}$. Then the natural map
$$H^i(\overline{\mathcal{X}}_{et},\mathbb{Z})\rightarrow \emph{Hom}(H^{2d+2-i}(\overline{\mathcal{X}}_{et},\mathbb{Z}(d))_{\geq0},\mathbb{Q}/\mathbb{Z})$$
is an isomorphism of abelian groups for $i\geq1$.
\end{lem}
\begin{proof}
The map of the lemma is given by the pairing
$$H^i(\overline{\mathcal{X}}_{et},\mathbb{Z})\times \mbox{Ext}^{2d+2-i}_{\overline{\mathcal{X}}}(\mathbb{Z},\mathbb{Z}(d))\rightarrow
H^{2d+2}(\overline{\mathcal{X}}_{et},\mathbb{Z}(d))\simeq\mathbb{Q}/\mathbb{Z}.$$
For $i=1$ the result follows from Lemma \ref{firstLemma}, since one has $H^1(\overline{\mathcal{X}}_{et},\mathbb{Z})=\mbox{Hom}_{cont}(\pi_1(\overline{\mathcal{X}}_{et},p),\mathbb{Z})=0$
because the fundamental group $\pi_1(\overline{\mathcal{X}}_{et},p)$ is profinite (hence compact).

The scheme $\mathcal{X}$ is connected and normal, hence $H^i(\overline{\mathcal{X}}_{et},\mathbb{Q})=H^i(\mathcal{X}_{et},\mathbb{Q})=\mathbb{Q},\,0$ for $i=0$ and $i\geq1$ respectively. We obtain $$H^i(\overline{\mathcal{X}}_{et},\mathbb{Z})=H^{i-1}(\overline{\mathcal{X}}_{et},\mathbb{Q}/\mathbb{Z})
=\underrightarrow{\mbox{lim}}\,H^{i-1}(\overline{\mathcal{X}}_{et},\mathbb{Z}/n\mathbb{Z})$$
for $i\geq2$, since $\mathcal{X}$ is quasi-compact and quasi-separated. For any positive integer $n$ the canonical map
$$\mbox{Ext}^{2d+2-i}_{\overline{\mathcal{X}}}(\mathbb{Z}/n\mathbb{Z},\mathbb{Z}(d))\times H^i(\overline{\mathcal{X}}_{et},\mathbb{Z}/n\mathbb{Z})\rightarrow\mathbb{Q}/\mathbb{Z}$$
is a perfect pairing of finite groups (see \cite{Geisser-Duality} Theorem 7.8). The short exact sequence $$0\rightarrow\mathbb{Z}\rightarrow \mathbb{Z}\rightarrow \mathbb{Z}/n\mathbb{Z}\rightarrow0$$
yields an exact sequence
$$H^{j-1}(\overline{\mathcal{X}}_{et},\mathbb{Z}(d))\rightarrow H^{j-1}(\overline{\mathcal{X}}_{et},\mathbb{Z}(d))\rightarrow \mbox{Ext}^{j}_{\overline{\mathcal{X}}}(\mathbb{Z}/n\mathbb{Z},\mathbb{Z}(d))$$
$$\rightarrow H^{j}(\overline{\mathcal{X}}_{et},\mathbb{Z}(d))\rightarrow H^{j}(\overline{\mathcal{X}}_{et},\mathbb{Z}(d)).$$
We obtain a short exact sequence
$$0\rightarrow H^{j-1}(\overline{\mathcal{X}}_{et},\mathbb{Z}(d))_n\rightarrow \mbox{Ext}^{j}_{\overline{\mathcal{X}}}(\mathbb{Z}/n\mathbb{Z},\mathbb{Z}(d)) \rightarrow {_n}H^{j}(\overline{\mathcal{X}}_{et},\mathbb{Z}(d))\rightarrow 0
$$
for any $j$. By left exactness of projective limits the sequence
$$0\rightarrow\underleftarrow{\mbox{lim}}\,H^{j-1}(\overline{\mathcal{X}}_{et},\mathbb{Z}(d))_n\rightarrow \underleftarrow{\mbox{lim}}\,\mbox{Ext}^{j}_{\overline{\mathcal{X}}}(\mathbb{Z}/n\mathbb{Z},\mathbb{Z}(d)) \rightarrow \underleftarrow{\mbox{lim}}\,_{n}H^{j}(\overline{\mathcal{X}}_{et},\mathbb{Z}(d))$$
is exact.
The module $\underleftarrow{\mbox{lim}}\,_{n}H^{j}(\overline{\mathcal{X}}_{et},\mathbb{Z}(d))$ vanishes for $0\leq j\leq 2d+1$ since $H^{j}(\overline{\mathcal{X}}_{et},\mathbb{Z}(d))$ is assumed to be finitely generated for such $j$. We have $\underleftarrow{\mbox{lim}}\,_{n}H^{j}(\overline{\mathcal{X}}_{et},\mathbb{Z}(d))=0$ for $j<0$ since $H^{j}(\overline{\mathcal{X}}_{et},\mathbb{Z}(d))$ is uniquely divisible for $j<0$ (see Lemma \ref{firstLemma}). This yields an isomorphism of profinite groups
$$\underleftarrow{\mbox{lim}}\,H^{j-1}(\overline{\mathcal{X}}_{et},\mathbb{Z}(d))_n\stackrel{\sim}{\rightarrow} \underleftarrow{\mbox{lim}}\,\mbox{Ext}^{j}_{\overline{\mathcal{X}}}(\mathbb{Z}/n\mathbb{Z},\mathbb{Z}(d))
\stackrel{\sim}{\rightarrow}(\underrightarrow{\mbox{lim}}\,H^{2d+2-j}(\overline{\mathcal{X}}_{et},\mathbb{Z}/n\mathbb{Z}))^D,$$
where the last isomorphism follows from the duality above (and from the fact that $\mbox{Ext}^{j}_{\overline{\mathcal{X}}}(\mathbb{Z}/n\mathbb{Z},\mathbb{Z}(d))$ is finite for any $n$). This gives isomorphisms of torsion groups
$$H^{2d+2-(j-1)}(\overline{\mathcal{X}}_{et},\mathbb{Z})\stackrel{\sim}{\rightarrow}
(\underrightarrow{\mbox{lim}}\,H^{2d+2-j}(\overline{\mathcal{X}}_{et},\mathbb{Z}/n\mathbb{Z}))^{DD}$$
$$\stackrel{\sim}{\rightarrow}
(\underleftarrow{\mbox{lim}}\,H^{j-1}(\overline{\mathcal{X}}_{et},\mathbb{Z}(d))_n)^D\stackrel{\sim}{\rightarrow} \mbox{Hom}(
H^{j-1}(\overline{\mathcal{X}}_{et},\mathbb{Z}(d))_{\geq0},\mathbb{Q}/\mathbb{Z})$$
for any $j\leq 2d+1$ (note that $2d+2-(j-1)\geq2 \Leftrightarrow j\leq 2d+1$). The last isomorphism above follows from the fact that $H^{j}(\overline{\mathcal{X}}_{et},\mathbb{Z}(d))$ is finitely generated for $0\leq j\leq 2d$ and uniquely divisible for $j<0$. Hence for any $i\geq2$ the natural map
$$H^i(\overline{\mathcal{X}}_{et},\mathbb{Z})\rightarrow \mbox{Hom}(H^{2d+2-i}(\overline{\mathcal{X}}_{et},\mathbb{Z}(d))_{\geq0},\mathbb{Q}/\mathbb{Z})$$
is an isomorphism.
\end{proof}

Recall that an abelian group $A$ is of \emph{cofinite type} if $A\simeq\mathrm{Hom}_{\mathbb{Z}}(B,\mathbb{Q}/\mathbb{Z})$ where $B$ is a finitely generated abelian group. If $A$ is of cofinite type then there exists an isomorphism $A\simeq(\mathbb{Q}/\mathbb{Z})^r\oplus T$, where $r\in\mathbb{N}$ and $T$ is a finite abelian group. If $\mathcal{X}$ satisfies $\textbf{L}(\mathcal{X}_{et},d)_{\geq0}$ then $H^i(\overline{\mathcal{X}}_{et},\mathbb{Z})$ is of cofinite type for $i\geq 1$ by Lemma \ref{prop-duality-Z-coefs}.

\begin{thm}\label{thm-alpha} Assume that $\mathcal{X}$ satisfies $\textbf{\emph{L}}(\mathcal{X}_{et},d)_{\geq0}$. There exists a unique morphism in $\mathcal{D}$
$$\alpha_{\mathcal{X}}:\mbox{\emph{RHom}}(\mbox{\emph{R}}\Gamma(\mathcal{X},\mathbb{Q}(d))_{\geq0},\mathbb{Q}[-2d-2])\rightarrow \mbox{\emph{R}}\Gamma(\overline{\mathcal{X}}_{et},\mathbb{Z})$$
such that $H^i(\alpha_{\mathcal{X}})$ is the following composite map
$$\mbox{\emph{Hom}}(H^{2d+2-i}(\mathcal{X},\mathbb{Q}(d))_{\geq0},\mathbb{Q})\stackrel{\sim}{\rightarrow}
\mbox{\emph{Hom}}(H^{2d+2-i}(\overline{\mathcal{X}}_{et},\mathbb{Z}(d))_{\geq0},\mathbb{Q})$$
$$\rightarrow
\mbox{\emph{Hom}}(H^{2d+2-i}(\overline{\mathcal{X}}_{et},\mathbb{Z}(d))_{\geq0},\mathbb{Q}/\mathbb{Z})
\stackrel{\sim}{\leftarrow}H^i(\overline{\mathcal{X}}_{et},\mathbb{Z})$$
for any $i\geq2$. 
\end{thm}

\begin{proof}

In order to ease the notation, we set $D_{\mathcal{X}}:= \mbox{RHom}(\mbox{R}\Gamma(\mathcal{X},\mathbb{Q}(d))_{\geq0},\mathbb{Q}[-2d-2])$. 
We consider the spectral sequence (see \cite{Verdier96} III, Section 4.6.10) 
\begin{equation}\label{VerdierSS}
E_2^{p,q}=\prod_{i\in\mathbb{Z}} \textrm{Ext}^p(H^{i-q}(D_{\mathcal{X}}),H^{i}(\overline{\mathcal{X}}_{et},\mathbb{Z}))\Rightarrow
H^{p+q}(\textrm{RHom}(D_{\mathcal{X}},\mbox{R}\Gamma(\overline{\mathcal{X}}_{et},\mathbb{Z}))).
\end{equation}
An edge morphism from (\ref{VerdierSS}) gives a  morphism
\begin{equation}\label{Bmap}
H^{0}(\textrm{RHom}(D_{\mathcal{X}},\mbox{R}\Gamma(\overline{\mathcal{X}}_{et},\mathbb{Z})))\rightarrow \prod_{i\in\mathbb{Z}} \textrm{Ext}^0(H^{i}(D_{\mathcal{X}}),H^{i}(\overline{\mathcal{X}}_{et},\mathbb{Z}))
\end{equation}
such that the composite map
$$\mbox{Hom}_{\mathcal{D}}(D_{\mathcal{X}},\mbox{R}\Gamma(\overline{\mathcal{X}}_{et},\mathbb{Z})))\stackrel{\sim}{\rightarrow} H^{0}(\textrm{RHom}(D_{\mathcal{X}},\mbox{R}\Gamma(\overline{\mathcal{X}}_{et},\mathbb{Z})))\rightarrow \prod_{i\in\mathbb{Z}} \textrm{Ext}^0(H^{i}(D_{\mathcal{X}}),H^{i}(\overline{\mathcal{X}}_{et},\mathbb{Z}))$$
is the obvious one. 

By Lemma \ref{prop-duality-Z-coefs}, $H^i(\overline{\mathcal{X}}_{et},\mathbb{Z})$ is of cofinite type for $i\neq 0$. It follows that, if both $i\neq0$ and $p\neq 0$, then $\textrm{Ext}^p(H^{i-q}(D_{\mathcal{X}}),H^{i}(\overline{\mathcal{X}}_{et},\mathbb{Z}))=0$, since $H^{i-q}(D_{\mathcal{X}})$ is uniquely divisible. Indeed, $\textrm{Ext}^p(H^{i-q}(D_{\mathcal{X}}),(\mathbb{Q}/\mathbb{Z})^r)=0$ for $p\geq1$, since $(\mathbb{Q}/\mathbb{Z})^r$ is divisible. Moreover, if $T$ is a finite abelian group of order $n$ then $\textrm{Ext}^p(H^{i-q}(D_{\mathcal{X}}),T)=0$ since $\textrm{Ext}^p(H^{i-q}(D_{\mathcal{X}}),T)$ is both uniquely divisible (since $H^{i-q}(D_{\mathcal{X}})$ is) and killed by $n$ (since $T$ is). For $i=0$, one has $\textrm{Ext}^p(H^{0-q}(D_{\mathcal{X}}),H^{0}(\overline{\mathcal{X}}_{et},\mathbb{Z}))=0$ as long as $p\geq2$, because $H^{0}(\overline{\mathcal{X}}_{et},\mathbb{Z})=\mathbb{Z}$ has an injective resolution of length one. In particular,  $E_2^{p,q}=0$ for $p\neq 0,1$, hence the spectral sequence (\ref{VerdierSS}) degenerates at $E_2$. Moreover, one has
$$E_2^{1,-1}= \prod_{i\in\mathbb{Z}} \textrm{Ext}^1(H^{i+1}(D_{\mathcal{X}}),H^{i}(\overline{\mathcal{X}}_{et},\mathbb{Z}))=\textrm{Ext}^1(H^{1}(D_{\mathcal{X}}),H^{0}(\overline{\mathcal{X}}_{et},\mathbb{Z}))=0$$
since $H^{1}(D_{\mathcal{X}})=0$. It follows that the edge morphism (\ref{Bmap}) is an isomorphism.

For $i\leq 1$, any map $H^i(\alpha_{\mathcal{X}}):H^{i}(D_{\mathcal{X}})\rightarrow H^i(\overline{\mathcal{X}}_{et},\mathbb{Z})$ must be trivial since $H^{i}(D_{\mathcal{X}})=0$. 
For any $i\geq 2$, we consider the morphism
$$\alpha^i_{\mathcal{X}}:H^{i}(D_{\mathcal{X}})=\textrm{Hom}(H^{2d+2-i}(\mathcal{X},\mathbb{Q}(d))_{\geq0},\mathbb{Q})\stackrel{\sim}{\rightarrow}
\textrm{Hom}(H^{2d+2-i}(\overline{\mathcal{X}}_{et},\mathbb{Z}(d))_{\geq0},\mathbb{Q})$$
$$\rightarrow
\textrm{Hom}(H^{2d+2-i}(\overline{\mathcal{X}}_{et},\mathbb{Z}(d))_{\geq0},\mathbb{Q}/\mathbb{Z})
\stackrel{\sim}{\leftarrow}H^i(\overline{\mathcal{X}}_{et},\mathbb{Z})$$
where the last isomorphism is given by Lemma \ref{prop-duality-Z-coefs}. But (\ref{Bmap}) is an isomorphism, hence there exists a unique map in $\mathcal{D}$
$$\alpha_{\mathcal{X}}:\mbox{RHom}(\mbox{R}\Gamma(\mathcal{X},\mathbb{Q}(d))_{\geq 0},\mathbb{Q}[-2d-2])\rightarrow \mbox{R}\Gamma(\overline{\mathcal{X}}_{et},\mathbb{Z})$$
such that $H^i(\alpha_{\mathcal{X}})=\alpha^i_{\mathcal{X}}$.
\end{proof}

\subsubsection{The general case}\label{sect-X(R)} In this subsection, we allow $\mathcal{X}(\mathbb{R})\neq\emptyset$. So $\mathcal{X}$ denotes a proper, regular and connected arithmetic scheme of dimension $d$. It is possible to define a dualizing complex $\mathbb{Z}(d)^{\overline{\mathcal{X}}}$ on  $\overline{\mathcal{X}}_{et}$, to prove Artin-Verdier duality in this setting and to define $\alpha_{\mathcal{X}}$ in the very same way as in Section \ref{sectNotX(R)}. For the sake of conciseness, we shall use a somewhat trickier construction, which is based on Milne's cohomology with compact support \cite{Milne-duality}.

We denote by 
\begin{equation}\label{Milnedef}
\mbox{R}\hat{\Gamma}_{c}(\mathcal{X}_{et},\mathcal{F}):=\mbox{R}\hat{\Gamma}_{c}(\mathrm{Spec}(\mathbb{Z})_{et},\mbox{R}f_*\mathcal{F})
\end{equation}
Milne's cohomology with compact support, where $\mbox{R}\hat{\Gamma}_{c}(\mathrm{Spec}(\mathbb{Z})_{et},-)$ is defined as in (\cite{Milne-duality} Section II.2) and $f:\X\rightarrow\mathrm{Spec}(\mathbb{Z})$ is the structure map.
Milne's definition yields a canonical map $\mbox{R}\hat{\Gamma}_{c}(\mathrm{Spec}(\mathbb{Z})_{et},-)\rightarrow \mbox{R}\Gamma(\mathrm{Spec}(\mathbb{Z})_{et},-)$ inducing
\begin{equation}\label{milnemap}
\mbox{R}\hat{\Gamma}_c(\mathcal{X}_{et},\mathbb{Z})\rightarrow \mbox{R}\Gamma(\mathcal{X}_{et},\mathbb{Z}).
\end{equation} 
We need the following lemma.

\begin{lem}\label{lem-hat-bar}
There is a morphism $\mbox{\emph{R}}\hat{\Gamma}_c(\mathcal{X}_{et},\mathbb{Z})\stackrel{c}{\rightarrow} \mbox{\emph{R}}\Gamma(\overline{\mathcal{X}}_{et},\mathbb{Z})$ satisfying the following properties. The composite map
$$\mbox{\emph{R}}\hat{\Gamma}_c(\mathcal{X}_{et},\mathbb{Z})\stackrel{c}{\rightarrow} \mbox{\emph{R}}\Gamma(\overline{\mathcal{X}}_{et},\mathbb{Z})\stackrel{\varphi^*}{\rightarrow} \mbox{\emph{R}}\Gamma(\mathcal{X}_{et},\mathbb{Z})$$
coincides with (\ref{milnemap}), $H^i(c)$ is an isomorphism for $i$ large and $H^i(c)$ has finite $2$-torsion kernel and cokernel for any $i\in\mathbb{Z}$.
\end{lem}
\begin{proof}
Recall from (\cite{Flach-moi} Section 4) that there is an isomorphism of functors $u_{\infty}^*\varphi_*\simeq \pi_*\alpha^*$ where $\alpha:Sh(G_{\mathbb{R}},\mathcal{X}(\mathbb{C}))\rightarrow \mathcal{X}_{et}$ and $\pi:Sh(G_{\mathbb{R}},\mathcal{X}(\mathbb{C}))\rightarrow Sh(\mathcal{X}_{\infty})$ are the canonical maps while $u_{\infty}:Sh(\mathcal{X}_{\infty})\rightarrow \overline{\mathcal{X}}_{et}$ and $\varphi:\mathcal{X}_{et}\rightarrow \overline{\mathcal{X}}_{et}$ are complementary closed and open embeddings (see \cite{Flach-moi} Section 4). Here $Sh(G_{\mathbb{R}},\mathcal{X}(\mathbb{C}))$ (respectively $Sh(\mathcal{X}_{\infty})$) denotes the topos  of $G_{\mathbb{R}}$-equivariant sheaves of sets on $\mathcal{X}(\mathbb{C})$ (respectively the topos of sheaves on $\mathcal{X}_{\infty}$). It is easy to see that $\alpha^*$ sends injective abelian sheaves to $\pi_*$-acyclics ones. We obtain
$$u_{\infty}^*R\varphi_*\simeq R(u_{\infty}^*\varphi_*)\simeq R(\pi_*\alpha^*)\simeq (R\pi_*)\alpha^*.$$
Consider the map
$$R\varphi_*\mathbb{Z}\rightarrow u_{\infty,*}u_{\infty}^*R\varphi_*\mathbb{Z}\simeq u_{\infty,*}R\pi_*\mathbb{Z}\rightarrow 
u_{\infty,*}\tau^{>0}R\pi_*\mathbb{Z}.$$
Since $\tau^{\leq0}R\pi_*\mathbb{Z}$ is the constant sheaf $\mathbb{Z}=u_{\infty}^*\mathbb{Z}$ put in degree $0$, we obtain an exact triangle
\begin{equation}\label{exTZbar}
\mathbb{Z}^{\overline{\mathcal{X}}}\rightarrow R\varphi_*\mathbb{Z}\rightarrow u_{\infty,*}\tau^{>0}R\pi_*\mathbb{Z}.
\end{equation}
where $\mathbb{Z}^{\overline{\mathcal{X}}}$ denotes the constant sheaf $\mathbb{Z}$ on $\overline{\mathcal{X}}_{et}$. Moreover, we have $\mbox{R}\Gamma(\overline{\mathcal{X}}_{et},u_{\infty,*}\tau^{>0}R\pi_*\mathbb{Z})\simeq\mbox{R}\Gamma(\mathcal{X}_{\infty},\tau^{>0}R\pi_*\mathbb{Z})\simeq\mbox{R}\Gamma(\mathcal{X}(\mathbb{R}),\tau^{>0}R\pi_*\mathbb{Z})$ because $u_{\infty,*}$ is exact and
$\tau^{>0}R\pi_*\mathbb{Z}$ is concentrated on $\mathcal{X}(\mathbb{R})\subset \mathcal{X}_{\infty}$.
Therefore, applying $\mbox{R}\Gamma(\overline{\mathcal{X}}_{et},-)$ to (\ref{exTZbar}), we get an exact triangle
\begin{equation}\label{exTZbarGamma}
\mbox{R}\Gamma(\overline{\mathcal{X}}_{et},\mathbb{Z})\rightarrow \mbox{R}\Gamma(\mathcal{X}_{et},\mathbb{Z})\rightarrow \mbox{R}\Gamma(\mathcal{X}(\mathbb{R}),\tau^{>0}R\pi_*\mathbb{Z}).
\end{equation}
On the other hand there is an exact triangle
\begin{equation}\label{exTZbarHatGamma}
\mbox{R}\hat{\Gamma}_c(\mathcal{X}_{et},\mathbb{Z})\rightarrow \mbox{R}\Gamma(\mathcal{X}_{et},\mathbb{Z})\rightarrow \mbox{R}\hat{\Gamma}(G_{\mathbb{R}},\mathcal{X}(\mathbb{R}),\mathbb{Z})
\end{equation}
where $\mbox{R}\hat{\Gamma}(G_{\mathbb{R}},\mathcal{X}(\mathbb{R}),-):=\mbox{R}\hat{\Gamma}(G_{\mathbb{R}},(\mbox{R}\Gamma(\mathcal{X}(\mathbb{R}),-))$ denotes equivariant Tate cohomology (which was introduced in \cite{Swan60}). Notice that the canonical map $\mbox{R}\hat{\Gamma}(G_{\mathbb{R}},\mathcal{X}(\mathbb{C}),\mathbb{Z})\rightarrow \mbox{R}\hat{\Gamma}(G_{\mathbb{R}},\mathcal{X}(\mathbb{R}),\mathbb{Z})$ is an isomorphism \cite{Swan60}.

We shall define below a canonical map $c_{\infty}:\mbox{R}\hat{\Gamma}(G_{\mathbb{R}},\mathcal{X}(\mathbb{R}),\mathbb{Z})\rightarrow \mbox{R}\Gamma(\mathcal{X}(\mathbb{R}),\tau^{>0}R\pi_*\mathbb{Z})$ such that $H^i(c_{\infty})$ is an isomorphism for $i$ large and $H^i(c_{\infty})$ has finite $2$-torsion kernel and cokernel for any $i\in\mathbb{Z}$. The map $c_{\infty}$ is compatible (in the obvious sense) with the identity map of $\mbox{R}\Gamma(\mathcal{X}_{et},\mathbb{Z})$, hence there exists a morphism $c:\mbox{R}\hat{\Gamma}_c(\mathcal{X}_{et},\mathbb{Z})\rightarrow \mbox{R}\Gamma(\overline{\mathcal{X}}_{et},\mathbb{Z})$ such that $(c,Id,c_{\infty})$ is a morphism of exact triangles from (\ref{exTZbarHatGamma}) to (\ref{exTZbarGamma}). Hence the result will follow.

It remains to define $c_{\infty}$. We denote by $\mbox{R}\Gamma(G_{\mathbb{R}},\mathcal{X}(\mathbb{R}),-):=\mbox{R}\Gamma(G_{\mathbb{R}},(\mbox{R}\Gamma(\mathcal{X}(\mathbb{R}),-))$ usual equivariant cohomology. We have isomorphisms
\begin{eqnarray}
\label{id0}\mathrm{R}\hat{\Gamma}(G_{\mathbb{R}},\mathcal{X}(\mathbb{R}),\mathbb{Z})_{>d}
&\simeq&\mathrm{R}\Gamma(G_{\mathbb{R}},\mathcal{X}(\mathbb{R}),\mathbb{Z})_{>d}\\
\label{id2}&\simeq&\mathrm{R}\Gamma(\mathcal{X}(\mathbb{R}),R\pi_*\mathbb{Z})_{>d}\\
\label{id3}&\simeq&\mathrm{R}\Gamma(\mathcal{X}(\mathbb{R}),(R\pi_*\mathbb{Z})_{>0})_{>d}\\
\label{id4}&\simeq&\mathrm{R}\Gamma(\mathcal{X}(\mathbb{R}),\mathrm{R}\Gamma(G_{\mathbb{R}},\mathbb{Z})_{>0})_{>d}\\
\label{id44}&\simeq&\mathrm{R}\Gamma(\mathcal{X}(\mathbb{R}),\mathrm{R}\hat{\Gamma}(G_{\mathbb{R}},\mathbb{Z})_{>0})_{>d}\\
\label{id5}&\simeq&\mathrm{R}\Gamma(\mathcal{X}(\mathbb{R}),\mathrm{R}\hat{\Gamma}(G_{\mathbb{R}},\mathbb{Z}))_{>d}
\end{eqnarray}
where we consider $\mathrm{R}\hat{\Gamma}(G_{\mathbb{R}},\mathbb{Z})$ and $\mathrm{R}\Gamma(G_{\mathbb{R}},\mathbb{Z})$ as complexes of constant sheaves on $\mathcal{X}(\mathbb{R})$. The canonical map $\mbox{R}\Gamma(G_{\mathbb{R}},\mathcal{X}(\mathbb{R}),\mathbb{Z})\rightarrow \mbox{R}\hat{\Gamma}(G_{\mathbb{R}},\mathcal{X}(\mathbb{R}),\mathbb{Z})$ and the corresponding morphism of hypercohomology spectral sequences give the isomorphism (\ref{id0}), since $\mathcal{X}(\mathbb{R})$ is of topological dimension $d$. Similarly (\ref{id3}) and (\ref{id5}) can be deduced from the corresponding morphisms of spectral sequences. Finally, (\ref{id4}) is given by  $(R\pi_*\mathbb{Z})_{\mid\mathcal{X}(\mathbb{R})}\simeq \mathrm{R}\Gamma(G_{\mathbb{R}},\mathbb{Z})$. Since $G_{\mathbb{R}}$ is cyclic, we have an isomorphism  
$\mathrm{R}\hat{\Gamma}(G_{\mathbb{R}},\mathbb{Z})
\stackrel{\sim}{\rightarrow}\mathrm{R}\hat{\Gamma}(G_{\mathbb{R}},\mathbb{Z})[2]$
in $\mathcal{D}$.
Applying $\mathrm{R}\Gamma(\mathcal{X}(\mathbb{R}),-)$ we obtain
$$\mathrm{R}\Gamma(\mathcal{X}(\mathbb{R}),\mathrm{R}\hat{\Gamma}(G_{\mathbb{R}},\mathbb{Z}))\stackrel{\sim}{\longrightarrow}
\mathrm{R}\Gamma(\mathcal{X}(\mathbb{R}),\mathrm{R}\hat{\Gamma}(G_{\mathbb{R}},\mathbb{Z}))[2].$$
The same is true for equivariant Tate cohomology: we have an isomorphism
$$\mathrm{R}\hat{\Gamma}(G_{\mathbb{R}},\mathcal{X}(\mathbb{R}),\mathbb{Z})\stackrel{\sim}{\longrightarrow}
\mathrm{R}\hat{\Gamma}(G_{\mathbb{R}},\mathcal{X}(\mathbb{R}),\mathbb{Z})[2].$$
Hence for any integer $k\geq1$, (\ref{id5}) induces an isomorphism
$$\mathrm{R}\hat{\Gamma}(G_{\mathbb{R}},\mathcal{X}(\mathbb{R}),\mathbb{Z})_{>d-2k}\simeq
\mathrm{R}\Gamma(\mathcal{X}(\mathbb{R}),\mathrm{R}\hat{\Gamma}(G_{\mathbb{R}},\mathbb{Z}))_{>d-2k}.$$ 
Taking $k$ large enough, we obtain a canonical isomorphism
\begin{equation}\label{jiii}
\mathrm{R}\hat{\Gamma}(G_{\mathbb{R}},\mathcal{X}(\mathbb{R}),\mathbb{Z})_{>0}\simeq
\mathrm{R}\Gamma(\mathcal{X}(\mathbb{R}),\mathrm{R}\hat{\Gamma}(G_{\mathbb{R}},\mathbb{Z}))_{>0}.
\end{equation} 
The map
$\mathrm{R}\hat{\Gamma}(G_{\mathbb{R}},\mathbb{Z})\rightarrow \mathrm{R}\hat{\Gamma}(G_{\mathbb{R}},\mathbb{Z})_{>0}$ induces
\begin{equation}\label{jeje}
\mathrm{R}\Gamma(\mathcal{X}(\mathbb{R}),\mathrm{R}\hat{\Gamma}(G_{\mathbb{R}},\mathbb{Z}))\longrightarrow \mathrm{R}\Gamma(\mathcal{X}(\mathbb{R}),\mathrm{R}\hat{\Gamma}(G_{\mathbb{R}},\mathbb{Z})_{>0}).
\end{equation}
Since $\mathrm{R}\Gamma(\mathcal{X}(\mathbb{R}),\mathrm{R}\hat{\Gamma}(G_{\mathbb{R}},\mathbb{Z})_{>0})$ is acyclic in degrees $\leq0$, (\ref{jeje}) is induced by a unique map
\begin{equation}\label{jejekr}
\mathrm{R}\Gamma(\mathcal{X}(\mathbb{R}),\mathrm{R}\hat{\Gamma}(G_{\mathbb{R}},\mathbb{Z}))_{>0}\longrightarrow \mathrm{R}\Gamma(\mathcal{X}(\mathbb{R}),\mathrm{R}\hat{\Gamma}(G_{\mathbb{R}},\mathbb{Z})_{>0})
\end{equation}
and $c_{\infty}$ is defined as the composition
$$\mathrm{R}\hat{\Gamma}(G_{\mathbb{R}},\mathcal{X}(\mathbb{R}),\mathbb{Z})\longrightarrow \mathrm{R}\hat{\Gamma}(G_{\mathbb{R}},\mathcal{X}(\mathbb{R}),\mathbb{Z})_{>0} \stackrel{(\ref{jiii})}{\longrightarrow}
\mathrm{R}\Gamma(\mathcal{X}(\mathbb{R}),\mathrm{R}\hat{\Gamma}(G_{\mathbb{R}},\mathbb{Z}))_{>0}$$
$$\stackrel{(\ref{jejekr})}{\longrightarrow}
\mathrm{R}\Gamma(\mathcal{X}(\mathbb{R}),\mathrm{R}\hat{\Gamma}(G_{\mathbb{R}},\mathbb{Z})_{>0})
\simeq \mathrm{R}\Gamma(\mathcal{X}(\mathbb{R}),\tau^{>0}R\pi_*\mathbb{Z}).$$

\end{proof}
If $\X(\mathbb{R})=\emptyset$, then the maps $c$ and $\varphi^*$ of Lemma \ref{lem-hat-bar} are isomorphisms in $\mathcal{D}$ (this follows from (\ref{exTZbarHatGamma}) and from the fact that $\varphi_*$ is exact in this case). Therefore, the following result generalizes Theorem \ref{thm-alpha}.
\begin{prop}\label{propXRneqempty}
Assume that $\mathcal{X}$ satisfies $\textbf{\emph{L}}(\mathcal{X}_{et},d)_{\geq0}$. Then there exists a unique morphism in $\mathcal{D}$
$$\alpha_{\mathcal{X}}:\mbox{\emph{RHom}}(\mbox{\emph{R}}\Gamma(\mathcal{X},\mathbb{Q}(d))_{\geq0},\mathbb{Q}[-2d-2])\rightarrow \mbox{\emph{R}}\Gamma(\overline{\mathcal{X}}_{et},\mathbb{Z})$$
such that $H^i(\alpha_{\mathcal{X}})$ is the following composite map
$$\mbox{\emph{Hom}}(H^{2d+2-i}(\mathcal{X},\mathbb{Q}(d))_{\geq0},\mathbb{Q})\stackrel{\sim}{\longrightarrow}
\mbox{\emph{Hom}}(H^{2d+2-i}(\mathcal{X}_{et},\mathbb{Z}(d))_{\geq0},\mathbb{Q})$$
$$\longrightarrow
\mbox{\emph{Hom}}(H^{2d+2-i}(\mathcal{X}_{et},\mathbb{Z}(d))_{\geq0},\mathbb{Q}/\mathbb{Z})
\stackrel{\sim}{\longleftarrow}\hat{H}_c^i(\mathcal{X}_{et},\mathbb{Z})\stackrel{H^i(c)}{\longrightarrow} H^i(\overline{\mathcal{X}}_{et},\mathbb{Z})$$
for any $i\geq2$. 
\end{prop}

\begin{proof}
We consider the composite map
$$\mbox{RHom}_{\mathcal{X}}(\mathbb{Z},\mathbb{Z}(d))\rightarrow\mbox{RHom}_{\mathrm{Spec}(\mathbb{Z})}(\mbox{R}f_*\mathbb{Z},\mbox{R}f_*\mathbb{Z}(d))
\stackrel{p}{\rightarrow}\mbox{RHom}_{\mathrm{Spec}(\mathbb{Z})}(\mbox{R}f_*\mathbb{Z},\mathbb{Z}(1)[-2d+2])$$
$$ \stackrel{q}{\rightarrow} \mbox{RHom}_{\mathcal{D}}(\mbox{R}\hat{\Gamma}_c(\mathrm{Spec}(\mathbb{Z})_{et},\mbox{R}f_*\mathbb{Z}),\mathbb{Q}/\mathbb{Z}[-2d-2]) =\mbox{RHom}_{\mathcal{D}}(\mbox{R}\hat{\Gamma}_c(\mathcal{X}_{et},\mathbb{Z}),\mathbb{Q}/\mathbb{Z}[-2d-2])$$
where $p$ is induced by the push-forward map of (\cite{Geisser-Duality} Corollary 7.2 (b)) and $q$ is given by 1-dimensional Artin-Verdier duality (see \cite{Milne-duality} II.2.6) since $\mathbb{Z}(1)\simeq\mathbb{G}_m[-1]$ (see \cite{Geisser-Duality} Lemma 7.4). This composite map induces
$$\mbox{R}\Gamma(\mathcal{X}_{et},\mathbb{Z}(d))_{\geq0}\rightarrow \mbox{RHom}(\mbox{R}\hat{\Gamma}_c(\mathcal{X}_{et},\mathbb{Z}),\mathbb{Q}/\mathbb{Z}[-2d-2])
$$
since $\mbox{R}\hat{\Gamma}_c(\mathcal{X}_{et},\mathbb{Z})$ is acyclic in degrees $>2d+2$. Indeed, $\hat{H}^i_c(\mathrm{Spec}(\mathbb{Z})_{et},F)=0$ for any $i>3$ and any torsion abelian sheaf $F$ (see \cite{Milne-duality} Theorem II.3.1(b)); hence by proper base change the hypercohomology spectral sequence associated to (\ref{Milnedef}) gives
$\hat{H}^i_c(\mathcal{X}_{et},\mathbb{Q}/\mathbb{Z})=0$ for $i>2d+1$. By adjunction, we obtain the product map
\begin{equation}\label{map-duality-cplexes2}
\mbox{R}\Gamma(\mathcal{X}_{et},\mathbb{Z}(d))_{\geq0}\otimes_{\mathbb{Z}}^L\mbox{R}\hat{\Gamma}_c(\mathcal{X}_{et},\mathbb{Z})\longrightarrow\mathbb{Q}/\mathbb{Z}[-2d-2].
\end{equation}
For any positive integer $n$ and any $i\in\mathbb{Z}$, the canonical map
\begin{equation}\label{perfpair}
\mbox{Ext}^{2d+2-i}_{\mathcal{X}}(\mathbb{Z}/n\mathbb{Z},\mathbb{Z}(d))\times \hat{H}_c^i(\mathcal{X}_{et},\mathbb{Z}/n\mathbb{Z})\rightarrow\mathbb{Q}/\mathbb{Z}
\end{equation}
is a perfect pairing of finite groups (see \cite{Geisser-Duality} Theorem 7.8). Note that (\ref{perfpair}) is compatible with (\ref{map-duality-cplexes2}). The arguments of the proof of Lemma \ref{prop-duality-Z-coefs} show that the morphism
\begin{equation}\label{cofinitetypehat}
\hat{H}_c^i(\mathcal{X}_{et},\mathbb{Z})\stackrel{\sim}{\rightarrow}\mbox{Hom}(H^{2d+2-i}(\mathcal{X}_{et},\mathbb{Z}(d))_{\geq0},\mathbb{Q}/\mathbb{Z}),
\end{equation}
induced by (\ref{map-duality-cplexes2}), is an isomorphism of abelian groups for $i\geq 2$. The argument of the proof of Theorem \ref{thm-alpha} provides us with a
unique map
$$\tilde{\alpha}_{\mathcal{X}}:\mbox{RHom}(\mbox{R}\Gamma(\mathcal{X},\mathbb{Q}(d))_{\geq0},\mathbb{Q}[-2d-2])\rightarrow 
\mbox{R}\hat{\Gamma}_c(\mathcal{X}_{et},\mathbb{Z})$$
such that $H^i(\tilde{\alpha}_{\mathcal{X}})$ is the following composite map
$$\mbox{Hom}(H^{2d+2-i}(\mathcal{X},\mathbb{Q}(d))_{\geq0},\mathbb{Q})\stackrel{\sim}{\rightarrow}
\mbox{Hom}(H^{2d+2-i}(\mathcal{X}_{et},\mathbb{Z}(d))_{\geq0},\mathbb{Q})$$
$$\rightarrow
\mbox{Hom}(H^{2d+2-i}(\mathcal{X}_{et},\mathbb{Z}(d))_{\geq0},\mathbb{Q}/\mathbb{Z})
\stackrel{\sim}{\leftarrow}\hat{H}_c^i(\mathcal{X}_{et},\mathbb{Z})$$
for any $i\geq2$. Then we define $\alpha_{\mathcal{X}}$ as the composite map
$$\mbox{RHom}(\mbox{R}\Gamma(\mathcal{X},\mathbb{Q}(d))_{\geq0},\mathbb{Q}[-\delta])\stackrel{\tilde{\alpha}_{\mathcal{X}}}{\longrightarrow} \mbox{R}\hat{\Gamma}_c(\mathcal{X}_{et},\mathbb{Z})\stackrel{c}{\longrightarrow} \mbox{R}\Gamma(\overline{\mathcal{X}}_{et},\mathbb{Z})$$
where $c$ is the map of Lemma \ref{lem-hat-bar}. It remains to show that $\alpha_{\mathcal{X}}$ does not depend on the choice of $c$ (note that $c$ is not uniquely defined). The map $\varphi^*\circ c$ is well defined (since it coincides with (\ref{milnemap}) by Lemma \ref{lem-hat-bar}), hence so is $\varphi^*\circ \alpha_{\mathcal{X}}: D_{\mathcal{X}}\rightarrow \mbox{R}\Gamma(\mathcal{X}_{et},\mathbb{Z}) $, where $D_{\mathcal{X}}:=\mbox{RHom}(\mbox{R}\Gamma(\mathcal{X},\mathbb{Q}(d))_{\geq0},\mathbb{Q}[-2d-2])$. By (\ref{exTZbarGamma}) and using the fact that $D_{\mathcal{X}}$ is a complex of $\mathbb{Q}$-vector spaces while $\mbox{R}\Gamma(\mathcal{X}(\mathbb{R}),\tau^{>0}R\pi_*\mathbb{Z})$ is a complex of $\mathbb{Z}/2\mathbb{Z}$-vector spaces, we see that composition with $\varphi^*$ induces an isomorphism $$\varphi^*\circ :\mathrm{Hom}_{\mathcal{D}}(D_{\mathcal{X}},\mbox{R}\Gamma(\overline{\mathcal{X}}_{et},\mathbb{Z}))\stackrel{\sim}{\rightarrow}\mathrm{Hom}_{\mathcal{D}}(D_{\mathcal{X}},\mbox{R}\Gamma(\mathcal{X}_{et},\mathbb{Z}))$$
so that $\alpha_{\mathcal{X}}$ does not depend on the choice of $c$. The result follows.
\end{proof}

\subsection{The complex $\mathrm{R}\Gamma_W(\overline{\mathcal{X}},\mathbb{Z})$}\label{sect-RGW}
Throughout this section $\mathcal{X}$ denotes a proper regular connected arithmetic scheme of dimension $d$ satisfying $\textbf{L}(\mathcal{X}_{et},d)_{\geq0}$.

\begin{defn}\label{def-cohomology}
There exists an exact triangle
$$\mbox{\emph{RHom}}(\mbox{\emph{R}}\Gamma(\mathcal{X},\mathbb{Q}(d))_{\geq0},\mathbb{Q}[-2d-2])
\stackrel{\alpha_{\mathcal{X}}}{\longrightarrow} \mbox{\emph{R}}\Gamma(\overline{\mathcal{X}}_{et},\mathbb{Z})\longrightarrow \mbox{\emph{R}}\Gamma_W(\overline{\mathcal{X}},\mathbb{Z})$$
well defined up to a unique isomorphism in $\mathcal{D}$. We define
$$H_W^i(\overline{\mathcal{X}},\mathbb{Z}):=H^i(\mbox{\emph{R}}\Gamma_W(\overline{\mathcal{X}},\mathbb{Z})).$$
\end{defn}
The existence of such an exact triangle follows from axiom TR1 of triangulated categories. Its uniqueness is given by Theorem \ref{cor-functoriality} and can be stated as follows. If we have two objects $\mbox{R}\Gamma_W(\overline{\mathcal{X}},\mathbb{Z})$ and $\mbox{R}\Gamma_W(\overline{\mathcal{X}},\mathbb{Z})'$ given with exact triangles as in Definition \ref{def-cohomology}, then there exists a \emph{unique} map
$f:\mbox{R}\Gamma_W(\overline{\mathcal{X}},\mathbb{Z})\rightarrow\mbox{R}\Gamma_W(\overline{\mathcal{X}},\mathbb{Z})'$ in $\mathcal{D}$ which yields a morphism of exact triangles
\[ \xymatrix{
\mbox{RHom}(\mbox{R}\Gamma(\mathcal{X},\mathbb{Q}(d))_{\geq0},\mathbb{Q}[-2d-2])\ar[d]^{Id}\ar[r]^{
\,\,\,\,\,\,\,\,\,\,\,\,\,\,\,\,\,\,\,\,\,\,\,\,\,\,\,\,\,\,\,\,\,\,\,{\alpha}_{\mathcal{X}}}
&\mbox{R}\Gamma(\overline{\mathcal{X}}_{et},\mathbb{Z})\ar[r]\ar[d]^{Id}
&\mbox{R}\Gamma_W(\overline{\mathcal{X}},\mathbb{Z})\ar[d]^f\\
\mbox{RHom}(\mbox{R}\Gamma(\mathcal{X},\mathbb{Q}(d))_{\geq0},\mathbb{Q}[-2d-2])\ar[r]^{
\,\,\,\,\,\,\,\,\,\,\,\,\,\,\,\,\,\,\,\,\,\,\,\,\,\,\,\,\,\,\,\,\,\,\,{\alpha}_{\mathcal{X}}}
&\mbox{R}\Gamma(\overline{\mathcal{X}}_{et},\mathbb{Z})\ar[r]
&\mbox{R}\Gamma_W(\overline{\mathcal{X}},\mathbb{Z})'
}
\]

\begin{prop}\label{finitelygenerated-cohomology} The following assertions are true.
\begin{enumerate}
\item The group $H_W^i(\overline{\mathcal{X}},\mathbb{Z})$ is finitely generated for any $i\in\mathbb{Z}$.  One has
$H_W^i(\overline{\mathcal{X}},\mathbb{Z})=0$ for $i<0$, $H_W^0(\overline{\mathcal{X}},\mathbb{Z})=\mathbb{Z}$ and $H_W^i(\overline{\mathcal{X}},\mathbb{Z})=0$ for $i$ large.
\item If $\mathcal{X}(\mathbb{R})=\emptyset$ then there is an exact sequence
$$0\rightarrow H^i(\overline{\mathcal{X}}_{et},\mathbb{Z})_{codiv} \rightarrow H_W^i(\overline{\mathcal{X}},\mathbb{Z})\rightarrow \mbox{\emph{Hom}}_{\mathbb{Z}}(H^{2d+2-(i+1)}(\overline{\mathcal{X}}_{et},\mathbb{Z}(d))_{\geq0},\mathbb{Z})\rightarrow 0$$
for any $i\in\mathbb{Z}$.
\item For any $\mathcal{X}$ and any $i\in\mathbb{Z}$, there is an exact sequence
$$0\rightarrow H^i(\overline{\mathcal{X}}_{et},\mathbb{Z})_{codiv} \rightarrow H_W^i(\overline{\mathcal{X}},\mathbb{Z})\rightarrow \textrm{\emph{Ker}}(H^{i+1}(\alpha_{\mathcal{X}}))\rightarrow 0$$
where $\textrm{\emph{Ker}}(H^{i+1}(\alpha_{\mathcal{X}}))\subseteq\mbox{\emph{Hom}}_{\mathbb{Z}}(H^{2d+2-(i+1)}(\mathcal{X}_{et},\mathbb{Q}(d))_{\geq0},\mathbb{Q})$ is a $\mathbb{Z}$-lattice.
\end{enumerate}
\end{prop}
\begin{proof}In order to ease the notation we set $\delta:=2d+2$. The exact triangle of Defintion \ref{def-cohomology} yields a long exact sequence
$$...\rightarrow \mbox{Hom}(H^{\delta-i}(\mathcal{X},\mathbb{Q}(d))_{\geq0},\mathbb{Q})\rightarrow
H^i(\overline{\mathcal{X}}_{et},\mathbb{Z})\rightarrow H_W^i(\overline{\mathcal{X}},\mathbb{Z})$$
$$\rightarrow \mbox{Hom}(H^{\delta-(i+1)}(\mathcal{X},\mathbb{Q}(d))_{\geq0},\mathbb{Q})\rightarrow
H^{i+1}(\overline{\mathcal{X}}_{et},\mathbb{Z})\rightarrow...$$
and a short exact sequence
\begin{equation}\label{oneexactsequence}
0\rightarrow \mbox{Coker}H^i(\alpha_{\mathcal{X}})\rightarrow H_W^i(\overline{\mathcal{X}},\mathbb{Z})\rightarrow \mbox{Ker}H^{i+1}(\alpha_{\mathcal{X}})\rightarrow0.
\end{equation}

Assume for the moment that $\X(\br)=\emptyset$. For $i\geq1$, the morphism
$$H^i(\alpha_{\mathcal{X}}):\mbox{Hom}(H^{\delta-i}(\mathcal{X},\mathbb{Q}(d))_{\geq0},\mathbb{Q})\rightarrow
H^i(\overline{\mathcal{X}}_{et},\mathbb{Z})
$$
is the following composite map:
$$\mbox{Hom}(H^{\delta-i}(\mathcal{X},\mathbb{Q}(d))_{\geq0},\mathbb{Q})\stackrel{\sim}{\rightarrow}\mbox{Hom}(H^{\delta-i}(\overline{\mathcal{X}}_{et},\mathbb{Z}(d))_{\geq0},\mathbb{Q})$$
$$\rightarrow \mbox{Hom}(H^{\delta-i}(\overline{\mathcal{X}}_{et},\mathbb{Z}(d))_{\geq0},\mathbb{Q}/\mathbb{Z})\simeq
H^i(\overline{\mathcal{X}}_{et},\mathbb{Z})$$
Since $H^{\delta-i}(\overline{\mathcal{X}}_{et},\mathbb{Z}(d))_{\geq0}$ is assumed to be finitely generated, the image of the morphism $H^i(\alpha_{\mathcal{X}})$ is the maximal divisible subgroup of $H^i(\overline{\mathcal{X}}_{et},\mathbb{Z})$, which we denote by $H^i(\overline{\mathcal{X}}_{et},\mathbb{Z})_{div}$. For $i=0$ one has
$$H^0(\alpha_{\mathcal{X}}):0=\mbox{Hom}(H^{\delta}(\mathcal{X},\mathbb{Q}(d)),\mathbb{Q})\rightarrow
H^0(\overline{\mathcal{X}}_{et},\mathbb{Z})=\mathbb{Z}$$
and $H^{\delta-i}(\mathcal{X},\mathbb{Q}(d))=H^i(\overline{\mathcal{X}}_{et},\mathbb{Z})=0$ for $i<0$. We obtain
$$\mbox{Coker}H^i(\alpha_{\mathcal{X}})=H^i(\overline{\mathcal{X}}_{et},\mathbb{Z})/H^i(\overline{\mathcal{X}}_{et},\mathbb{Z})_{div}
=:H^i(\overline{\mathcal{X}}_{et},\mathbb{Z})_{codiv}$$
for any $i\in\mathbb{Z}$. By definition of $\alpha_{\mathcal{X}}$, the kernel of $H^i(\alpha_{\mathcal{X}})$ can be identified with the kernel of the map 
$\mbox{Hom}(H^{\delta-i}(\overline{\mathcal{X}}_{et},\mathbb{Z}(d))_{\geq0},\mathbb{Q})\rightarrow \mbox{Hom}(H^{\delta-i}(\overline{\mathcal{X}}_{et},\mathbb{Z}(d))_{\geq0},\mathbb{Q}/\mathbb{Z})$. In other words, one has
$$\mathrm{Ker}H^i(\alpha_{\mathcal{X}})\simeq \mbox{Hom}(H^{\delta-i}(\overline{\mathcal{X}}_{et},\mathbb{Z}(d))_{\geq0},\mathbb{Z}).$$
The exact sequence (\ref{oneexactsequence}) therefore takes the form
$$0\rightarrow H^i(\overline{\mathcal{X}}_{et},\mathbb{Z})_{codiv} \rightarrow H_W^i(\overline{\mathcal{X}},\mathbb{Z})\rightarrow \mbox{Hom}_{\mathbb{Z}}(H^{2d+2-(i+1)}(\overline{\mathcal{X}}_{et},\mathbb{Z}(d))_{\geq0},\mathbb{Z})\rightarrow 0.$$
We obtain the second claim of the proposition.

Assume now that $\X(\br)\neq\emptyset$. Then  for $i\geq2$, the morphism $H^i(\tilde{\alpha}_{\mathcal{X}})$
is the map
$$\mbox{Hom}(H^{\delta-i}(\mathcal{X},\mathbb{Q}(d))_{\geq0},\mathbb{Q})\stackrel{\sim}{\rightarrow} \mbox{Hom}(H^{\delta-i}(\mathcal{X}_{et},\mathbb{Z}(d))_{\geq0},\mathbb{Q})$$
$$\rightarrow \mbox{Hom}(H^{\delta-i}(\mathcal{X}_{et},\mathbb{Z}(d))_{\geq0},\mathbb{Q}/\mathbb{Z})\simeq
\hat{H}_c^i(\mathcal{X}_{et},\mathbb{Z})$$
and $H^i(\alpha_{\mathcal{X}})$ is defined as the composite map 
$$H^i(\alpha_{\mathcal{X}}):\mbox{Hom}(H^{\delta-i}(\mathcal{X},\mathbb{Q}(d))_{\geq0},\mathbb{Q})\stackrel{H^i(\tilde{\alpha}_{\mathcal{X}})}{\longrightarrow}
\hat{H}_c^i(\mathcal{X}_{et},\mathbb{Z})\longrightarrow H^i(\overline{\mathcal{X}}_{et},\mathbb{Z}).$$
The same argument as in the case $\mathcal{X}(\mathbb{R})=\emptyset$ shows $\textrm{Im}(H^i(\tilde{\alpha}_{\mathcal{X}}))=\hat{H}_c^i(\mathcal{X}_{et},\mathbb{Z})_{div}$. Let us show that $h:\hat{H}_c^i(\mathcal{X}_{et},\mathbb{Z})_{div}\rightarrow H^i(\overline{\mathcal{X}}_{et},\mathbb{Z})_{div}$ is surjective for any $i\in\mathbb{Z}$. This is obvious for  $i\leq1$, since $H^i(\overline{\mathcal{X}}_{et},\mathbb{Z})_{div}=0$ for $i\leq1$.  For $i\geq2$, the maps $\hat{H}_c^i(\mathcal{X}_{et},\mathbb{Z})\rightarrow H^i(\overline{\mathcal{X}}_{et},\mathbb{Z})$ and 
$\hat{H}_c^i(\mathcal{X}_{et},\mathbb{Z})_{div}\rightarrow\hat{H}_c^i(\mathcal{X}_{et},\mathbb{Z})$
have both finite cokernels (note that $\hat{H}_c^i(\mathcal{X}_{et},\mathbb{Z})$ is of cofinite type for $i\geq 2$ by (\ref{cofinitetypehat})). It follows immediately that the cokernel of $h:\hat{H}_c^i(\mathcal{X}_{et},\mathbb{Z})_{div}\rightarrow H^i(\overline{\mathcal{X}}_{et},\mathbb{Z})_{div}$ is of finite exponent and divisible, hence vanishes. So $h$ is indeed surjective. We obtain $\textrm{Im}(H^i(\alpha_{\mathcal{X}}))=H^i(\overline{\mathcal{X}}_{et},\mathbb{Z})_{div}$ hence
$$\mbox{Coker}H^i(\alpha_{\mathcal{X}})=H^i(\overline{\mathcal{X}}_{et},\mathbb{Z})_{codiv}$$
for any $i\in\mathbb{Z}$. We now compute $\textrm{Ker}H^i(\alpha_{\mathcal{X}})$. By definition of $\tilde{\alpha}_{\mathcal{X}}$, the kernel of $H^i(\tilde{\alpha}_{\mathcal{X}})$ can be identified with the kernel of the map 
$\mbox{Hom}(H^{\delta-i}(\mathcal{X}_{et},\mathbb{Z}(d))_{\geq0},\mathbb{Q})\rightarrow \mbox{Hom}(H^{\delta-i}(\mathcal{X}_{et},\mathbb{Z}(d))_{\geq0},\mathbb{Q}/\mathbb{Z})$. Hence one has
\begin{equation}\label{lele}
\textrm{Ker}H^i(\tilde{\alpha}_{\mathcal{X}})\simeq \mbox{Hom}(H^{\delta-i}(\mathcal{X}_{et},\mathbb{Z}(d))_{\geq0},\mathbb{Z}).
\end{equation} 
We denote by $T_i$ the kernel  of the surjective map $h:\hat{H}_c^i(\mathcal{X}_{et},\mathbb{Z})_{div}\rightarrow H^i(\overline{\mathcal{X}}_{et},\mathbb{Z})_{div}$ so that there is an exact sequence 
\begin{equation}\label{ex-sequ-lele}
0\longrightarrow T_i\longrightarrow \hat{H}_c^i(\mathcal{X}_{et},\mathbb{Z})_{div}\stackrel{h}{\longrightarrow} H^i(\overline{\mathcal{X}}_{et},\mathbb{Z})_{div}\longrightarrow 0.
\end{equation}
Consider the obvious injective map $f:\textrm{Ker}H^i(\tilde{\alpha}_{\mathcal{X}})\rightarrow \textrm{Ker}H^i(\alpha_{\mathcal{X}})$, and the following morphism of exact sequences:
\[ \xymatrix{
0\ar[r]&\mbox{Ker}H^i(\tilde{\alpha}_{\mathcal{X}})\ar[r]\ar[d]^f& \mbox{Hom}(H^{\delta-i}(\mathcal{X},\mathbb{Q}(d))_{\geq0},\mathbb{Q})\ar[r]^{\hspace{1cm}H^i(\tilde{\alpha}_{\mathcal{X}})}\ar[d]^=&\hat{H}_c^{i}(\mathcal{X}_{et},\mathbb{Z})_{div}\ar[d]^h\ar[r]&0\\
0\ar[r]& \mbox{Ker}H^i(\alpha_{\mathcal{X}})\ar[r]& \mbox{Hom}(H^{\delta-i}(\mathcal{X},\mathbb{Q}(d))_{\geq0},\mathbb{Q})\ar[r]^{\hspace{1cm}H^i(\alpha_{\mathcal{X}})}& H^{i}(\overline{\mathcal{X}}_{et},\mathbb{Z})_{div}\ar[r]&0
}
\]
The snake lemma then yields an isomorphism $T_i=\textrm{Ker}(h)\simeq \textrm{Coker}(f)$. Using (\ref{lele}), we obtain an exact sequence
$$0\rightarrow \mbox{Hom}(H^{\delta-i}(\mathcal{X}_{et},\mathbb{Z}(d))_{\geq0},\mathbb{Z})\rightarrow \textrm{Ker}H^i(\alpha_{\mathcal{X}})\rightarrow T_i\rightarrow 0.$$
Moreover, $T_i$ is finite and $2$-torsion. We obtain the third claim of the proposition.

It follows from the third claim that  $H^i_W(\overline{\mathcal{X}},\mathbb{Z})$ is finitely generated for any $i\in\mathbb{Z}$. Indeed, $\textrm{Ker}H^{i+1}(\alpha_{\mathcal{X}})$ is finitely generated since it is a $\mathbb{Z}$-lattice in a finite dimensional $\mathbb{Q}$-vector space. Moreover, $\hat{H}^i_c(\mathcal{X}_{et},\mathbb{Z})$ is of cofinite type for $i\geq2$ and finitely generated for $i\leq1$. Hence $\hat{H}^i_c(\mathcal{X}_{et},\mathbb{Z})_{codiv}$ is finite for any $i\in\mathbb{Z}$ and so is $H^i(\overline{\mathcal{X}}_{et},\mathbb{Z})_{codiv}$ by (\ref{ex-sequ-lele}).

Moreover, one has $H^i(\overline{\mathcal{X}}_{et},\mathbb{Z})=0$ for $i<0$ and $H^0(\overline{\mathcal{X}}_{et},\mathbb{Z})=\mathbb{Z}$. Moreover one has $H^i(\overline{\mathcal{X}}_{et},\mathbb{Z})\simeq\hat{H}_c^i(\mathcal{X}_{et},\mathbb{Z})$ for $i$ large (by Lemma \ref{lem-hat-bar}) and $\hat{H}_c^i(\mathcal{X}_{et},\mathbb{Z})=0$ for $i>2d+2$ by (\ref{cofinitetypehat}), hence $H^i(\overline{\mathcal{X}}_{et},\mathbb{Z})=0$ for $i$ large. Finally, $\textrm{Ker}H^{i+1}(\alpha_{\mathcal{X}})\subset \mbox{Hom}(H^{2d+2-(i+1)}(\mathcal{X},\mathbb{Q}(d))_{\geq0},\mathbb{Q})=0$ for $i\leq 0$ as well as for $i\geq 2d+2$. The first claim of the proposition follows. 

\end{proof}

\begin{thm}\label{cor-functoriality}
Let $f:\mathcal{X}\rightarrow\mathcal{Y}$ be a  morphism of relative dimension $c$ between proper regular connected arithmetic schemes, such that $\textbf{\emph{L}}(\mathcal{X}_{et},d_{\mathcal{X}})_{\geq0}$ and $\textbf{\emph{L}}(\mathcal{Y}_{et},d_{\mathcal{Y}})_{\geq0}$ hold, where $d_{\mathcal{X}}$ (respectively $d_{\mathcal{Y}}$) denotes the dimension of $\mathcal{X}$ (respectively of $\mathcal{Y}$). We choose complexes $\mbox{\emph{R}}\Gamma_W(\overline{\mathcal{X}},\mathbb{Z})$ and $\mbox{\emph{R}}\Gamma_W(\overline{\mathcal{Y}},\mathbb{Z})$ as in Definition \ref{def-cohomology}. Assume either that $c\geq0$ or that $\textbf{\emph{L}}(\mathcal{X}_{et},d_{\mathcal{X}})$ holds.

Then there exists
a \emph{unique} map in $\mathcal{D}$
$$f^*_{W}:\mbox{\emph{R}}\Gamma_W(\overline{\mathcal{Y}},\mathbb{Z})\rightarrow \mbox{\emph{R}}\Gamma_W(\overline{\mathcal{X}},\mathbb{Z})$$
which sits in a morphism of exact triangles
\[ \xymatrix{
\mbox{\emph{RHom}}(\mbox{\emph{R}}\Gamma(\mathcal{X},\mathbb{Q}(d_{\mathcal{X}}))_{\geq0},\mathbb{Q}[-2d_{\mathcal{X}}-2])\ar[r]& \mbox{\emph{R}}\Gamma(\overline{\mathcal{X}}_{et},\mathbb{Z})\ar[r]& \mbox{\emph{R}}\Gamma_W(\overline{\mathcal{X}},\mathbb{Z})\\
\mbox{\emph{RHom}}(\mbox{\emph{R}}\Gamma(\mathcal{Y},\mathbb{Q}(d_{\mathcal{Y}}))_{\geq0},\mathbb{Q}[-2d_{\mathcal{Y}}-2])\ar[r]\ar[u]& \mbox{\emph{R}}\Gamma(\overline{\mathcal{Y}}_{et},\mathbb{Z})\ar[r]\ar[u]_{f^*_{et}}& \mbox{\emph{R}}\Gamma_W(\overline{\mathcal{Y}},\mathbb{Z})\ar[u]_{f^*_{W}}
}
\]
In particular, if $\mathcal{X}$ satisfies $\textbf{\emph{L}}(\mathcal{X}_{et},d_{\X})_{\geq0}$ then $\mbox{\emph{R}}\Gamma_{W}(\overline{\mathcal{X}},\mathbb{Z})$ is well defined up to a \emph{unique} isomorphism in $\mathcal{D}$.
\end{thm}
\begin{proof}
Let $\mathcal{X}$ and $\mathcal{Y}$ be connected, proper and regular arithmetic schemes of dimension $d_{\mathcal{X}}$ and $d_{\mathcal{Y}}$ respectively. We set $\delta_{\X}:=2d_{\X}+2$ and $\delta_{\Y}:=2d_{\Y}+2$. We choose complexes $\mbox{R}\Gamma_W(\overline{\mathcal{X}},\mathbb{Z})$ and $\mbox{R}\Gamma_W(\overline{\mathcal{Y}},\mathbb{Z})$ as in Definition \ref{def-cohomology}. Let $f:\mathcal{X}\rightarrow\mathcal{Y}$ be a morphism of relative dimension $c=d_{\mathcal{X}}-d_{\mathcal{Y}}$. The morphism $f$ is proper and the map
$$z^n(\mathcal{X},*)\rightarrow z^{n-c}(\mathcal{Y},*)$$
induces a morphism
$$f_*\mathbb{Q}(d_{\mathcal{X}})\rightarrow \mathbb{Q}(d_{\mathcal{Y}}))[-2c]$$
of complexes of abelian Zariski sheaves on $\mathcal{Y}$. We need to see that
\begin{equation}\label{f=Rf}
f_*\mathbb{Q}(d_{\mathcal{X}})\simeq\mbox{R}f_*\mathbb{Q}(d_{\mathcal{X}}).
\end{equation}
Localizing over the base, it is enough to show this fact for $f$ a proper map over a discrete valuation ring. Over a discrete valuation ring, Zariski hypercohomology of the cycle complex coincides with its cohomology as a complex of abelian group (see \cite{Levine-localization} and \cite{Geisser-MotvicCohDed} Theorem 3.2). This yields (\ref{f=Rf}) and a morphism of complexes
$$\mbox{R}\Gamma(\mathcal{X},\mathbb{Q}(d_{\mathcal{X}}))
\simeq \mbox{R}\Gamma(\mathcal{Y},f_*\mathbb{Q}(d_{\mathcal{X}}))\rightarrow
\mbox{R}\Gamma(\mathcal{Y},\mathbb{Q}(d_{\mathcal{Y}}))[-2c].$$
If $c\geq0$ this induces a morphism
$$\mbox{R}\Gamma(\mathcal{X},\mathbb{Q}(d_{\mathcal{X}}))_{\geq0}
\rightarrow\mbox{R}\Gamma(\mathcal{Y},\mathbb{Q}(d_{\mathcal{Y}}))_{\geq0}[-2c].$$
If $c<0$ then  $\textbf{L}(\mathcal{X}_{et},d_{\mathcal{X}})$ holds by assumption. It follows that $H^i(\mathcal{X},\mathbb{Q}(d_{\mathcal{X}}))=0$ for $i<0$. Indeed, we have $H^i(\X_{et},\mathbb{Z}(d))\simeq H^i(\mathcal{X},\mathbb{Q}(d_{\mathcal{X}}))$ for $i<0$ (by Lemma \ref{firstLemma} for $\X(\mathbb{R})=\emptyset$ and by (\ref{perfpair}) for the general case). Hence for $i<0$, $H^i(\mathcal{X},\mathbb{Q}(d_{\mathcal{X}}))\simeq H^i(\X_{et},\mathbb{Z}(d))$ is both uniquely divisible and finitely generated by $\textbf{L}(\mathcal{X}_{et},d_{\mathcal{X}})$, hence must be trivial. So we may consider the map
$$\mbox{R}\Gamma(\mathcal{X},\mathbb{Q}(d_{\mathcal{X}}))_{\geq0}
\stackrel{\sim}{\leftarrow}\mbox{R}\Gamma(\mathcal{X},\mathbb{Q}(d_{\mathcal{X}}))
\rightarrow
\mbox{R}\Gamma(\mathcal{Y},\mathbb{Q}(d_{\mathcal{Y}}))[-2c]\rightarrow \mbox{R}\Gamma(\mathcal{Y},\mathbb{Q}(d_{\mathcal{Y}}))_{\geq0}[-2c].$$
In both cases we get a morphism
$$\mbox{RHom}(\mbox{R}\Gamma(\mathcal{Y},\mathbb{Q}(d_{\mathcal{Y}}))_{\geq0},
\mathbb{Q}[-\delta_{\mathcal{Y}}])
\rightarrow \mbox{RHom}(\mbox{R}\Gamma(\mathcal{X},\mathbb{Q}(d_{\mathcal{X}}))_{\geq0},
\mathbb{Q}[-\delta_{\mathcal{X}}])$$
such that the following square is commutative:
\[ \xymatrix{
\mbox{RHom}(\mbox{R}\Gamma(\mathcal{Y},\mathbb{Q}(d_{\mathcal{Y}}))_{\geq0},
\mathbb{Q}[-\delta_{\mathcal{Y}}])\ar[r]^{\,\,\,\,\,\,\,\,\,\,\,\,\,\,\,\,\,\,\,\,\,\,\,\,\,\,\,\,\,\,\,\,\,\,\,\,\alpha_{\mathcal{Y}}}\ar[d]
&\mbox{R}\Gamma(\overline{\mathcal{Y}}_{et},\mathbb{Z})\ar[d]_{}  \\
 \mbox{RHom}(\mbox{R}\Gamma(\mathcal{X},\mathbb{Q}(d_{\mathcal{X}}))_{\geq0},
\mathbb{Q}[-\delta_{\mathcal{X}}])\ar[r]^{\,\,\,\,\,\,\,\,\,\,\,\,\,\,\,\,\,\,\,\,\,\,\,\,\,\,\,\,\,\,\,\,\,\,\,\,\alpha_{\mathcal{X}}}
&\mbox{R}\Gamma(\overline{\mathcal{X}}_{et},\mathbb{Z})
}
\]
Showing that the diagram above is indeed commutative is tedious but straightforward, using the push-forward map defined in \cite{Geisser-Duality} Corollary 7.2 (b). Hence there exists a morphism
$$f_W^*:\mbox{R}\Gamma_W(\overline{\mathcal{Y}},\mathbb{Z})\rightarrow
\mbox{R}\Gamma_W(\overline{\mathcal{X}},\mathbb{Z})$$
inducing a morphism of exact triangles. We claim that such a morphism $f_W^*$ is unique. In order to ease the notations, we set
$$D_{\mathcal{X}}:=\mbox{RHom}(\mbox{R}\Gamma(\mathcal{X},\mathbb{Q}(d_{\mathcal{X}}))_{\geq0},
\mathbb{Q}[-\delta_{\mathcal{X}}])\mbox{ and }D_{\mathcal{Y}}:=\mbox{RHom}(\mbox{R}\Gamma(\mathcal{Y},\mathbb{Q}(d_{\mathcal{Y}}))_{\geq0},
\mathbb{Q}[-\delta_{\mathcal{Y}}]).$$
The complexes $\mbox{R}\Gamma_W(\overline{\mathcal{X}},\mathbb{Z})$ and $\mbox{R}\Gamma_W(\overline{\mathcal{Y}},\mathbb{Z})$ are both perfect complexes of abelian groups, since they are bounded complexes with finitely generated cohomology groups. Applying the functor $\mbox{Hom}_{\mathcal{D}}(-,\mbox{R}\Gamma_W(\overline{\mathcal{X}},\mathbb{Z}))$ to the exact triangle
$$D_{\mathcal{Y}}\rightarrow \mbox{R}\Gamma(\overline{\mathcal{Y}}_{et},\mathbb{Z})\rightarrow \mbox{R}\Gamma_W(\overline{\mathcal{Y}},\mathbb{Z})\rightarrow D_{\mathcal{Y}}[1]$$
we obtain an exact sequence of abelian groups:
$$\mbox{Hom}_{\mathcal{D}}(D_{\mathcal{Y}}[1],\mbox{R}\Gamma_W(\overline{\mathcal{X}},\mathbb{Z}))\rightarrow
\mbox{Hom}_{\mathcal{D}}(\mbox{R}\Gamma_W(\overline{\mathcal{Y}},\mathbb{Z}),\mbox{R}\Gamma_W(\overline{\mathcal{X}},\mathbb{Z}))$$
$$\rightarrow\mbox{Hom}_{\mathcal{D}}
(\mbox{R}\Gamma(\overline{\mathcal{Y}}_{et},\mathbb{Z}),\mbox{R}\Gamma_W(\overline{\mathcal{X}},\mathbb{Z})).$$
On the one hand, $\mbox{Hom}_{\mathcal{D}}(D_{\mathcal{Y}}[1],\mbox{R}\Gamma_W(\overline{\mathcal{X}},\mathbb{Z}))$ is uniquely divisible since $D_{\mathcal{Y}}[1]$ is a complex of $\mathbb{Q}$-vector spaces. On the other hand, the abelian group $\mbox{Hom}_{\mathcal{D}}(\mbox{R}\Gamma_W(\overline{\mathcal{Y}},\mathbb{Z}),
\mbox{R}\Gamma_W(\overline{\mathcal{X}},\mathbb{Z}))$ is finitely generated
as it follows from the spectral sequence
$$\prod_{i\in\mathbb{Z}}\mbox{Ext}^p(H^i_W(\overline{\mathcal{Y}},\mathbb{Z}),H^{q+i}_W(\overline{\mathcal{X}},\mathbb{Z}))
\Rightarrow H^{p+q}(\mbox{RHom}(\mbox{R}\Gamma_W(\overline{\mathcal{Y}},\mathbb{Z}),\mbox{R}\Gamma_W(\overline{\mathcal{X}},\mathbb{Z})))$$
since $\mbox{R}\Gamma_W(\overline{\mathcal{X}},\mathbb{Z})$ and $\mbox{R}\Gamma_W(\overline{\mathcal{Y}},\mathbb{Z})$ are both perfect. Hence the morphism
$$\mbox{Hom}_{\mathcal{D}}(\mbox{R}\Gamma_W(\overline{\mathcal{Y}},\mathbb{Z}),
\mbox{R}\Gamma_W(\overline{\mathcal{X}},\mathbb{Z}))\rightarrow
\mbox{Hom}_{\mathcal{D}}
(\mbox{R}\Gamma(\overline{\mathcal{Y}}_{et},\mathbb{Z}),\mbox{R}\Gamma_W(\overline{\mathcal{X}},\mathbb{Z}))$$
is injective, which implies the uniqueness of the morphism $f_W^*$ sitting in the morphism of exact triangles of the theorem.

Assume now that $\mathcal{X}$ satisfies $\textbf{L}(\mathcal{X}_{et},d_{\X})_{\geq0}$, and let $\mbox{R}\Gamma_{W}(\overline{\mathcal{X}},\mathbb{Z})$ and $\mbox{R}\Gamma_{W}(\overline{\mathcal{X}},\mathbb{Z})'$
be two complexes given with exact triangles as in Definition \ref{def-cohomology}. Then the identity map $Id:\mathcal{X}\rightarrow\mathcal{X}$ induces a unique isomorphism $\mbox{R}\Gamma_{W}(\overline{\mathcal{X}},\mathbb{Z})
\simeq\mbox{R}\Gamma_{W}(\overline{\mathcal{X}},\mathbb{Z})'$ in $\mathcal{D}$.

\end{proof}

We denote by $H^i_{cont}(\overline{\mathcal{X}}_{et},\mathbb{Z}_l):=H^i(\textrm{R}\underleftarrow{\,\textrm{lim}}\, \textrm{R}\Gamma(\overline{\mathcal{X}}_{et},\mathbb{Z}/l^{\nu}\mathbb{Z}))$ continuous $l$-adic cohomology. 

\begin{cor}\label{cor-comp-l-adic} Let $\X$ be a proper regular connected arithmetic scheme of dimension $d$ satisfying $\textbf{\emph{L}}(\mathcal{X}_{et},d)_{\geq0}$. For any prime number $l$ and any $i\in\mathbb{Z}$, there is an isomorphism
$$H^i_W(\overline{\mathcal{X}},\mathbb{Z})\otimes\mathbb{Z}_l\simeq H^i_{cont}(\overline{\mathcal{X}}_{et},\mathbb{Z}_l).$$
\end{cor}
\begin{proof}
Consider the exact triangle
$$\mbox{RHom}(\mbox{R}\Gamma(\mathcal{X},\mathbb{Q}(d))_{\geq0},\mathbb{Q}[-2d-2])\rightarrow \mbox{R}\Gamma(\overline{\mathcal{X}}_{et},\mathbb{Z})\rightarrow \mbox{R}\Gamma_W(\overline{\mathcal{X}},\mathbb{Z})$$
Let $n$ be any positive integer. Applying the functor $-\otimes^L_{\mathbb{Z}}\mathbb{Z}/n\mathbb{Z}$
we obtain an isomorphism 
$$\mbox{R}\Gamma(\overline{\mathcal{X}}_{et},\mathbb{Z}/n\mathbb{Z})
\simeq \mbox{R}\Gamma(\overline{\mathcal{X}}_{et},\mathbb{Z})\otimes^L_{\mathbb{Z}}\mathbb{Z}/n\mathbb{Z}\stackrel{\sim}{\longrightarrow}\mbox{R}\Gamma_W(\overline{\mathcal{X}},\mathbb{Z})\otimes^L_{\mathbb{Z}}\mathbb{Z}/n\mathbb{Z}$$
since $\mbox{RHom}(\mbox{R}\Gamma(\mathcal{X},\mathbb{Q}(d_{\mathcal{X}}))_{\geq0},\mathbb{Q}[-2d-2])\otimes^L_{\mathbb{Z}}\mathbb{Z}/n\mathbb{Z}\simeq 0$.
Let $l$ be a prime number and let $\nu$ be a positive integer. We obtain a short exact sequence
\begin{equation}\label{shorty}
0\rightarrow H^i_W(\overline{\mathcal{X}},\mathbb{Z})_{l^{\nu}}\rightarrow H^i(\overline{\mathcal{X}}_{et},\mathbb{Z}/l^{\nu}\mathbb{Z})
\rightarrow {_{l^{\nu}}}H^{i+1}_W(\overline{\mathcal{X}},\mathbb{Z})\rightarrow0
\end{equation}
for any $i\in\mathbb{Z}$. By left exactness of projective limits we get
$$0\rightarrow\underleftarrow{\mbox{lim}}\,H^i_W(\overline{\mathcal{X}},\mathbb{Z})_{l^{\nu}}\rightarrow \underleftarrow{\mbox{lim}}\,H^i(\overline{\mathcal{X}}_{et},\mathbb{Z}/l^{\nu}\mathbb{Z})
\rightarrow \underleftarrow{\mbox{lim}}\,{_{l^{\nu}}}H^{i+1}_W(\overline{\mathcal{X}},\mathbb{Z}).$$
But $\underleftarrow{\mbox{lim}}\,{_{l^{\nu}}}H^{i+1}_W(\overline{\mathcal{X}},\mathbb{Z})=0$ since $H^{i+1}_W(\overline{\mathcal{X}},\mathbb{Z})$ is finitely generated. Moreover,
we have
$$H^i_{cont}(\overline{\mathcal{X}}_{et},\mathbb{Z}_l)\simeq \underleftarrow{\mbox{lim}}\,H^i(\overline{\mathcal{X}}_{et},\mathbb{Z}/l^{\nu}\mathbb{Z})$$
since $H^i(\overline{\mathcal{X}}_{et},\mathbb{Z}/l^{\nu}\mathbb{Z})$ is finite by (\ref{shorty}).
\end{proof}

\subsection{Relationship with Lichtenbaum's definition over finite fields}

Let $Y$ be a scheme of finite type over a finite field $k$. We denote by $G_k$ and $W_k$
the Galois group and the Weil group of $k$ respectively.
The (small) Weil-\'etale topos $Y^{sm}_W$ is the category of $W_k$-equivariant
sheaves of sets on the \'etale site of $Y\otimes_{k}\overline{k}$.
The big Weil-\'etale topos is defined  \cite{Flach-moi} as the fiber product
$$Y_W:=Y_{et}\times_{\overline{\mbox{Spec}(\mathbb{Z})}_{et}}\overline{\mbox{Spec}(\mathbb{Z})}_{W}\simeq
Y_{et}\times_{B^{sm}_{G_k}}B_{W_k}$$
where ${B^{sm}_{G_k}}$ (resp. $B_{W_k}$) denotes
the small classifying topos of $G_k$
(resp. the big classifying topos of $W_k$). The topoi $Y_W$ and $Y_{W}^{sm}$ are cohomologically equivalent (see \cite{Flach-moi} Corollary 2). Therefore, by \cite{Geisser-Weiletale} one has an exact triangle in the derived category of abelian sheaves on $Y_{et}$
$$\mathbb{Z}\rightarrow \mbox{R}\gamma_*\mathbb{Z}\rightarrow \mathbb{Q}[-1]\rightarrow\mathbb{Z}[1]$$
where $\gamma:Y_W\rightarrow Y_{et}$ is the first projection. Applying $\mbox{R}\Gamma(Y_{et},-)$ and rotating, we get
$$\mbox{R}\Gamma(Y_{et},\mathbb{Q}[-2])\stackrel{a_{Y}}{\rightarrow} \mbox{R}\Gamma(Y_{et},\mathbb{Z})\rightarrow \mbox{R}\Gamma(Y_{W},\mathbb{Z})\rightarrow \mbox{R}\Gamma(Y_{et},\mathbb{Q}[-2])[1].$$

\begin{thm}\label{thm-comparison-char-p} Let $Y$ be a $d$-dimensional connected projective smooth scheme over $k$ satisfying $\textbf{\emph{L}}(Y_{et},d)_{\geq0}$. Then there is an isomorphism in $\mathcal{D}$
\begin{equation}\label{iso66}
\mbox{\emph{R}}\Gamma(Y_{W},\mathbb{Z})\stackrel{\sim}{\longrightarrow}\mbox{\emph{R}}\Gamma_W(Y,\mathbb{Z})
\end{equation}
where $\mbox{\emph{R}}\Gamma(Y_{W},\mathbb{Z})$ is the cohomology of the Weil-\'etale topos and $\mbox{\emph{R}}\Gamma_W(Y,\mathbb{Z})$ is the complex defined in this paper. Moreover, the exact triangle of Definition \ref{def-cohomology} is isomorphic to Geisser's triangle (\cite{Geisser-Weiletale} Corollary 5.2).
\end{thm}
\begin{proof} We shall define a commutative square in $\mathcal{D}$
\[ \xymatrix{
\mbox{R}\Gamma(Y_{et},\mathbb{Q}[-2])\ar[d]_{\simeq}\ar[r]^{a_Y}&\mbox{R}\Gamma(Y_{et},\mathbb{Z})\ar[d]_{Id}\\
\mbox{RHom}(\mbox{R}\Gamma(Y,\mathbb{Q}(d))_{\geq0},\mathbb{Q}[-2d-2])
\ar[r]^{\,\,\,\,\,\,\,\,\,\,\,\,\,\,\,\,\,\,\,\,\,\,\,\,\,\,\,\,\,\,\,\,\,\,\,\,\alpha_Y} &\mbox{R}\Gamma(Y_{et},\mathbb{Z})
}
\]
where the vertical maps are isomorphisms. The existence of the isomorphism (\ref{iso66}) and the compatibility with Geisser's triangle (\cite{Geisser-Weiletale} Corollary 5.2 for $\mathcal{G}^{\cdot}=\bz$) will immediately follow. Moreover, the uniqueness of (\ref{iso66}) will follow from the argument of the proof of Theorem \ref{cor-functoriality}.

One is therefore reduced to define the commutative square above. Replacing $k$ with a finite extension if necessary, one may suppose that $Y$ is geometrically connected over $k$. One has $H^{2d}(Y,\mathbb{Q}(d))=CH^d(Y)_{\mathbb{Q}}\simeq\mathbb{Q}$ (\cite{Kato-Saito86} Theorem 6.1) and $H^{i}(Y,\mathbb{Q}(d))=0$ for $i>2d$. This yields a map $\mbox{R}\Gamma(Y,\mathbb{Q}(d))\rightarrow \mathbb{Q}[-2d]$. The morphism
$$\mbox{RHom}_{Y}(\mathbb{Q},\mathbb{Q}(d))\rightarrow \mbox{RHom}(\mbox{R}\Gamma(Y,\mathbb{Q}),\mbox{R}\Gamma(Y,\mathbb{Q}(d)))
\rightarrow \mbox{RHom}(\mbox{R}\Gamma(Y,\mathbb{Q}),\mathbb{Q}[-2d])$$
induces a morphism
\begin{equation}\label{unemapdeplus+}
\mbox{R}\Gamma(Y_{et},\mathbb{Q})\simeq\mbox{R}\Gamma(Y,\mathbb{Q})\rightarrow \mbox{RHom}(\mbox{R}\Gamma(Y,\mathbb{Q}(d))_{\geq0},\mathbb{Q}[-2d]).
\end{equation}
It follows from Conjecture $\textbf{L}(Y_{et},d)_{\geq0}$, Lemma \ref{prop-duality-Z-coefs} and from the fact that $H^i(Y_{et},\mathbb{Z})$ is finite for $i\neq 0,2$, that the group $H^i(Y_{et},\mathbb{Z}(d))_{\geq0}$ is finite for $i\neq 2d,2d+2$ and torsion for $i\neq 2d+2$. It follows easily that (\ref{unemapdeplus+}) is a quasi-isomorphism.

It remains to check the commutativity of the square above. The complex $$D_Y:= \mbox{RHom}(\mbox{R}\Gamma(Y,\mathbb{Q}(d))_{\geq0},\mathbb{Q}[-2d-2])\simeq\mbox{R}\Gamma(Y_{et},\mathbb{Q}[-2])$$
is concentrated in degree 2. It follows that both $a_Y$ and $\alpha_Y$ uniquely factor through the truncated complex
$\mbox{R}\Gamma(Y_{et},\mathbb{Z})_{\leq2}$. It is therefore enough to show that the square
\[ \xymatrix{
\mbox{R}\Gamma(Y_{et},\mathbb{Q}[-2])\ar[d]\ar[r]&\mbox{R}\Gamma(Y_{et},\mathbb{Z})_{\leq2}\ar[d]_{Id}\\
\mbox{RHom}(\Gamma(Y,\mathbb{Q}(d))_{\geq0},\mathbb{Q}[-2d-2])
\ar[r]
&\mbox{R}\Gamma(Y_{et},\mathbb{Z})_{\leq2}
}
\]
commutes. The exact triangle
$$\mathbb{Z}[0]\rightarrow\mbox{R}\Gamma(Y_{et},\mathbb{Z})_{\leq2}\rightarrow\pi_1(Y_{et})^D[-2]\rightarrow \mathbb{Z}[1]$$
induces an exact sequence of abelian groups
$$\mbox{Hom}_{\mathcal{D}}(D_Y,\mathbb{Z}[0])\rightarrow\mbox{Hom}_{\mathcal{D}}(D_Y,\mbox{R}\Gamma(Y_{et},\mathbb{Z})_{\leq2})\rightarrow\mbox{Hom}_{\mathcal{D}}(D_Y,\pi_1(Y_{et})^D[-2])$$
which shows that the map
$$\mbox{Hom}_{\mathcal{D}}(D_Y,\mbox{R}\Gamma(Y_{et},\mathbb{Z})_{\leq2})\rightarrow \mbox{Hom}_{\mathcal{D}}(D_Y,\pi_1(Y_{et})^D[-2])$$
is injective. Indeed, $\mathbb{Z}[0]\simeq[\mathbb{Q}\rightarrow \mathbb{Q}/\mathbb{Z}]$ has an injective resolution of length one and $D_Y$ is concentrated in degree $2$, hence $\mbox{Hom}_{\mathcal{D}}(D_Y,\mathbb{Z}[0])=0$.
In view of the quasi-isomorphism $$D_Y\simeq\mbox{Hom}(H^{2d}(Y,\mathbb{Q}(d)),\mathbb{Q})[-2]\simeq H^0(Y,\mathbb{Q})[-2],$$
one is reduced to show the commutativity of the following square (of abelian groups):
\[ \xymatrix{
H^0(Y_{et},\mathbb{Q})\ar[d]\ar[r]^{d_2^{0,1}}&H^2(Y_{et},\mathbb{Z})\ar[d]_{Id}\\
\mbox{Hom}(H^{2d}(Y,\mathbb{Q}(d)),\mathbb{Q})
\ar[r]^{\,\,\,\,\,\,\,\,\,\,\,\,\,\,\,\,\,\,\,\,\,\,\,\,H^2(\alpha_Y)} &H^2(Y_{et},\mathbb{Z})
}
\]
By construction the map $H^2(\alpha_Y)$ is the following composition
$$\mbox{Hom}(H^{2d}(Y,\mathbb{Q}(d)),\mathbb{Q})\simeq
\mbox{Hom}(CH^{d}(Y),\mathbb{Q}) \rightarrow \mbox{Hom}(CH^{d}(Y),\mathbb{Q}/\mathbb{Z})\stackrel{\sim}{\leftarrow}\pi_1(Y_{et})^D\simeq H^2(Y_{et},\mathbb{Z})$$
where  the isomorphism $CH^{d}(Y)^D\stackrel{\sim}{\leftarrow}\pi_1(Y_{et})^D$ is the dual of the map $CH^{d}(Y)\rightarrow\pi_1(Y_{et})^{ab}$ given by class field theory, which is injective with dense image (see \cite{Wiesend07} Corollary 3). The top horizontal map in the last commutative square is the differential $d_2^{0,1}$ of the spectral sequence
$$H^i(Y_{et},R^j(\gamma_*)\mathbb{Z})\Rightarrow H^{i+j}(Y_{W},\mathbb{Z}).$$
There is a canonical isomorphism
$$H^0(Y_{et},R^1(\gamma_*)\mathbb{Z})=\underrightarrow{\mbox{lim}}_{_{k'/k}}\,\mbox{Hom}(W_{k'},\mathbb{Z})=\mbox{Hom}(W_k,\mathbb{Q}),$$
$k'/k$ runs over the finite extensions of $k$, as it follows from the isomorphism of pro-discrete groups $\pi_1(Y'_W,p)\simeq \pi_1(Y'_{et},p)\times_{G_k} W_k$,
which is valid for any $Y'$ connected \'etale over $Y$. Then the left vertical map in the last square above is the map
$$\mbox{deg}^*:\mbox{Hom}(W_k,\mathbb{Q})\rightarrow \mbox{Hom}(CH^d(Y),\mathbb{Q})$$ induced by the degree map
$$\mbox{deg}: CH^d(Y)\rightarrow \mathbb{Z}\simeq W_k$$
and the differential $d_2^{0,1}$  is the following map:
$$\mbox{Hom}(W_k,\mathbb{Q})=\mbox{Hom}(\pi_1(Y_W,p),\mathbb{Q})\rightarrow \mbox{Hom}_c(\pi_1(Y_W,p),\mathbb{Q}/\bz)\simeq \pi_1(Y_{et},p)^D$$
where $\mathrm{Hom}_c(-,-)$ denotes the group of continuous morphisms. One is therefore reduced to observe that the square
\[ \xymatrix{
\mbox{Hom}(W_k,\mathbb{Q})\ar[d]_{\mbox{deg}^*}\ar[r]^{d_2^{0,1}}&\pi_1(Y_{et})^D\ar[d]_{Id}\\
\mbox{Hom}(CH^d(Y),\mathbb{Q})
\ar[r]^{^{\,\,\,\,\,\,\,\,\,\,\,\,\,\,\,\,\,\,H^2(\alpha_Y)}} &\pi_1(Y_{et})^D
}
\]
commutes.

\end{proof}

We shall need the following result in Section \ref{sect-TNC}.
\begin{prop}\label{prop-functoriality-flat-p}
Let $f:Y\rightarrow\mathcal{X}$ be a  morphism of proper regular arithmetic schemes, such that $\mathcal{X}$ is flat over $\mathrm{Spec}(\mathbb{Z})$ and $Y$ has characteristic $p$. Assume that $\mathcal{X}$ has pure dimension $d$ and that $\textbf{\emph{L}}(\mathcal{X}_{et},d)_{\geq0}$ holds. Then there exists
a \emph{unique} map in $\mathcal{D}$
$$\tilde{f}^*_{W}:\mbox{\emph{R}}\Gamma_W(\overline{\mathcal{X}},\mathbb{Z})\rightarrow \mbox{\emph{R}}\Gamma(Y_W,\mathbb{Z})$$
which renders the following square commutative
\[ \xymatrix{
 \mbox{\emph{R}}\Gamma(Y_{et},\mathbb{Z})\ar[r]& \mbox{\emph{R}}\Gamma(Y_W,\mathbb{Z}) \\
 \mbox{\emph{R}}\Gamma(\overline{\mathcal{X}}_{et},\mathbb{Z})\ar[r]\ar[u]_{f^*_{et}}& \mbox{\emph{R}}\Gamma_W(\overline{\mathcal{X}},\mathbb{Z})\ar[u]_{\tilde{f}^*_{W}}
}
\]
where $f_{et}^*$ is induced by the map $Y_{et}\rightarrow\overline{\mathcal{X}}_{et}$ and $\mbox{\emph{R}}\Gamma(Y_W,\mathbb{Z})$ is the cohomology of the Weil-\'etale topos.
\end{prop}

\begin{proof}
We shall define a morphism of exact triangles
\[ \xymatrix{
\mbox{R}\Gamma(Y,\mathbb{Q})[-2]\ar[r]^{a_Y}& \mbox{R}\Gamma(Y_{et},\mathbb{Z})\ar[r]& \mbox{R}\Gamma(Y_W,\mathbb{Z}) \\
\mbox{RHom}(R\Gamma(\mathcal{X},\mathbb{Q}(d))_{\geq0},\mathbb{Q}[-\delta])\ar[r]^{\hspace{1.6cm}\alpha_{\X}}\ar[u]_0& \mbox{R}\Gamma(\overline{\mathcal{X}}_{et},\mathbb{Z})\ar[r]\ar[u]_{f^*_{et}}& \mbox{R}\Gamma_W(\overline{\mathcal{X}},\mathbb{Z})\ar[u]
}
\]
where the left vertical map is the zero map and $\delta:=2d+2$. The existence of $\tilde{f}_W^*$ will follow from the commutativity of the left square and the uniqueness of $\tilde{f}_W^*$ will follow from the facts that  $\mbox{R}\Gamma(Y_W,\mathbb{Z})$ and $\mbox{R}\Gamma_W(\overline{\mathcal{X}},\mathbb{Z})$ are both perfect (by \cite{Lichtenbaum-finite-field} Theorem 7.4 and Proposition \ref{finitelygenerated-cohomology} respectively) and that $\mbox{RHom}(\mbox{R}\Gamma(\mathcal{X},\mathbb{Q}(d))_{\geq0},\mathbb{Q}[-\delta])$ is a complex of $\mathbb{Q}$-vector spaces, as in the proof of Theorem \ref{cor-functoriality}.

In order to show that such a morphism of exact triangles does exist, we only need to check that the composite map
\begin{equation}\label{triv-map}
\mbox{RHom}(\mbox{R}\Gamma(\mathcal{X},\mathbb{Q}(d))_{\geq0},\mathbb{Q}[-\delta])\stackrel{\alpha_{\X}}{\longrightarrow}
\mbox{R}\Gamma(\overline{\mathcal{X}}_{et},\mathbb{Z})\stackrel{f^*_{et}}{\longrightarrow} \mbox{R}\Gamma(Y_{et},\mathbb{Z})
\end{equation}
is the zero map. The complex of $\mathbb{Q}$-vector spaces $D_{\mathcal{X}}:=\mbox{RHom}(\mbox{R}\Gamma(\mathcal{X},\mathbb{Q}(d))_{\geq0},\mathbb{Q}[-\delta])$ is concentrated in degrees $\geq3$, because $\mathcal{X}$ is flat (see (\ref{flat-vanish})). Moreover, we have $H^0(Y_{et},\mathbb{Z})=\mathbb{Z}^{\pi_0(Y)}$, $H^1(Y_{et},\mathbb{Z})=0$,  $H^2(Y_{et},\mathbb{Z})\simeq(\mathbb{Q}/\mathbb{Z})^{\pi_0(Y)}\oplus A$ is the direct sum of a divisible group and a finite group $A$, and
 $H^i(Y_{et},\mathbb{Z})$ is  finite for $i>2$. The spectral sequence
$$\prod_{i\in\mathbb{Z}}\mbox{Ext}^p(H^i(D_{\mathcal{X}}),H^{q+i}(Y_{et},\mathbb{Z}))
\Rightarrow H^{p+q}(\mbox{RHom}(D_{\mathcal{X}},\mbox{R}\Gamma(Y_{et},\mathbb{Z})))$$
then shows that 
$$\mathrm{Hom}_{\mathcal{D}}(D_{\mathcal{X}},\mbox{R}\Gamma(Y_{et},\mathbb{Z}))=H^{0}(\mbox{RHom}(D_{\mathcal{X}},\mbox{R}\Gamma(Y_{et},\mathbb{Z})))=0.$$
Hence (\ref{triv-map}) must be the zero map, and the result follows.
\end{proof}

In cases where Theorem \ref{cor-functoriality} and Proposition \ref{prop-functoriality-flat-p} both apply, we have a commutative diagram:
\[ \xymatrix{
\mbox{R}\Gamma_W(\overline{\mathcal{X}},\mathbb{Z})\ar[r]^{\tilde{f}_W^*}\ar[dr]_{f_W^*}& \mbox{R}\Gamma(Y_W,\mathbb{Z})\ar[d]^{\simeq} \\
&\mbox{R}\Gamma_W(Y,\mathbb{Z})
}
\]
where $f_W^*$ is the map defined in Theorem \ref{cor-functoriality}, $\tilde{f}_W^*$ is the map defined in Proposition \ref{prop-functoriality-flat-p} and the vertical isomorphism is defined in Theorem \ref{thm-comparison-char-p}. Indeed, up to the identification given by Theorem \ref{thm-comparison-char-p}, the map $f_W^*$ sits in the commutative square of Proposition \ref{prop-functoriality-flat-p}, hence must coincide with $\tilde{f}_W^*$.

\subsection{Relationship with Lichtenbaum's definition for number rings}
In this section we consider a totally imaginary number field $F$ and we set $\mathcal{X}=\mbox{Spec}(\mathcal{O}_F)$. The complex $\mbox{R}\Gamma_W(\overline{\mathcal{X}},\mathbb{Z})$ is well defined since ${\bf{L}}(\mathcal{X}_{et},\mathbb{Z}(1))$ holds (see Theorem \ref{thm-comparison-mot-K}).
\begin{thm}\label{thm-compare-nbrrings}
There is a canonical isomorphism in $\mathcal{D}$
$$\mbox{\emph{R}}\Gamma_W(\overline{\mathcal{X}},\mathbb{Z})
\stackrel{\sim}{\longrightarrow}\mbox{\emph{R}}\Gamma(\overline{\mathcal{X}}_W,\mathbb{Z})_{\leq3}
$$
where $\mbox{\emph{R}}\Gamma(\overline{\mathcal{X}}_W,\mathbb{Z})_{\leq3}$ is the truncation of Lichtenbaum's complex \cite{Lichtenbaum} and $\mbox{\emph{R}}\Gamma_W(\overline{\mathcal{X}},\mathbb{Z})$ is the complex defined in this paper.
\end{thm}

\begin{proof}
By \cite{On the WE} Theorem 9.5, we have a quasi-isomorphism
$$\mbox{R}\Gamma(\overline{\mathcal{X}}_{et},R\mathbb{Z})\stackrel{\sim}{\longrightarrow} \mbox{R}\Gamma(\overline{\mathcal{X}}_W,\mathbb{Z})$$
inducing
\begin{equation}\label{hehe}
\mbox{R}\Gamma(\overline{\mathcal{X}}_{et},R_W\mathbb{Z})\stackrel{\sim}{\longrightarrow} \mbox{R}\Gamma(\overline{\mathcal{X}}_W,\mathbb{Z})_{\leq3}
\end{equation}
where $R\mathbb{Z}$ is the complex defined in \cite{On the WE} Theorem 8.5 and $R_W\mathbb{Z}:=R\mathbb{Z}_{\leq2}$. The complex $\mbox{R}\Gamma(\overline{\mathcal{X}}_W,\mathbb{Z})$, defined in \cite{Lichtenbaum}, is the cohomology of the Weil-\'etale topos $\overline{\mathcal{X}}_W$ which is defined in \cite{Fund-group-II}. We have an exact triangle
$$\mathbb{Z}[0]\rightarrow R_W\mathbb{Z}\rightarrow R^2_W\mathbb{Z}[-2]$$
in the derived category of \'etale sheaves on $\overline{\mathcal{X}}$.
Rotating and applying $\mbox{R}\Gamma(\overline{\mathcal{X}}_{et},-)$
we get an exact triangle
$$\mbox{R}\Gamma(\overline{\mathcal{X}}_{et},R^2_W\mathbb{Z})[-3]\rightarrow
\mbox{R}\Gamma(\overline{\mathcal{X}}_{et},\mathbb{Z})\rightarrow\mbox{R}\Gamma(\overline{\mathcal{X}}_{et},R_W\bz).$$
We have canonical quasi-isomorphisms (see \cite{On the WE} Theorem 9.4 and \cite{On the WE} Isomorphism (35))
$$\mbox{R}\Gamma(\overline{\mathcal{X}}_{et},R^2_W\mathbb{Z})[-3]\simeq \mbox{Hom}_{\mathbb{Z}}(\mathcal{O}_F^{\times},\mathbb{Q})[-3]
\simeq \mbox{RHom}(\mbox{R}\Gamma(\mathcal{X},\mathbb{Q}(1)),\mathbb{Q}[-4]).$$
It follows that the morphism $\mbox{R}\Gamma(\overline{\mathcal{X}}_{et},R^2_W\mathbb{Z})[-3]\rightarrow
\mbox{R}\Gamma(\overline{\mathcal{X}}_{et},\mathbb{Z})$ in the triangle above is determined by the induced map
$$\mbox{Hom}_{\mathbb{Z}}(\mathcal{O}_F^{\times},\mathbb{Q})=H^0(\overline{\mathcal{X}}_{et},R^2_W\mathbb{Z})\rightarrow
H^3(\overline{\mathcal{X}}_{et},\mathbb{Z})=\mbox{Hom}_{\mathbb{Z}}(\mathcal{O}_F^{\times},\mathbb{Q}/\mathbb{Z})$$
and so is the morphism $\alpha_{\mathcal{X}}$. In both cases this map is the obvious one. Hence the square
\[ \xymatrix{
\mbox{R}\Gamma(\overline{\mathcal{X}}_{et},R^2_W\mathbb{Z})[-3]\ar[r]&\mbox{R}\Gamma(\overline{\mathcal{X}}_{et},\mathbb{Z})\\
 \mbox{RHom}(\mbox{R}\Gamma(\mathcal{X},\mathbb{Q}(1)),\mathbb{Q}[-4])\ar[r]\ar[u]_{\simeq} &\mbox{R}\Gamma(\overline{\mathcal{X}}_{et},\mathbb{Z})\ar[u]_{Id}
}
\]
is commutative. Hence there exists an isomorphism
$\mbox{R}\Gamma_W(\overline{\mathcal{X}},\mathbb{Z})\simeq \mbox{R}\Gamma(\overline{\mathcal{X}}_{et},R_W\mathbb{Z})$. The uniqueness of this isomorphism can be shown as in the proof of Theorem \ref{cor-functoriality}. Composing this isomorphism with (\ref{hehe}), we obtain the result.
\end{proof}

\section{Weil-\'etale Cohomology with compact support}

We recall below the definition given in \cite{Flach-moi} of the Weil-\'etale topos and some results concerning its cohomology with $\tilde{\mathbb{R}}$-coefficients. Then we define Weil-\'etale cohomology with compact support and $\mathbb{Z}$-coefficients and we study the expected map from $\mathbb{Z}$ to $\mathbb{R}$-coefficients.

\subsection{Cohomology with $\tilde{\mathbb{R}}$-coefficients}\label{sect-Rcoefs}

 Let $\mathcal{X}$ be any proper regular connected arithmetic scheme. The Weil-\'etale topos is defined as a 2-fiber product of topoi $$\overline{\mathcal{X}}_W:=\overline{\mathcal{X}}_{et}\times_{\overline{\mbox{Spec}(\mathbb{Z})}_{et}}
\overline{\mbox{Spec}(\mathbb{Z})}_{W}.$$
There is a canonical morphism
$$\mathfrak{f}:\overline{\mathcal{X}}_W\rightarrow\overline{\mbox{Spec}(\mathbb{Z})}_{W}\rightarrow B_{\mathbb{R}}$$
where $B_{\mathbb{R}}$ is Grothendieck's classifying topos of the topological group $\mathbb{R}$ (see \cite{SGA4} or \cite{Flach-moi}). Consider the sheaf $y\mathbb{R}$ on $B_{\mathbb{R}}$ represented by $\mathbb{R}$ with the standard topology and trivial $\mathbb{R}$-action. Then one defines the sheaf $$\tilde{\mathbb{R}}:=\mathfrak{f}^*(y\mathbb{R})$$ on $\overline{\mathcal{X}}_W$. By \cite{Flach-moi}, the following diagram consists of two pull-back squares of topoi, and the rows give open-closed decompositions:
\[ \xymatrix{
\mathcal{X}_W\ar[d]\ar[r]^{\phi}&\overline{\mathcal{X}}_W\ar[d]& \mathcal{X}_{\infty,W}\ar[d]\ar[l]_{i_{\infty}}\\
\mathcal{X}_{et}\ar[r]^{\varphi}&\overline{\mathcal{X}}_{et}& Sh(\mathcal{X}_{\infty})\ar[l]_{u_{\infty}}
}
\]
Here the map
$$\overline{\mathcal{X}}_W:=\overline{\mathcal{X}}_{et}\times_{\overline{\mbox{Spec}(\mathbb{Z})}_{et}}
\overline{\mbox{Spec}(\mathbb{Z})}_{W}\longrightarrow\overline{\mathcal{X}}_{et}$$
is the first projection,  $Sh(\mathcal{X}_{\infty})$ is the category of sheaves on the space $\mathcal{X}_{\infty}$ and $$\mathcal{X}_{\infty,W}=B_{\mathbb{R}}\times Sh(\mathcal{X}_{\infty})$$
where the product is taken over the final topos. As shown in \cite{Flach-moi}, the topos $\overline{\mathcal{X}}_W$ has the right $\tilde{\mathbb{R}}$-cohomology with and without compact supports. We have
$$\mbox{R}\Gamma_{W}(\overline{\mathcal{X}},\tilde{\mathbb{R}}):=\mbox{R}\Gamma(\overline{\mathcal{X}}_{W},\tilde{\mathbb{R}})\simeq \mbox{R}\Gamma(B_{\mathbb{R}},\tilde{\mathbb{R}})\simeq\mathbb{R}[-1]\oplus \mathbb{R}.$$
Concerning the cohomology with compact support, one has
\begin{equation}\label{cohomology-R-cpct-decompo}
\mbox{R}\Gamma_{W,c}(\mathcal{X},\tilde{\mathbb{R}}):=\mbox{R}\Gamma(\overline{\mathcal{X}}_{W},\phi_!\tilde{\mathbb{R}})\simeq
\mbox{R}\Gamma(\overline{\mathcal{X}}_{et},\varphi_!\mathbb{R})[-1]\oplus \mbox{R}\Gamma(\overline{\mathcal{X}}_{et},\varphi_!\mathbb{R})
\end{equation}
where the complex $\mbox{R}\Gamma_c(\mathcal{X}_{et},\mathbb{R}):=\mbox{R}\Gamma(\overline{\mathcal{X}}_{et},\varphi_!\mathbb{R})$ is quasi-isomorphic to
$$\mbox{Cone}(\mathbb{R}[0]\rightarrow\mbox{R}\Gamma(\mathcal{X}_{\infty},\mathbb{R}))[-1].$$
Cup-product with the fundamental class $\theta\in H^1(\mathcal{X}_W,\tilde{\mathbb{R}})$  yields a morphism
\begin{equation}\label{mapcup}
\cup\theta:\mbox{R}\Gamma_{W,c}(\mathcal{X},\tilde{\mathbb{R}})\rightarrow \mbox{R}\Gamma_{W,c}(\mathcal{X},\tilde{\mathbb{R}})[1]
\end{equation}
such that the induced sequence
$$...\rightarrow H_{W,c}^{i-1}(\mathcal{X},\tilde{\mathbb{R}})\rightarrow H_{W,c}^{i}(\mathcal{X},\tilde{\mathbb{R}})\rightarrow H_{W,c}^{i+1}(\mathcal{X},\tilde{\mathbb{R}})\rightarrow$$
is a bounded acyclic complex of finite dimensional $\mathbb{R}$-vector spaces. In view of
$\mbox{R}\Gamma_{W,c}(\mathcal{X},\tilde{\mathbb{R}})[1]= \mbox{R}\Gamma(\overline{\mathcal{X}}_{et},\varphi_!\mathbb{R})\oplus \mbox{R}\Gamma(\overline{\mathcal{X}}_{et},\varphi_!\mathbb{R})[1]$,
the map (\ref{mapcup}) is simply given by projection and inclusion 
$$\mbox{R}\Gamma_{W,c}(\mathcal{X},\tilde{\mathbb{R}})\twoheadrightarrow \mbox{R}\Gamma(\overline{\mathcal{X}}_{et},\varphi_!\mathbb{R})\hookrightarrow
\mbox{R}\Gamma_{W,c}(\mathcal{X},\tilde{\mathbb{R}})[1].$$

\subsection{Cohomology with $\mathbb{Z}$-coefficients}
In the remaining part of this section, $\mathcal{X}$ denotes a proper regular connected arithmetic scheme of dimension $d$ satisfying $\textbf{L}(\mathcal{X}_{et},d)_{\geq0}$. If $\mathcal{X}$ has characteristic $p$ then we set $\mbox{R}\Gamma_{W,c}(\mathcal{X},\mathbb{Z}):=\mbox{R}\Gamma_{W}(\mathcal{X},\mathbb{Z})$
and the cohomology with compact support is defined as $H_{W,c}^i(\mathcal{X},\mathbb{Z}):=H_{W}^i(\mathcal{X},\mathbb{Z})$.
The case when $\mathcal{X}$ is flat over $\mathbb{Z}$ is the case of interest. The closed embedding $u_{\infty}:Sh(\mathcal{X}_{\infty})\rightarrow\mathcal{X}_{et}$
induces a morphism $u_{\infty}^*:\mbox{R}\Gamma(\overline{\mathcal{X}}_{et},\mathbb{Z})\rightarrow \mbox{R}\Gamma(\mathcal{X}_{\infty},\mathbb{Z})$.
\begin{prop}\label{prop-iinfty}
There exists a unique morphism
$i^*_{\infty}:\mbox{\emph{R}}\Gamma_W(\overline{\mathcal{X}},\mathbb{Z})\rightarrow \mbox{\emph{R}}\Gamma(\mathcal{X}_{\infty,W},\mathbb{Z})$ which makes the following diagram commutative:
\[ \xymatrix{
\mbox{\emph{RHom}}(\mbox{\emph{R}}\Gamma(\mathcal{X},\mathbb{Q}(d))_{\geq0},\mathbb{Q}[-2d-2])\ar[d]\ar[r]& \mbox{\emph{R}}\Gamma(\overline{\mathcal{X}}_{et},\mathbb{Z})\ar[d]^{u^*_{\infty}}\ar[r]& \mbox{\emph{R}}\Gamma_W(\overline{\mathcal{X}},\mathbb{Z})\ar[d]_{\exists\,!}^{i^*_{\infty}}\\
0\ar[r]&\mbox{\emph{R}}\Gamma(\mathcal{X}_{\infty},\mathbb{Z})\ar[r]&\mbox{\emph{R}}\Gamma(\mathcal{X}_{\infty,W},\mathbb{Z})
}
\]
\end{prop}
\begin{proof}
Recall from \cite{Flach-moi} that the second projection $$\mathcal{X}_{\infty,W}:=B_{\mathbb{R}}\times Sh(\mathcal{X}_{\infty})\rightarrow Sh(\mathcal{X}_{\infty})=\mathcal{X}_{\infty,et}$$
induces a quasi-isomorphism $\mbox{R}\Gamma(\mathcal{X}_{\infty},\mathbb{Z})\stackrel{\sim}{\longrightarrow}\mbox{R}\Gamma(\mathcal{X}_{\infty,W},\mathbb{Z})$.
Hence the existence of the map $i^*_{\infty}$ will follow (Axiom TR3 of triangulated categories) from the fact that the map
\begin{equation}\label{zero-map}
\beta:\mbox{RHom}(\mbox{R}\Gamma(\mathcal{X},\mathbb{Q}(d))_{\geq0},\mathbb{Q}[-2d-2])
\stackrel{\alpha_{\X}}{\longrightarrow }\mbox{R}\Gamma(\overline{\mathcal{X}}_{et},\mathbb{Z})\stackrel{u_{\infty}^*}{\longrightarrow }\mbox{R}\Gamma(\mathcal{X}_{\infty},\mathbb{Z})
\end{equation}
is the zero map. Again, we set $D_{\mathcal{X}}=\mbox{RHom}(\mbox{R}\Gamma(\mathcal{X},\mathbb{Q}(d))_{\geq0},\mathbb{Q}[-2d-2])$ for brevity. Then the uniqueness of $i_{\infty}^*$ follows from the exact sequence
$$\mbox{Hom}_{\mathcal{D}}(D_{\mathcal{X}}[1],\mbox{R}\Gamma(\mathcal{X}_{\infty,W},\mathbb{Z}))\stackrel{0}{\rightarrow}
\mbox{Hom}_{\mathcal{D}}(\mbox{R}\Gamma_W(\overline{\mathcal{X}},\mathbb{Z}),\mbox{R}\Gamma(\mathcal{X}_{\infty,W},\mathbb{Z}))$$
$$\rightarrow
\mbox{Hom}_{\mathcal{D}}(\mbox{R}\Gamma(\overline{\mathcal{X}}_{et},\mathbb{Z}),\mbox{R}\Gamma(\mathcal{X}_{\infty,W},\mathbb{Z})),$$
whose exactness follows from the fact that $\mbox{Hom}_{\mathcal{D}}(D_{\mathcal{X}}[1],\mbox{R}\Gamma(\mathcal{X}_{\infty,W},\mathbb{Z}))$ is divisible while
$\mbox{Hom}_{\mathcal{D}}(\mbox{R}\Gamma_W(\overline{\mathcal{X}},\mathbb{Z}),\mbox{R}\Gamma(\mathcal{X}_{\infty,W},\mathbb{Z}))$ is finitely generated.

It remains to show that the morphism (\ref{zero-map}) is indeed trivial. Since $D_{\mathcal{X}}$ is a bounded complex of $\mathbb{Q}$-vector spaces acyclic in degrees $<2$, one can choose a (non-canonical) isomorphism
$D_{\mathcal{X}}\simeq \bigoplus_{k\geq2}H^k(D_{\mathcal{X}})[-k]$.
Then $\beta$ is identified with the collection of maps $(\beta_k)_{k\geq2}$ in $\mathcal{D}$, with
$$\beta_k:H^k(D_{\mathcal{X}})[-k]\rightarrow\bigoplus_{k\geq2}H^k(D_{\mathcal{X}})[-k]\simeq D_{\mathcal{X}}\rightarrow
\mbox{R}\Gamma(\mathcal{X}_{\infty},\mathbb{Z}).$$
It is enough to show that $\beta_k=0$ for $k\geq2$. 
We fix such a $k$, and we consider the spectral sequence
$$E_2^{p,q}=\mbox{Ext}^p(H^kD_{\mathcal{X}},H^{q+k}(\mathcal{X}_{\infty},\mathbb{Z}))
\Rightarrow H^{p+q}(\mbox{RHom}(H^k(D_{\mathcal{X}})[-k],\mbox{R}\Gamma(\mathcal{X}_{\infty},\mathbb{Z}))).$$
The group $H^kD_{\mathcal{X}}$ is uniquely divisible and $H^{q+k}(\mathcal{X}_{\infty},\mathbb{Z})$ is finitely generated (hence has an injective resolution of length 1), so that $\mbox{Ext}^p(H^k(D_{\mathcal{X}}),H^{q+k}(\mathcal{X}_{\infty},\mathbb{Z}))=0$ for $p\neq1$. Hence the spectral sequence above degenerates and gives a canonical isomorphism
$$\mbox{Hom}_{\mathcal{D}}(H^kD_{\mathcal{X}}[-k],\mbox{R}\Gamma(\mathcal{X}_{\infty},\mathbb{Z}))\simeq
\mbox{Ext}^1(H^kD_{\mathcal{X}},H^{k-1}(\mathcal{X}_{\infty},\mathbb{Z})).$$
Moreover, the long exact sequence for $\mbox{Ext}^*(H^kD_{\mathcal{X}},-)$ yields
$$\mbox{Ext}^1(H^kD_{\mathcal{X}},H^{k-1}(\mathcal{X}_{\infty},\mathbb{Z}))\simeq \mbox{Ext}^1(H^kD_{\mathcal{X}},H^{k-1}(\mathcal{X}_{\infty},\mathbb{Z})_{cotor})$$
since the maximal torsion subgroup of $H^{k-1}(\mathcal{X}_{\infty},\mathbb{Z})$ is finite and $H^kD_{\mathcal{X}}$ is uniquely divisible.
The short exact sequence
$$0\rightarrow H^{k-1}(\mathcal{X}_{\infty},\mathbb{Z})_{cotor}\rightarrow H^{k-1}(\mathcal{X}_{\infty},\mathbb{Q})\rightarrow H^{k-1}(\mathcal{X}_{\infty},\mathbb{Q}/\mathbb{Z})_{div}\rightarrow0$$
is an injective resolution of the $\mathbb{Z}$-module $H^{k-1}(\mathcal{X}_{\infty},\mathbb{Z})_{cotor}$.
This yields an exact sequence
$$0\rightarrow\mbox{Hom}(H^kD_{\mathcal{X}},H^{k-1}(\mathcal{X}_{\infty},\mathbb{Q}))\rightarrow
\mbox{Hom}(H^kD_{\mathcal{X}},H^{k-1}(\mathcal{X}_{\infty},\mathbb{Q}/\mathbb{Z})_{div})$$
$$\rightarrow\mbox{Ext}^1(H^kD_{\mathcal{X}},H^{k-1}(\mathcal{X}_{\infty},\mathbb{Z})_{cotor})\rightarrow 0.$$
Let us define a natural lifting $$\tilde{\beta}_k\in \mbox{Hom}(H^kD_{\mathcal{X}},H^{k-1}(\mathcal{X}_{\infty},\mathbb{Q}/\mathbb{Z})_{div})$$ of $\beta_k\in \mbox{Ext}^1(H^kD_{\mathcal{X}},H^{k-1}(\mathcal{X}_{\infty},\mathbb{Z})_{cotor})$ and show that this lifting $\tilde{\beta}_k$ is already zero. One can assume $k\geq2$. Recall that $H^kD_{\mathcal{X}}=\mbox{Hom}(H^{\delta-k}(\mathcal{X},\mathbb{Q}(d))_{\geq0},\mathbb{Q})$. We have the following commutative diagram:
\[ \xymatrix{
 &\mbox{Hom}(H^{\delta-k}(\mathcal{X},\mathbb{Q}(d))_{\geq0},\mathbb{Q})\ar[d]\ar[dr]^{H^k(\alpha_{\mathcal{X}})}&\\
&H^{k-1}(\overline{\mathcal{X}}_{et},\mathbb{Q}/\mathbb{Z})\ar[d]\ar[r]^{\simeq}&H^k(\overline{\mathcal{X}}_{et},\mathbb{Z})\ar[d]\\
&H^{k-1}(\mathcal{X}_{\mathbb{Q},\,et},\mathbb{Q}/\mathbb{Z})\ar[d]\ar[r]^{\simeq}
&H^k(\mathcal{X}_{\mathbb{Q},\,et},\mathbb{Z})\ar[d]\\
&H^{k-1}(\mathcal{X}_{\overline{\mathbb{Q}},\,et},\mathbb{Q}/\mathbb{Z})\ar[d]^{\simeq}\ar[r]^{\simeq}
&H^k(\mathcal{X}_{\overline{\mathbb{Q}},\,et},\mathbb{Z})\ar[d]\\
H^{k-1}(\mathcal{X}(\mathbb{C}),\mathbb{Q})\ar[r]
&H^{k-1}(\mathcal{X}(\mathbb{C}),\mathbb{Q}/\mathbb{Z})\ar[r]&H^k(\mathcal{X}(\mathbb{C}),\mathbb{Z})
}
\]
The morphism given by the central column of the diagram above factors through
$$H^{k-1}(\mathcal{X}_{\infty},\mathbb{Q}/\mathbb{Z})_{div}\subseteq H^{k-1}(\mathcal{X}_{\infty},\mathbb{Q}/\mathbb{Z})\rightarrow H^{k-1}(\mathcal{X}(\mathbb{C}),\mathbb{Q}/\mathbb{Z})$$
and yields the desired lifting
$$\tilde{\beta}_k:\mbox{Hom}(H^{\delta-k}(\mathcal{X},\mathbb{Q}(d))_{\geq0},\mathbb{Q})\rightarrow H^{k-1}(\mathcal{X}_{\infty},\mathbb{Q}/\mathbb{Z})_{div}.$$
Here the morphism 
\begin{equation}\label{one++map}
H^{k-1}(\mathcal{X}_{\infty},\mathbb{Q}/\mathbb{Z})\rightarrow H^{k-1}(\mathcal{X}(\mathbb{C}),\mathbb{Q}/\mathbb{Z})
\end{equation} 
is induced by the projection $\mathcal{X}(\mathbb{C})\rightarrow \mathcal{X}_{\infty}$. Using standard spectral sequences for equivariant cohomology, it is easy to see that the kernel of the morphism (\ref{one++map}) is of finite exponent (more precisely, this kernel is finite and killed by a power of $2$). It follows that
$\tilde{\beta}_k=0$ if and only if the map
\begin{equation}\label{one---map}
\mbox{Hom}(H^{\delta-k}(\mathcal{X},\mathbb{Q}(d))_{\geq0},\mathbb{Q})\rightarrow H^{k-1}(\mathcal{X}(\mathbb{C}),\mathbb{Q}/\mathbb{Z}),
\end{equation}
given by the central column of the previous diagram, is the zero map. Moreover, the map $$H^{k-1}(\mathcal{X}_{\mathbb{Q},\,et},\mathbb{Q}/\mathbb{Z})\rightarrow
H^{k-1}(\mathcal{X}_{\overline{\mathbb{Q}},\,et},\mathbb{Q}/\mathbb{Z})$$
factors through $H^{k-1}(\mathcal{X}_{\overline{\mathbb{Q}},\,et},\mathbb{Q}/\mathbb{Z})^{G_{\mathbb{Q}}}$ hence so does
the map (\ref{one---map}). In order to show that $\tilde{\beta}_k=0$ it is therefore enough to show that
$$(H^{k-1}(\mathcal{X}_{\overline{\mathbb{Q}},\,et},\mathbb{Q}/\mathbb{Z})^{G_{\mathbb{Q}}})_{div}=
\bigoplus_{l}(H^{k-1}(\mathcal{X}_{\overline{\mathbb{Q}},\,et},\mathbb{Q}_l/\mathbb{Z}_l)^{G_{\mathbb{Q}}})_{div}=0.$$
Let $l$ be a fixed prime number. Let $U\subseteq \Spec(\mathbb{Z})$ on which $l$ is invertible and such that $\mathcal{X}_U\rightarrow U$ is smooth. Let $p\in U$. By smooth and proper base change we have:
$$H^{k-1}(\mathcal{X}_{\overline{\mathbb{Q}},\,et},\mathbb{Q}_l/\mathbb{Z}_l)^{I_{p}}\simeq H^{k-1}(\mathcal{X}_{\overline{\mathbb{F}}_p,\,et},\mathbb{Q}_l/\mathbb{Z}_l).$$
Recall that $H^{k-1}(\mathcal{X}_{\overline{\mathbb{F}}_p,\,et},\mathbb{Z}_l)$ is a finitely generated
$\mathbb{Z}_l$-module. We have an exact sequence
$$0\rightarrow H^{k-1}(\mathcal{X}_{\overline{\mathbb{F}}_p,\,et},\mathbb{Z}_l)_{cotor}\rightarrow H^{k-1}(\mathcal{X}_{\overline{\mathbb{F}}_p,\,et},\mathbb{Q}_l)\rightarrow
H^{k-1}(\mathcal{X}_{\overline{\mathbb{F}}_p,\,et},\mathbb{Q}_l/\mathbb{Z}_l)_{div}\rightarrow 0.$$
We get
$$0\rightarrow (H^{k-1}(\mathcal{X}_{\overline{\mathbb{F}}_p,\,et},\mathbb{Z}_l)_{cotor})^{G_{\mathbb{F}_p}}\rightarrow H^{k-1}(\mathcal{X}_{\overline{\mathbb{F}}_p,\,et},\mathbb{Q}_l)^{G_{\mathbb{F}_p}}$$
$$\rightarrow
(H^{k-1}(\mathcal{X}_{\overline{\mathbb{F}}_p,\,et},\mathbb{Q}_l/\mathbb{Z}_l)_{div})^{G_{\mathbb{F}_p}}
\rightarrow H^1(G_{\mathbb{F}_p},H^{k-1}(\mathcal{X}_{\overline{\mathbb{F}}_p,\,et},\mathbb{Z}_l)_{cotor}).$$
Again, $ H^1(G_{\mathbb{F}_p},H^{k-1}(\mathcal{X}_{\overline{\mathbb{F}}_p,\,et},\mathbb{Z}_l)_{cotor})$ is a finitely generated
$\mathbb{Z}_l$-module, hence we get a surjective map
$$H^{k-1}(\mathcal{X}_{\overline{\mathbb{F}}_p,\,et},\mathbb{Q}_l)^{G_{\mathbb{F}_p}}\rightarrow
((H^{k-1}(\mathcal{X}_{\overline{\mathbb{F}}_p,\,et},\mathbb{Q}_l/\mathbb{Z}_l)_{div})^{G_{\mathbb{F}_p}})_{div}
\rightarrow0.$$
Note that
$$((H^{k-1}(\mathcal{X}_{\overline{\mathbb{F}}_p,\,et},\mathbb{Q}_l/\mathbb{Z}_l)_{div})^{G_{\mathbb{F}_p}})_{div}
= (H^{k-1}(\mathcal{X}_{\overline{\mathbb{F}}_p,\,et},\mathbb{Q}_l/\mathbb{Z}_l)^{G_{\mathbb{F}_p}})_{div}.$$
But  $H^{k-1}(\mathcal{X}_{\overline{\mathbb{F}}_p,\,et},\mathbb{Q}_l)$ is pure of weight $k-1>0$ by \cite{Deligne74}, hence there is no non-trivial element in $H^{k-1}(\mathcal{X}_{\overline{\mathbb{F}}_p,\,et},\mathbb{Q}_l)$ fixed by the Frobenius. This shows that
$$
(H^{k-1}(\mathcal{X}_{\overline{\mathbb{Q}},\,et},\mathbb{Q}_l/\mathbb{Z}_l)^{G_{\mathbb{Q}_p}})_{div}
=(H^{k-1}(\mathcal{X}_{\overline{\mathbb{F}}_p,\,et},\mathbb{Q}_l/\mathbb{Z}_l)^{G_{\mathbb{F}_p}})_{div}
=H^{k-1}(\mathcal{X}_{\overline{\mathbb{F}}_p,\,et},\mathbb{Q}_l)^{G_{\mathbb{F}_p}}=0.$$
A fortiori, one has $(H^{k-1}(\mathcal{X}_{\overline{\mathbb{Q}},\,et},\mathbb{Q}_l/\mathbb{Z}_l)^{G_{\mathbb{Q}}})_{div}=0$
and the result follows.

\end{proof}

\begin{defn}\label{defncpctsuppcoh}
There exists an object $\mbox{\emph{R}}\Gamma_{W,c}(\mathcal{X},\mathbb{Z})$, well defined up to isomorphism in $\mathcal{D}$, \emph{endowed with} an exact triangle
\begin{equation}\label{exact-triangleforcompact}
\mbox{\emph{R}}\Gamma_{W,c}(\mathcal{X},\mathbb{Z})\rightarrow \mbox{\emph{R}}\Gamma_W(\overline{\mathcal{X}},\mathbb{Z})\stackrel{i_{\infty}^*}{\rightarrow} \mbox{\emph{R}}\Gamma(\mathcal{X}_{\infty,W},\mathbb{Z}).
\end{equation}
The determinant $\mbox{\emph{det}}_{\mathbb{Z}} \mbox{\emph{R}}\Gamma_{W,c}(\mathcal{X},\mathbb{Z}):=\bigotimes_{i\in\mathbb{Z}}
\mbox{\emph{det}}_{\mathbb{Z}} H^i_{W,c}(\mathcal{X},\mathbb{Z})^{(-1)^i}$ is well defined up to a \emph{canonical} isomorphism.
\end{defn}
The cohomology with compact support is defined (up to isomorphism only) as follows:
$$H_{W,c}^i(\mathcal{X},\mathbb{Z}):=H^i(\mbox{R}\Gamma_{W,c}(\mathcal{X},\mathbb{Z})).$$
To see that $\mbox{det}_{\mathbb{Z}} \mbox{R}\Gamma_{W,c}(\mathcal{X},\mathbb{Z})$ is indeed well defined, consider another object $\mbox{R}\Gamma_{W,c}(\mathcal{X},\mathbb{Z})'$ of $\mathcal{D}$ endowed with an exact triangle (\ref{exact-triangleforcompact}). There exists a (non-unique) morphism $u: \mbox{R}\Gamma_{W,c}(\mathcal{X},\mathbb{Z})\rightarrow\mbox{R}\Gamma_{W,c}(\mathcal{X},\mathbb{Z})'$ lying in a morphism of exact triangles
\[ \xymatrix{
\mbox{R}\Gamma_{W,c}(\mathcal{X},\mathbb{Z})\ar[d]_{\exists\,u}^{\simeq}\ar[r]
&\mbox{R}\Gamma_W(\overline{\mathcal{X}},\mathbb{Z})\ar[d]_{Id}\ar[r]
&\mbox{R}\Gamma(\mathcal{X}_{\infty,W},\mathbb{Z})\ar[d]_{Id}\\
\mbox{R}\Gamma_{W,c}(\mathcal{X},\mathbb{Z})'\ar[r]
&\mbox{R}\Gamma_W(\overline{\mathcal{X}},\mathbb{Z})\ar[r]
&\mbox{R}\Gamma(\mathcal{X}_{\infty,W},\mathbb{Z})
}
\]
The map $u$ induces
$$\mbox{det}_{\mathbb{Z}}(u):\mbox{det}_{\mathbb{Z}} \mbox{R}\Gamma_{W,c}(\mathcal{X},\mathbb{Z})\stackrel{\sim}{\longrightarrow}\mbox{det}_{\mathbb{Z}} \mbox{R}\Gamma_{W,c}(\mathcal{X},\mathbb{Z})'.$$
By \cite{Knudsen-Mumford} p. 43 Corollary 2, $\mbox{det}_{\mathbb{Z}}(u)$ does not depend on the choice of $u$, since it coincides with the following canonical isomorphism
\begin{eqnarray*}
\mbox{det}_{\mathbb{Z}} \mbox{R}\Gamma_{W,c}(\mathcal{X},\mathbb{Z})&\simeq&
\mbox{det}_{\mathbb{Z}} \mbox{R}\Gamma_{W}(\overline{\mathcal{X}},\mathbb{Z})\otimes \mbox{det}^{-1}_{\mathbb{Z}} \mbox{R}\Gamma(\mathcal{X}_{\infty,W},\mathbb{Z})\\
&\simeq&\mbox{det}_{\mathbb{Z}} \mbox{R}\Gamma_{W,c}(\mathcal{X},\mathbb{Z})'.
\end{eqnarray*}
Given a complex  of abelian groups $C$, we write $C_{\mathbb{R}}$ for $C\otimes\mathbb{R}$.

\begin{prop}\label{propiinftydecompos}
We set $\delta:=2d+2$. There is a \emph{canonical and functorial} direct sum decomposition in $\mathcal{D}$:
$$\mbox{\emph{R}}\Gamma_W(\overline{\mathcal{X}},\mathbb{Z})_{\mathbb{R}}\simeq \mbox{\emph{R}}\Gamma(\overline{\mathcal{X}}_{et},\mathbb{R})\oplus \mbox{\emph{RHom}}(\mbox{\emph{R}}\Gamma(\mathcal{X},\mathbb{Q}(d))_{\geq0},\mathbb{R}[-\delta])[1]$$
such that the following square commutes:
\[ \xymatrix{
\mbox{\emph{R}}\Gamma_W(\overline{\mathcal{X}},\mathbb{Z})_{\mathbb{R}}\ar[d]_{\simeq}\ar[r]^{i^*_{\infty}\otimes\mathbb{R}}& \mbox{\emph{R}}\Gamma(\mathcal{X}_{\infty,W},\mathbb{Z})_{\mathbb{R}}\ar[d]_{\simeq}\\
\mbox{\emph{R}}\Gamma(\overline{\mathcal{X}}_{et},\mathbb{R})\oplus \mbox{\emph{RHom}}(\mbox{\emph{R}}\Gamma(\mathcal{X},\mathbb{Q}(d))_{\geq0},\mathbb{R}[-\delta])[1]\ar[r]^{\hspace{3cm}(u^*_{\infty}\otimes \mathbb{R},0)}
&\mbox{\emph{R}}\Gamma(\mathcal{X}_{\infty},\mathbb{R})
}
\]

\end{prop}
\begin{proof}
Applying $(-)\otimes\mathbb{R}$ to the exact triangle of Definition \ref{def-cohomology}, we obtain an exact triangle
$$\mbox{RHom}(\mbox{R}\Gamma(\mathcal{X},\mathbb{Q}(d))_{\geq0},\mathbb{R}[-\delta])\rightarrow
\mbox{R}\Gamma(\overline{\mathcal{X}}_{et},\mathbb{R})
\rightarrow\mbox{R}\Gamma_W(\overline{\mathcal{X}},\mathbb{Z})_{\mathbb{R}}.$$
But the map $$\mbox{RHom}(\mbox{R}\Gamma(\mathcal{X},\mathbb{Q}(d))_{\geq0},\mathbb{R}[-\delta])\rightarrow
\mbox{R}\Gamma(\overline{\mathcal{X}}_{et},\mathbb{R})$$
is trivial since $\mbox{R}\Gamma(\overline{\mathcal{X}}_{et},\mathbb{R})\simeq\mathbb{R}[0]$ is injective and $\mbox{RHom}(\mbox{R}\Gamma(\mathcal{X},\mathbb{Q}(d))_{\geq0},\mathbb{R}[-\delta])$ is acyclic in degrees $\leq1$. This shows the existence of the direct sum decomposition. We write $D_{\mathbb{R}}:=\mbox{RHom}(\mbox{R}\Gamma(\mathcal{X},\mathbb{Q}(d))_{\geq0},\mathbb{R}[-\delta])$ and $\mbox{R}\Gamma(\overline{\mathcal{X}}_{et},\mathbb{R})\simeq\mathbb{R}[0]$ for brevity. The exact sequence
$$\mbox{Hom}_{\mathcal{D}}(D_{\mathbb{R}}[1],\mathbb{R}[0])\rightarrow \mbox{Hom}_{\mathcal{D}}(\mbox{R}\Gamma_W(\overline{\mathcal{X}},\mathbb{Z}),\mathbb{R}[0])$$
$$\rightarrow \mbox{Hom}_{\mathcal{D}}(\mathbb{R}[0],\mathbb{R}[0])\rightarrow \mbox{Hom}_{\mathcal{D}}(D_{\mathbb{R}},\mathbb{R}[0])$$
yields an isomorphism $\mbox{Hom}_{\mathcal{D}}(\mbox{R}\Gamma_W(\overline{\mathcal{X}},\mathbb{Z}),\mathbb{R}[0])\stackrel{\sim}{\rightarrow}
\mbox{Hom}_{\mathcal{D}}(\mathbb{R}[0],\mathbb{R}[0])$. Hence there exists a unique map $s_{\overline{\mathcal{X}}}:\mbox{R}\Gamma_W(\overline{\mathcal{X}},\mathbb{Z})_{\mathbb{R}}\rightarrow \mathbb{R}[0]$
such that $s_{\overline{\mathcal{X}}}\circ \gamma^*_{\overline{\mathcal{X}}}=\mbox{Id}_{\mathbb{R}[0]}$ where $\gamma^*_{\overline{\mathcal{X}}}:\mathbb{R}[0]\rightarrow \mbox{R}\Gamma_W(\overline{\mathcal{X}},\mathbb{Z})$ is the given map. The functorial behavior of $s_{\overline{\mathcal{X}}}$ follows from the fact that it is the unique map such that $s_{\overline{\mathcal{X}}}\circ \gamma^*_{\overline{\mathcal{X}}}=\mbox{Id}_{\mathbb{R}[0]}$. The direct sum decomposition is therefore canonical and functorial. The commutativity of the square follows from  Proposition \ref{prop-iinfty}.
\end{proof}
Recall that we denote $\mbox{R}\Gamma_{c}(\mathcal{X}_{et},\mathbb{R}):=\mbox{R}\Gamma(\overline{\mathcal{X}}_{et},\varphi_!\mathbb{R})$.
\begin{prop}\label{lapropcanonique}
We set $\delta:=2d+2$. There is a \emph{non-canonical} direct sum decomposition
\begin{equation}\label{directsumdecompo}
\mbox{\emph{R}}\Gamma_{W,c}(\mathcal{X},\mathbb{Z})_{\mathbb{R}}\simeq \mbox{\emph{R}}\Gamma_{c}(\mathcal{X}_{et},\mathbb{R})\oplus \mbox{\emph{RHom}}(\mbox{\emph{R}}\Gamma(\mathcal{X},\mathbb{Q}(d))_{\geq0},\mathbb{R}[-\delta])[1]
\end{equation}
inducing a \emph{canonical} isomorphism
$$\mbox{\emph{det}}_{\mathbb{R}} \mbox{\emph{R}}\Gamma_{W,c}(\mathcal{X},\mathbb{Z})_{\mathbb{R}}\simeq
\mbox{\emph{det}}_{\mathbb{R}}\mbox{\emph{R}}\Gamma_{c}(\mathcal{X}_{et},\mathbb{R})\otimes \mbox{\emph{det}}^{-1}_{\mathbb{R}}
\mbox{\emph{RHom}}(\mbox{\emph{R}}\Gamma(\mathcal{X},\mathbb{Q}(d))_{\geq0},\mathbb{R}[-\delta]).$$
\end{prop}
\begin{proof}
Again we write $D_{\mathbb{R}}:=\mbox{RHom}(\mbox{R}\Gamma(\mathcal{X},\mathbb{Q}(d))_{\geq0},\mathbb{R}[-\delta])$ and $\mbox{R}\Gamma(\mathcal{X}_{et},\mathbb{R})\simeq\mathbb{R}[0]$ for brevity (recall that $\mathcal{X}$ is connected).
Consider the following morphism of exact triangles:
\[ \xymatrix{
 \mbox{R}\Gamma_{W,c}(\mathcal{X},\mathbb{Z})_{\mathbb{R}}\ar[r]\ar[d]_{\exists\,u}
 &\mbox{R}\Gamma_W(\overline{\mathcal{X}},\mathbb{Z})_{\mathbb{R}}\ar[d]_{\simeq}\ar[r]^{i^*_{\infty}\otimes\mathbb{R}}& \mbox{R}\Gamma(\mathcal{X}_{\infty,W},\mathbb{Z})_{\mathbb{R}}\ar[d]_{\simeq}\\
\mbox{R}\Gamma_c(\mathcal{X}_{et},\mathbb{R})\oplus D_{\mathbb{R}}[1] \ar[r]&\mbox{R}\Gamma(\overline{\mathcal{X}}_{et},\mathbb{R})\oplus D_{\mathbb{R}}[1]\ar[r]
&\mbox{R}\Gamma(\mathcal{X}_{\infty},\mathbb{R})
}
\]
Here all the maps but the isomorphism $u$ are canonical. The non-canonical direct sum decomposition (\ref{directsumdecompo}) follows. A choice of such an isomorphism $u$ induces
$$\mbox{det}_{\mathbb{R}}(u):\mbox{det}_{\mathbb{R}} \mbox{R}\Gamma_{W,c}(\mathcal{X},\mathbb{Z})_{\mathbb{R}}
\stackrel{\sim}{\longrightarrow} \mbox{det}_{\mathbb{R}} (\mbox{R}\Gamma_c(\mathcal{X}_{et},\mathbb{R})\oplus D_{\mathbb{R}}[1]).$$
But $\mbox{det}_{\mathbb{R}}(u)$  coincides (see \cite{Knudsen-Mumford} p. 43 Corollary 2) with the following (canonical) isomorphism
\begin{eqnarray*}
\mbox{det}_{\mathbb{R}} \mbox{R}\Gamma_{W,c}(\mathcal{X},\mathbb{Z})_{\mathbb{R}}&\simeq&
\mbox{det}_{\mathbb{R}} \mbox{R}\Gamma_{W}(\overline{\mathcal{X}},\mathbb{Z})_{\mathbb{R}}\otimes \mbox{det}^{-1}_{\mathbb{R}} \mbox{R}\Gamma(\mathcal{X}_{\infty,W},\mathbb{Z})_{\mathbb{R}}\\
&\simeq& \mbox{det}_{\mathbb{R}} (\mbox{R}\Gamma(\overline{\mathcal{X}}_{et},\mathbb{R})\oplus D_{\mathbb{R}}[1])\otimes \mbox{det}^{-1}_{\mathbb{R}}\mbox{R}\Gamma(\overline{\mathcal{X}}_{et},\mathbb{R})\\
&\simeq& \mbox{det}_{\mathbb{R}} (\mbox{R}\Gamma_c(\mathcal{X}_{et},\mathbb{R})\oplus D_{\mathbb{R}}[1])
\end{eqnarray*}
hence does not depend on the choice of $u$.
\end{proof}

\subsection{The conjecture $\textbf{B}(\mathcal{X},d)$ and the regulator map}

Let $\mathcal{X}$ be a proper flat regular connected arithmetic scheme of dimension $d$ with generic fibre $X=\mathcal{X}_\mathbb{Q}$. The "integral part in the motivic cohomology" $H^{2d-1-i}_M(X_{/\mathbb{Z}},\mathbb{Q}(d))$ is defined as the image of the morphism
$$H^{2d-1-i}(\mathcal{X},\mathbb{Q}(d))\rightarrow H^{2d-1-i}(X,\mathbb{Q}(d)).$$
Let $H^p_{\mathcal D}(X_{/\mathbb{R}},\mathbb{R}(q))$ denote the real Deligne cohomology and let
$$\rho^i_\infty:H^{2d-1-i}_M(X_{/\mathbb{Z}},\mathbb{Q}(d))_{\mathbb{R}}\rightarrow H^{2d-1-i}_{\mathcal{D}}(X_{/\mathbb{R}},\mathbb{R}(d))$$
be the Beilinson regulator. According to a classical conjecture of Beilinson, the map $\rho^i_\infty$ should be an isomorphism for $i\geq 1$ and there should be an
exact sequence
$$0\rightarrow
H^{2d-1}_M(X_{/\mathbb{Z}},\mathbb{Q}(d))_{\mathbb{R}}\xrightarrow {\rho^0_\infty} H^{2d-1}_{\mathcal{D}}(X_{/\mathbb{R}},\mathbb{R}(d))\rightarrow CH^0(X)^*_{\mathbb{R}}\rightarrow 0$$
for $i=0$. Moreover, the natural map
$$H^{2d-1-i}(\mathcal{X},\mathbb{Q}(d))_{\mathbb{R}}\rightarrow H^{2d-1-i}(X,\mathbb{Q}(d))_{\mathbb{R}}$$ is expected to be injective for $i\geq0$. This suggests the following conjecture, where we consider the composite map
$$H^{2d-1-i}(\mathcal{X},\mathbb{Q}(d))_{\mathbb{R}}\rightarrow H^{2d-1-i}(X,\mathbb{Q}(d))_{\mathbb{R}}
\rightarrow H^{2d-1-i}_{\mathcal{D}}(X_{/\mathbb{R}},\mathbb{R}(d)).$$
\begin{conj}\emph{$\textbf{B}(\mathcal{X},d)$}
The map $$H^{2d-1-i}(\mathcal{X},\mathbb{Q}(d))_{\mathbb{R}}\rightarrow H^{2d-1-i}_{\mathcal{D}}(X_{/\mathbb{R}},\mathbb{R}(d))$$ is an isomorphism for $1\leq i\leq 2d-1$ and there is an
exact sequence
$$0\rightarrow
H^{2d-1}(\mathcal{X},\mathbb{Q}(d))_{\mathbb{R}}\xrightarrow {} H^{2d-1}_{\mathcal{D}}(X_{/\mathbb{R}},\mathbb{R}(d))\rightarrow CH^0(X)^*_{\mathbb{R}}\rightarrow 0
$$
for $i=0$.
\end{conj}
By \cite{Goncharov}, there is a canonical morphism of complexes
$$\Gamma(X,\mathbb{Q}(d))\rightarrow C_{\mathcal{D}}(X_{/\mathbb{R}},\mathbb{R}(d))$$
inducing Beilinson's regulator on cohomology \cite{Burgos-GoncharovRegulator}, where the complex $C_{\mathcal{D}}(X_{/\mathbb{R}},\mathbb{R}(d))$ computes real Deligne cohomology. Let $j:X_{Zar}\rightarrow\mathcal{X}_{Zar}$ be the natural embedding. Consider the map
$$\mbox{R}\Gamma(\mathcal{X},\mathbb{Q}(d))\rightarrow\mbox{R}\Gamma(\mathcal{X},j_*\mathbb{Q}(d))
\rightarrow\mbox{R}\Gamma(\mathcal{X},\mbox{R}j_*\mathbb{Q}(d))\simeq\mbox{R}\Gamma(X,\mathbb{Q}(d))\simeq\Gamma(X,\mathbb{Q}(d))$$
where the last isomorphism follows from the fact that, over a field, the Zariski hypercohomology of the cycle complex coincides with its cohomology. Then we consider the composite map
$$\rho_{\infty}:\mbox{R}\Gamma(\mathcal{X},\mathbb{Q}(d))\rightarrow\Gamma(X,\mathbb{Q}(d))\rightarrow C_{\mathcal{D}}(X_{/\mathbb{R}},\mathbb{R}(d)).$$
and we denote by $\mathcal{D}(\br)$ the derived category of $\br$-vector spaces. 
\begin{thm}\label{thm-beilinson} Let $\mathcal{X}$ be a proper flat regular connected scheme of dimension $d$ satisfying 
$\textbf{\emph{L}}(\mathcal{X}_{et},d)_{\geq 0}$ and  $\textbf{\emph{B}}(\mathcal{X},d)$. A choice of a direct sum decomposition (\ref{directsumdecompo}) induces, in a canonical way, an isomorphism in $\mathcal{D}(\br)$:
$$\mbox{\emph{R}}\Gamma_{W,c}(\mathcal{X},\mathbb{Z})\otimes\mathbb{R}\stackrel{\sim}{\longrightarrow} \mbox{\emph{R}}\Gamma_{W,c}(\mathcal{X},\tilde{\mathbb{R}}).$$
\end{thm}
\begin{proof}
Duality for Deligne cohomology (see \cite{Burgos-Gil} Corollary 2.28)
\begin{equation}\label{duality-Deligne-coh}
H^i_{\mathcal{D}}(X_{/\mathbb{R}},\mathbb{R}(p))\times H^{2d-1-i}_{\mathcal{D}}(X_{/\mathbb{R}},\mathbb{R}(d-p))\rightarrow\mathbb{R}
\end{equation}
yields an isomorphism in $\mathcal{D}(\br)$
$$C_{\mathcal{D}}(X_{/\mathbb{R}},\mathbb{R})\rightarrow
\mbox{RHom}(C_{\mathcal{D}}(X_{/\mathbb{R}},\mathbb{R}(d)),\mathbb{R}[-2d+1]).$$
Composing with $\mbox{R}\Gamma(\mathcal{X}_{\infty},\mathbb{R})\stackrel{\sim}{\longrightarrow}
C_{\mathcal{D}}(X_{/\mathbb{R}},\mathbb{R})$
we obtain
$$\mbox{R}\Gamma(\mathcal{X}_{\infty},\mathbb{R})\stackrel{\sim}{\longrightarrow}
\mbox{RHom}(C_{\mathcal{D}}(X_{/\mathbb{R}},\mathbb{R}(d)),\mathbb{R}[-2d+1]).$$

One has a morphism of complexes $C_{\mathcal{D}}(X_{/\mathbb{R}},\mathbb{R}(d))\rightarrow CH^0(X)_{\mathbb{R}}^*[-2d+1]$ and we define $\tilde{C}_{\mathcal{D}}(X_{/\mathbb{R}},\mathbb{R}(d))$ to be its mapping fiber. Applying the functor $\mbox{RHom}(-,\mathbb{R}[-2d+1])$, we obtain an exact triangle (note that $X$ is irreducible)
$$\mathbb{R}[0]\rightarrow \mbox{RHom}(C_{\mathcal{D}}(X_{/\mathbb{R}},\mathbb{R}(d)),\mathbb{R}[-2d+1])
\rightarrow \mbox{RHom}(\tilde{C}_{\mathcal{D}}(X_{/\mathbb{R}},\mathbb{R}(d)),\mathbb{R}[-2d+1]).$$
We have a morphism of exact triangles
\[ \xymatrix{
\mathbb{R}[0]\ar[d]\ar[r]&\mbox{R}\Gamma(\mathcal{X}_{\infty},\mathbb{R})\ar[d]\ar[r]&\mbox{R}\Gamma_c(\mathcal{X}_{et},\mathbb{R})[1] \ar[d]_{\exists !}\\
\mathbb{R}[0]\ar[r]&\mbox{RHom}(C_{\mathcal{D}}(X_{/\mathbb{R}},\mathbb{R}(d)),\mathbb{R}[-2d+1])
\ar[r]&\mbox{RHom}(\tilde{C}_{\mathcal{D}}(X_{/\mathbb{R}},\mathbb{R}(d)),\mathbb{R}[-2d+1])
}
\]
where the vertical map on the right hand side is uniquely determined. This yields a canonical quasi-isomorphism
\begin{equation}\label{uneptitmorphism}
\mbox{R}\Gamma_c(\mathcal{X}_{et},\mathbb{R})[-1]\stackrel{\sim}{\longrightarrow}
\mbox{RHom}(\tilde{C}_{\mathcal{D}}(X_{/\mathbb{R}},\mathbb{R}(d)),\mathbb{R}[-2d-1]).
\end{equation}
It follows from Conjecture $\textbf{B}(\mathcal{X},d)$ that the morphism of complexes
$\rho_{\infty}:\mbox{R}\Gamma(\mathcal{X},\mathbb{Q}(d))_{\geq0}\rightarrow C_{\mathcal{D}}(X_{/\mathbb{R}},\mathbb{R}(d))$
induces a quasi-isomorphism
$$\tilde{\rho}_{\infty,\mathbb{R}}:\mbox{R}\Gamma(\mathcal{X},\mathbb{Q}(d))_{\geq0,\mathbb{R}}\stackrel{\sim}{\longrightarrow} \tilde{C}_{\mathcal{D}}(X_{/\mathbb{R}},\mathbb{R}(d)).$$
Applying  the functor $\mbox{RHom}(-,\mathbb{R}[-2d-1])$ and composing with the map (\ref{uneptitmorphism}), we obtain an isomorphism:
$$\mbox{R}\Gamma_c(\mathcal{X}_{et},\mathbb{R})[-1]
\stackrel{\sim}{\rightarrow}\mbox{RHom}(\tilde{C}_{\mathcal{D}}(X_{/\mathbb{R}},\mathbb{R}(d)),\mathbb{R}[-2d-1])
\stackrel{\sim}{\rightarrow}\mbox{RHom}(\mbox{R}\Gamma(\mathcal{X},\mathbb{Q}(d))_{\geq0},\mathbb{R}[-2d-1]).$$
The inverse isomorphism
$\mbox{RHom}(\mbox{R}\Gamma(\mathcal{X},\mathbb{Q}(d))_{\geq0},\mathbb{R}[-2d-1])\stackrel{\sim}{\rightarrow} \mbox{R}\Gamma_c(\mathcal{X}_{et},\mathbb{R})[-1]$
together with a direct sum decomposition (\ref{directsumdecompo}) provide us with the desired map:
$$\mbox{R}\Gamma_{W,c}(\mathcal{X},\mathbb{Z})_{\mathbb{R}}
\stackrel{\sim}{\rightarrow}\mbox{R}\Gamma_c(\mathcal{X}_{et},\mathbb{Z})_{\mathbb{R}}\oplus
\mbox{RHom}(\mbox{R}\Gamma(\mathcal{X},\mathbb{Q}(d))_{\geq0},\mathbb{R}[-2d-1])$$
$$\stackrel{\sim}{\rightarrow}\mbox{R}\Gamma_c(\mathcal{X}_{et},\mathbb{R})\oplus \mbox{R}\Gamma_c(\mathcal{X}_{et},\mathbb{R})[-1]\stackrel{\sim}{\rightarrow}\mbox{R}\Gamma_{W,c}(\mathcal{X},\tilde{\mathbb{R}}).
$$
\end{proof}

\begin{cor}\label{cor-canmapondet} Let $\mathcal{X}$ be a proper regular connected scheme of dimension $d$ satisfying 
$\textbf{\emph{L}}(\mathcal{X}_{et},d)_{\geq0}$ and  $\textbf{\emph{B}}(\mathcal{X},d)$.
Then there is a \emph{canonical} isomorphism
$$\mbox{\emph{det}}_{\mathbb{R}} \mbox{\emph{R}}\Gamma_{W,c}(\mathcal{X},\mathbb{Z})_{\mathbb{R}}\stackrel{\sim}{\longrightarrow}
\mbox{\emph{det}}_{\mathbb{R}}\mbox{\emph{R}}\Gamma_{W,c}(\mathcal{X},\tilde{\mathbb{R}}).$$
\end{cor}

\begin{proof}
If $\mathcal{X}$ is flat over $\mathbb{Z}$, the result follows from Proposition \ref{lapropcanonique} and Theorem \ref{thm-beilinson}. So we may assume that $\mathcal{X}$ is smooth and proper over a finite field. Then we have
$\mbox{R}\Gamma_{W,c}(\mathcal{X},\mathbb{Z})=\mbox{R}\Gamma_{W}(\mathcal{X},\mathbb{Z})$
and a canonical isomorphism
$$\mbox{R}\Gamma_{W}(\mathcal{X},\mathbb{Z})_{\mathbb{R}}\simeq \mbox{R}\Gamma(\mathcal{X}_{et},\mathbb{R})\oplus
\mbox{RHom}(\mbox{R}\Gamma(\mathcal{X},\mathbb{Q}(d))_{\geq0},\mathbb{R}[-2d-1]).$$
It follows from $\textbf{L}(\mathcal{X}_{et},d)_{\geq0}$ that the natural map
$$\mbox{R}\Gamma(\mathcal{X}_{et},\mathbb{R})[-1]\rightarrow \mbox{RHom}(\mbox{R}\Gamma(\mathcal{X},\mathbb{Q}(d))_{\geq0},\mathbb{R}[-2d-1])$$
is a quasi-isomorphism (see the proof of Theorem \ref{thm-comparison-char-p}). This yields a canonical map
\begin{eqnarray}
\mbox{R}\Gamma_{W}(\mathcal{X},\mathbb{Z})_{\mathbb{R}}&\stackrel{\sim}{\rightarrow}&\mbox{R}\Gamma(\mathcal{X}_{et},\mathbb{R})\oplus
\mbox{RHom}(\mbox{R}\Gamma(\mathcal{X},\mathbb{Q}(d))_{\geq0},\mathbb{R}[-2d-1])\\
\label{charpiso}&\stackrel{\sim}{\rightarrow}& \mbox{R}\Gamma(\mathcal{X}_{et},\mathbb{R})\oplus\mbox{R}\Gamma(\mathcal{X}_{et},\mathbb{R})[-1]
\stackrel{\sim}{\rightarrow}\mbox{R}\Gamma_{W}(\mathcal{X},\tilde{\mathbb{R}}).
\end{eqnarray}

\end{proof}

\section{Zeta functions at $s=0$}

Recall that the zeta-function of a scheme $\X$ of finite type over $\mathrm{Spec}(\mathbb{Z})$ is defined as an infinite product 
\begin{equation}\label{zeta-product}
\zeta(\X,s):=\prod_{x\in \X_0}\frac{1}{1-N(x)^{-s}}
\end{equation}
where $\X_0$ denotes the set of closed points of $\X$ and $N(x)$ is the cardinality of the residue field of $x\in\X_0$. The product (\ref{zeta-product}) converges for $\Re(s)>\mathrm{dim}(\X)$. It is expected to have a meromorphic continuation to the whole complex plane. We refer to \cite{Knudsen-Mumford} for generalities on the determinant functor.

\subsection{The main Conjecture}
The following theorem summarizes some results obtained previously (see \cite{Flach-moi} for a) and b), Proposition \ref{finitelygenerated-cohomology} and Definition \ref{defncpctsuppcoh} for c) and Theorem \ref{thm-beilinson} and (\ref{charpiso}) for d)).
\begin{thm}\label{thm-recallat0}
Let $\mathcal{X}$ be a proper regular arithmetic scheme of pure dimension $d$.
\begin{itemize}
\item[a)] The compact support cohomology groups $H^i_{W,c}(\mathcal{X},\tilde{\mathbb{R}})$
are finite dimensional vector spaces over $\mathbb{R}$, vanish for almost
all $i$ and satisfy
\[ \sum_{i\in\mathbb{Z}}(-1)^i\dim_\mathbb{R} H^i_{W,c}(\mathcal{X},\tilde{\mathbb{R}})=0.\]
\item[b)] Cup product with the fundamental class $\theta\in H^1_W(\X,\tilde{\mathbb{R}})$ yields an acyclic complex
\[\cdots\xrightarrow{\cup\theta}H^i_{W,c}(\mathcal{X},\tilde{\mathbb{R}})\xrightarrow{\cup\theta}
H^{i+1}_{W,c}(\mathcal{X},\tilde{\mathbb{R}})\xrightarrow{\cup\theta}\cdots\]
\item[c)] If $\textbf{\emph{L}}(\mathcal{X}_{et},d)_{\geq0}$ holds, then the compact support cohomology groups $H^i_{W,c}(\mathcal{X},\mathbb{Z})$ are finitely generated over $\mathbb{Z}$ and they vanish for almost all $i$.
\item[d)] If $\textbf{\emph{B}}(\mathcal{X},d)$ also holds then there are isomorphisms
\[H^i_{W,c}(\mathcal{X},\mathbb{Z})\otimes_\mathbb{Z}\mathbb{R}\xrightarrow{\sim}H^i_{W,c}(\mathcal{X},\tilde{\mathbb{R}}).\]
\end{itemize}
\end{thm}
Consider now a proper regular connected arithmetic scheme $\X$ of dimension $d$ satisfying $\textbf{L}(\mathcal{X}_{et},d)_{\geq0}$ and $\textbf{B}(\mathcal{X},d)$. By Corollary \ref{cor-canmapondet} we have canonical isomorphisms
\begin{eqnarray*}
(\mbox{det}_{\mathbb{Z}}\mbox{R}\Gamma_{W,c}(\mathcal{X},\mathbb{Z}))_{\mathbb{R}}
&\simeq&\mbox{det}_{\mathbb{R}}(\mbox{R}\Gamma_{W,c}(\mathcal{X},\mathbb{Z})_{\mathbb{R}})\\
&\simeq& \mbox{det}_{\mathbb{R}}\mbox{R}\Gamma_{W,c}(\mathcal{X},\tilde{\mathbb{R}}). %\\
\end{eqnarray*}
Moreover, the exact sequence b) above provides us with
$$
\lambda:\mathbb{R}\stackrel{\sim}{\longrightarrow}
\bigotimes_{i\in\mathbb{Z}}\mbox{det}_{\mathbb{R}}H^i_{W,c}(\mathcal{X},\tilde{\mathbb{R}})^{(-1)^i}
\stackrel{\sim}{\longrightarrow}\mbox{det}_{\mathbb{R}}\mbox{R}\Gamma_{W,c}(\mathcal{X},\tilde{\mathbb{R}})
\stackrel{\sim}{\longrightarrow}(\mbox{det}_{\mathbb{Z}}\mbox{R}\Gamma_{W,c}(\mathcal{X},\mathbb{Z}))_{\mathbb{R}}.
$$
The following conjecture presupposes that $\zeta(\mathcal{X},s)$ has a meromorphic continuation to a neighborhood of $s=0$.
\begin{conj}\label{Lichtenbaum Conjecture}\
\begin{itemize}
\item[e)] The vanishing order of $\zeta(\mathcal{X},s)$ at $s=0$ is given by
\[ \mbox{\emph{ord}}_{s=0}\zeta(\mathcal{X},s)=\sum_{i\in\mathbb{Z}}(-1)^i\cdot i\cdot\mathrm{rank}_{\mathbb{Z}}H^i_{W,c}(\mathcal{X},\mathbb{Z}).\]
\item[f)] The leading coefficient $\zeta^*(\mathcal{X},0)$ in the Taylor expansion
 of $\zeta(\mathcal{X},s)$ at $s=0$ is given up to sign by
\[\mathbb{Z}\cdot\lambda(\zeta^*(\mathcal{X},0)^{-1})=\mbox{\emph{det}}_{\mathbb{Z}}\mbox{\emph{R}}\Gamma_{W,c}(\mathcal{X},\mathbb{Z}).\]
\end{itemize}
\end{conj}

\subsection{Relation to Soul\'e's conjecture}
Let $\mathcal{X}$ be a regular connected arithmetic scheme of dimension $d$. We assume in this subsection that $\mathcal{X}$ is moreover flat and projective over $\mbox{Spec}(\mathbb{Z})$.
The following conjecture is Bloch's reformulation (see \cite{Bloch86} Section 7) of Soul\'e's conjecture \cite{Soule} in terms of motivic cohomology (thanks to \cite{Levine-motivic-and-K-theory} Theorem 14.7 (5)). It presupposes that $\zeta(\mathcal{X},s)$ has a meromorphic continuation near $s=0$ and that the $\mathbb{Q}$-vector space $H^{i}(\mathcal{X},\mathbb{Q}(d))$ is finite dimensional for all $i$ and zero for almost all $i$.

\begin{conj}\label{Soul\'e-Conjectute} (Soul\'e)
One has $$\mbox{\emph{ord}}_{s=0}\zeta(\mathcal{X},s)=\sum_{i}(-1)^{i+1}\mbox{\emph{dim}}_{\mathbb{Q}}H^{2d-i}(\mathcal{X},\mathbb{Q}(d)).$$ \end{conj}

If $\mathcal{X}$ satisfies $\textbf{L}(\mathcal{X}_{et},d)$, then $H^{i}(\mathcal{X},\mathbb{Q}(d))=0$ for $i<0$ (see the proof of Theorem \ref{cor-functoriality}). Hence the conjunction of $\textbf{L}(\mathcal{X}_{et},d)$ and $\textbf{B}(\mathcal{X},d)$ is equivalent to the conjunction of Conjecture \ref{conjLintro} and Conjecture \ref{confBFintro} for $\X$.
\begin{prop} Assume that $\mathcal{X}$ satisfies $\textbf{\emph{L}}(\mathcal{X}_{et},d)$ and $\textbf{\emph{B}}(\mathcal{X},d)$. Then Conjecture \ref{Lichtenbaum Conjecture} e) is equivalent to Conjecture \ref{Soul\'e-Conjectute}.
\end{prop}
\begin{proof}
Assuming $\textbf{L}(\mathcal{X}_{et},d)$ and $\textbf{B}(\mathcal{X},d)$ one has
\begin{eqnarray*}
\sum_{i\in\mathbb{Z}}(-1)^i\cdot i\cdot\mathrm{rank}_{\mathbb{Z}}H^i_{W,c}(\mathcal{X},\mathbb{Z})&=&\sum_{i\in\mathbb{Z}}(-1)^i\cdot i\cdot\mbox{\mbox{dim}}_{\mathbb{Q}}H^i_{W,c}(\mathcal{X},\mathbb{Z})_{\mathbb{Q}}\\
&=&\sum_{i\in\mathbb{Z}}(-1)^i\cdot i\cdot(\mbox{\mbox{dim}}_{\mathbb{Q}}H^i_{c}(\mathcal{X}_{et},\mathbb{Q})+
\mbox{\mbox{dim}}_{\mathbb{Q}}H^{2d+2-i-1}(\mathcal{X},\mathbb{Q}(d))^*)\\
&=&\sum_{i\in\mathbb{Z}}(-1)^i\cdot i\cdot(\mbox{\mbox{dim}}_{\mathbb{Q}}H^{2d-i}(\mathcal{X},\mathbb{Q}(d))^*+
\mbox{\mbox{dim}}_{\mathbb{Q}}H^{2d+1-i}(\mathcal{X},\mathbb{Q}(d))^*)\\
&=&\sum_{i\in\mathbb{Z}}(-1)^i
\mbox{\mbox{dim}}_{\mathbb{Q}}H^{2d+1-i}(\mathcal{X},\mathbb{Q}(d)).
\end{eqnarray*}
The second equality follows from Proposition \ref{lapropcanonique}. Since $H^{i}(\mathcal{X},\mathbb{Q}(d))=0$ for $i<0$ by  $\textbf{L}(\mathcal{X}_{et},d)$, the third equality follows from Conjecture $\textbf{B}(\mathcal{X},d)$ and from duality for Deligne cohomology, see the proof of Theorem \ref{thm-beilinson}. The fourth equality follows from $\textbf{L}(\mathcal{X}_{et},d)$ since it implies that $H^{2d-i}(\mathcal{X},\mathbb{Q}(d))$ is finite dimensional and zero for almost all $i$. The result follows.
\end{proof}

\subsection{Relation to the Tamagawa number conjecture}\label{sect-TNC} In this section we consider the Tamagawa number conjecture of Bloch and Kato in the formulation of Fontaine and Perrin-Riou (see \cite{Fontaine92} and \cite{Fontaine-Perrin-Riou}).

Throughout this section we let $\X$ be a smooth projective scheme over a number ring $\mathcal{O}_F$, and we assume that $\mathcal{X}$ is connected of dimension $d$.  We set $\mathcal{X}_F:=\mathcal{X}\otimes_{\mathcal{O}_F}F$ and $\mathcal{X}_{\p}:=\mathcal{X}\otimes_{\mathcal{O}_F}\mathbb{F}_{\p}$ where $\mathbb{F}_{\p}:=\mathcal{O}_F/\p$ for any maximal ideal $\p\subset\mathcal{O}_F$.  We assume further that $\mathcal{X}$ satisfies ${\bf L}(\mathcal{X}_{et},d)_{\geq0}$ and ${\bf B}(\mathcal{X},d)$. In order to ease the notation, we also assume $\mathcal{O}_F=\mathcal{O}_{\X}(\X)$. In particular, $\X_F$ and $\mathcal{X}_{\p}$ (for any finite prime $\p$) are geometrically irreducible (see \cite{Liu} Corollary 5.3.17).

Note that there is no loss of generality caused by the assumption $\mathcal{O}_F=\mathcal{O}_{\X}(\X)$. Indeed, if $\X$ is connected,  smooth and projective over $\mathcal{O}_F$, then $\mathcal{O}_{\X}(\X)$ is a number ring (finite over $\mathcal{O}_F$) and $\X$ is also smooth and projective over $\mathcal{O}_{\X}(\X)$.

\subsubsection{Motivic L-functions}

Recall that for any connected, smooth and projective scheme $X/F$ and $0\leq i\leq 2\cdot \textrm{dim}(X)$ one defines the $L$-function
\begin{equation}\label{L-function}
L(h^i(X),s)=\prod_{\mathfrak{p}} L_{\mathfrak{p}}(h^i(X),s)=\prod_{\mathfrak{p}} P_{\mathfrak{p}}(h^i(X),N({\mathfrak{p}})^{-s})^{-1}
\end{equation}
as an Euler product over all finite primes $\mathfrak{p}$ of $F$ where
$N(\mathfrak{p})$ is the cardinal of $\mathbb{F}_{\mathfrak{p}}$ and
$$P_{\mathfrak{p}}(h^i(X),T)= \mydet_{\bq_l}\bigl(1-\mathrm{Fr}_{\mathfrak{p}}^{-1}\cdot T\vert
H^i(X_{\overline{F},et},\bq_l)^{I_{\mathfrak{p}}}\bigr)$$ is a polynomial (conjecturally) with
rational coefficients and independent of the chosen prime
$l$ as long as  $\mathfrak{p}$ does not divide $l$. Here $\mathrm{Fr}_{\mathfrak{p}}$ denotes a Frobenius element. The product (\ref{L-function}) is known to converge for
$\Re(s)>\frac{i}{2}+1$ by \cite{Deligne74}.

Assume now that $X/F$ is the generic fiber $\mathcal{X}\otimes_{\mathcal{O}_F}F$ of the smooth proper scheme $\mathcal{X}$ over $\mathcal{O}_F$.
By smooth and proper base change one has
$$H^i(\mathcal{X}_{\overline{\mathbb{F}}_{\mathfrak{p}},et},\mathbb{Q}_l)\simeq
H^i(\mathcal{X}_{\overline{F},et},\mathbb{Q}_l)\simeq
H^i(\mathcal{X}_{\overline{F},et},\mathbb{Q}_l)^{I_p}$$
and by Grothendieck's
formula, one has
$$\zeta(\X,s):=\prod_{x\in \X_0}\frac{1}{1-N(x)^{-s}}=\prod_{\mathfrak{p}}\zeta(\mathcal{X}_{\mathfrak{p}},s)=\prod_{i=0}^{2\,\textrm{dim}(\mathcal{X}_F)}L(h^i(\mathcal{X}_F),s)^{(-1)^i}$$
where $\X_0$ is the set of closed points in $\mathcal{X}$.

\subsubsection{Statement of the Tamagawa number conjecture}

Let $\mathcal{X}/\mathcal{O}_F$ be a scheme satisfying the assumptions of the introduction of this section \ref{sect-TNC}. We set $X:=\mathcal{X}_F$. Recall that the "integral part in the motivic cohomology" $H^{j}_M(X_{/\mathbb{Z}},\mathbb{Q}(d))$ is defined as the image of the map $H^{j}(\mathcal{X},\mathbb{Q}(d))\rightarrow H^{j}(X,\mathbb{Q}(d))$. But this map is injective for $j\geq0$ by ${\bf B}(\mathcal{X},d)$, hence we may identity $H^{j}_M(X_{/\mathbb{Z}},\mathbb{Q}(d))=H^{j}(\mathcal{X},\mathbb{Q}(d))$ for $j\geq0$. Define the fundamental line
$$\Delta_f(h^i(X))=\textrm{det}_{\mathbb{Q}}^{-1}(H^i(X(\mathbb{C}),\mathbb{Q})^+)\otimes_{\mathbb{Q}}
\textrm{det}_{\mathbb{Q}} H^{2d-1-i}_M(X_{/\mathbb{Z}},\mathbb{Q}(d))^*$$ for $0<i\leq 2d-2$ and
$$\Delta_f(h^0(X))=\textrm{det}_{\mathbb{Q}}CH^0(X)_{\mathbb{Q}}\otimes_{\mathbb{Q}}\textrm{det}_{\mathbb{Q}}^{-1}(H^0(X(\mathbb{C}),\mathbb{Q})^+)\otimes_{\mathbb{Q}}
\textrm{det}_{\mathbb{Q}} H^{2d-1}_M(X_{/\mathbb{Z}},\mathbb{Q}(d))^*$$ for $i=0$. There is an
isomorphism
$$ \vartheta^i_\infty:\mathbb{R}\simeq \Delta_f(h^i(X))_{\mathbb{R}}$$
induced by ${\bf B}(\mathcal{X},d)$ and duality for Deligne cohomology (\ref{duality-Deligne-coh}). Assuming that $L(h^i(X),0)$ has a meromorphic continuation near $s=0$, we denote by $L^*(h^i(X),0)$ the leading coefficient in the Taylor expansion of $L(h^i(X),0)$ at $s=0$.

\begin{conj}\label{beil} (Beilinson--Deligne) There is an identity of $\bq$-subspaces of  $\Delta_f(h^i(X))_{\mathbb{R}}$:
$$\bq\cdot \vartheta^i_\infty(L^*(h^i(X),0)^{-1})=\Delta_f(h^i(X)).$$
\end{conj}
Now fix a prime number $l$ and let $U\subseteq\textrm{Spec}(\mathcal{O}_F)$ be an open
subscheme in which $l$ is invertible. For any locally constant $l$-adic sheaf
$V$ on $U$ (i.e. any $\mathbb{Q}_l$-representation of the fundamental group $\pi_1(U_{et},\bar{u})$ with base point $\bar{u}:\textrm{Spec}(\overline{F})\rightarrow U$) and any finite prime $\mathfrak{p}$ of $F$ not dividing $l$, one defines a complex concentrated in
degrees $0$ and $1$
$$R\Gamma_f(F_{\mathfrak{p}},V)=R\Gamma(\mathbb{F}_{\mathfrak{p}},V^{I_\mathfrak{p}})
=V^{I_{\mathfrak{p}}}\stackrel{1-\textrm{Fr}_{\mathfrak{p}}^{-1}}{\longrightarrow}V^{I_{\mathfrak{p}}}$$
where $I_{\mathfrak{p}}$ is the inertia subgroup at ${\mathfrak{p}}$. For $\mathfrak{p}$ dividing $l$ define
$$R\Gamma_f(F_{\mathfrak{p}},V)=D_{cris,\p}(V) \stackrel{(1-\phi,\iota)}{\longrightarrow}D_{cris,\p}(V)\oplus D_{dR,\p}(V)/F^0D_{dR,\p}(V)$$
where $D_{cris,\p}(V)=(B_{cris,\mathfrak{p}}\otimes_{\mathbb{Q}_p}V)^{G_{F_\mathfrak{p}}}$ and $D_{dR,\p}(V)=(B_{dR,\mathfrak{p}}\otimes_{\mathbb{Q}_p}V)^{G_{F_\mathfrak{p}}}$ (see \cite{Fontaine92} and \cite{Fontaine-Perrin-Riou}). In both cases there is a map of complexes
$ R\Gamma_f(F_{\mathfrak{p}},V)\to R\Gamma(F_{\mathfrak{p}},V)$
and one defines $R\Gamma_{/f}(F_{\mathfrak{p}},V)$ as the mapping cone.  Then one defines a global complex $R\Gamma_f(F,V)$ as the mapping fibre of the composite map
$$R\Gamma(U_{et},V)\rightarrow\bigoplus_{\mathfrak{p}\notin U}R\Gamma(F_{\mathfrak{p}},V)\rightarrow\bigoplus_{\mathfrak{p}\notin U}R\Gamma_{/f}(F_{\mathfrak{p}},V).$$
There is an exact triangle in the derived category of
$\bq_l$-vector spaces 
\begin{equation}
R\Gamma_c(U_{\et},V)\to
R\Gamma_f(F,V)\to\bigoplus_{\mathfrak{p}\notin U}R\Gamma_f(F_{\mathfrak{p}},V)\label{tri1}
\end{equation} where the primes $\mathfrak{p}\notin U$ include archimedean $\mathfrak{p}$ with the convention $R\Gamma_f(\mathbb{R},V)=R\Gamma(\mathbb{R},V)$ and $R\Gamma_f(\mathbb{C},V)=R\Gamma(\mathbb{C},V)$. 

\begin{notation}\label{onenotat}
For an archimedean prime $\p$ of $F$, we choose a complex embedding $\sigma_{\p}:F\rightarrow \bc$ representing $\p$, and we set $\X_{\p}:=\X_{F,\sigma_{\p}}(\mathbb{C})/G_{F_{\p}}$, where $\mathcal{X}_{F,\sigma_{\p}}(\bc):=\mathrm{Hom}_{\mathrm{Spec}(F)}(\mathrm{Spec}(\mathbb{C}),\mathcal{X}_F)$ is the space of complex points of $\X_{F}$ lying over  $\sigma_{\p}$. One has $\X_{\infty}=\coprod_{\p\mid\infty}\X_{\p}$. Finally, we set $\X_{\p,\et}:=Sh(\X_{\p})$ and  $\X_{\p,W}:=Sh(\X_{\p})\times B_{\br}$.
\end{notation}
The following result is proven in \cite{Flach-moi} Proposition 9.1. Note that (\cite{Flach-moi} Conjecture 7) is known in the smooth case.

\begin{prop}\label{prop-matthias} Let
$\pi:\X\rightarrow\Spec(\mathcal{O}_F)$ be as above and let
$\overline{\X}_\et$ be its Artin-Verdier \'etale topos. Let
$U\subseteq\Spec(\mathcal{O}_F)$ be an open subscheme in which the prime number $l$ is invertible. Then there is an isomorphism of exact triangles in the derived
category of $\bq_l$-vector spaces:
\[\minCDarrowwidth1em\begin{CD}
R\Gamma_c(\X_{U,\et},\bq_l) @>>> R\Gamma(\overline{\X}_{\et},\bq_l) @>>>
\bigoplus\limits_{\mathfrak{p}\notin U}R\Gamma(\X_{\mathfrak{p},\et},\bq_l) \\
@VVV @VVV @VVV @.\\
\bigoplus\limits_{i=0}^{2d-2}R\Gamma_c(U_\et,V^i_l)[-i] @>>>
\bigoplus\limits_{i=0}^{2d-2}R\Gamma_f(F,V^i_l)[-i]
@>>>\bigoplus\limits_{\mathfrak{p}\notin U}\bigoplus\limits_{i=0}^{2d-2}R\Gamma_f(F_{\mathfrak{p}},V^i_l)[-i]
\end{CD}\]
Here $V^i_l:=H^i(X_{\overline{F},\et},\bq_l)$ is viewed as a $\bq_l$-representation of $G_F$, the bottom exact
triangle is a sum over triangles (\ref{tri1}) and $R\Gamma_c(\X_{U,\et},\bq_l):=R\Gamma(\overline{\X}_{\et},j_{!}\bq_l)$
where $j:\X_{U,et}\rightarrow \overline{\X}_{\et} $ is the canonical open embedding.
\label{reform}\end{prop}
The statement of the Tamagawa number conjecture requires the following
\begin{conj} (Bloch-Kato) We have $H^1_f(F,V_l^i)=0$ for any $i$.
\label{bk}\end{conj}
One can show that $H^0_f(F,V_l^0)\cong CH^0(\mathcal{X}_F)_{\bq_l}$,
$H^0_f(F,V_l^i)=0$ for $i>0$ and $H^3_f(F,V_l^i)=0$. Therefore, by Corollary \ref{cor-comp-l-adic} and Proposition \ref{prop-matthias}, Conjecture \ref{bk} induces an isomorphism:
\begin{equation}\label{matmap}
\rho^i_l:H^2_f(F,V_l^i)\stackrel{\sim}{\rightarrow}H^{i+2}(\overline{\X}_{et},\mathbb{Q}_l)\stackrel{\sim}{\rightarrow}
H_W^{i+2}(\overline{\X},\mathbb{Z})_{\mathbb{Q}_l}\stackrel{\sim}{\rightarrow} H^{2d-1-i}(\mathcal{X},\bq(d))_{\bq_l}^*.
\end{equation}
Moreover Artin's comparison isomorphism yields
\begin{eqnarray}
H^i(\mathcal{X}(\bc),\bq)^+_{\bq_l}&\simeq&(\bigoplus_{\sigma:F\hookrightarrow\mathbb{C}}H^i(\mathcal{X}_{F,\sigma}(\bc),\bq))_{\bq_l}^+ \\
\label{atinfty}&\simeq& \bigoplus_{\mathfrak{p}|\infty}H^i(\mathcal{X}_{\overline{F},\et},\bq_l)^{G_{\mathfrak{p}}}=  \bigoplus_{\mathfrak{p}|\infty}(V^i_l)^{G_{\mathfrak{p}}}
\end{eqnarray}
where $\sigma$ runs over the complex embeddings of $F$ and $\mathcal{X}_{F,\sigma}(\bc)=\textrm{Hom}_{\textrm{Spec}(F)}(\textrm{Spec}(\mathbb{C}),\mathcal{X}_F)$ with respect to the map $\sigma:\textrm{Spec}(\mathbb{C})\rightarrow \textrm{Spec}(F)$. We obtain an isomorphism (for $0\leq i\leq 2d-2$):
\begin{eqnarray}
\label{onedet}\vartheta^i_l:\,\,\, \Delta_f(h^i(X))_{\bq_l} &\stackrel{\sim}{\longrightarrow}& \mydet_{\bq_l}R\Gamma_f(F,V_l^i)\otimes\bigotimes_{\mathfrak{p}|\infty}\mydet^{-1}_{\bq_l}R\Gamma(F_{\mathfrak{p}},V_l^i)\\
\label{twodet}&\stackrel{\sim}{\longrightarrow}& \mydet_{\bq_l}R\Gamma_c(U_{\et},V^i_l)\otimes\bigotimes_{\mathfrak{p}\in Z}\mydet_{\bq_l}R\Gamma_f(F_{\mathfrak{p}},V_l^i)\\
\label{threedet}&\stackrel{\sim}{\longrightarrow}& \mydet_{\bq_l}R\Gamma_c(U_{\et},V^i_l)
\end{eqnarray}
where $Z=\mathrm{Spec}(\mathcal{O}_F)-U$ is the closed complement of $U$. Indeed, (\ref{onedet}) is induced by (\ref{matmap}) and (\ref{atinfty}), (\ref{twodet}) is induced by (\ref{tri1}), and (\ref{threedet}) is induced
by the isomorphism (for $\mathfrak{p}\in Z$)
$$\iota_{\mathfrak{p}}:\mydet_{\bq_l}R\Gamma_f(F_{\mathfrak{p}},V_l^i)\cong \bq_l  $$
which is in turn induced by the identity map on $(V_l^i)^{I_{\mathfrak{p}}}$ and $D_{cris,\p}(V_l^i)$. Below is the Tamagawa number conjecture in the formulation of Fontaine and Perrin-Riou (see \cite{Fontaine92} and \cite{Fontaine-Perrin-Riou}). It presupposes Conjecture \ref{beil}.
\begin{conj} ($l$-part of the Tamagawa number conjecture) There is an identity of
free rank one $\bz_l$-submodules of
$\mydet_{\bq_l}R\Gamma_c(U_{\et},V^i_l)$
\[
\bz_l\cdot\vartheta^i_l\circ\vartheta^i_\infty(L^*(h^i(X),0)^{-1})=\mydet_{\bz_l}R\Gamma_c(U_{\et},T^i_l)\]
for any constructible $\mathbb{Z}_l$-sheaf $T^i_l$ such that $T^i_l\otimes_{\mathbb{Z}_l}\mathbb{Q}_l= V^i_l$.
\label{tam}\end{conj}This conjecture is independent of the choice of $T_l^i$, as well as of the choice of $U$ in which $l$ is invertible \cite{Fontaine92}.

\subsubsection{}
One can reformulate the Tamagawa number conjecture in terms of the L-function
$$ L_U(h^i(X),s)=\prod_{\mathfrak{p}\in U} L_{\mathfrak{p}}(h^i(X),s)$$
associated to the smooth $l$-adic sheaf $V_l^i$ over $U$, using a second isomorphism
\[ \tilde{\iota}_{\mathfrak{p}}:\mydet_{\bq_l}R\Gamma_f(F_{\mathfrak{p}},V^i_l)\simeq \bq_l  \]
which satisfies
\begin{equation}\label{treslong}
\iota_{\mathfrak{p}}=P^*_{\mathfrak{p}}(h^i(X),1)^{-1}\tilde{\iota}_{\mathfrak{p}}=L_{\mathfrak{p}}^*(h^i(X),0)\log(N(\mathfrak{p})^{r_{i,\mathfrak{p}}})
\tilde{\iota}_{\mathfrak{p}}
\end{equation}
where $r_{i,\mathfrak{p}}=\ord_{T=1}P_{\mathfrak{p}}(h^i(X),T)=-\ord_{s=0}L_{\mathfrak{p}}(h^i(X),s)$. The isomorphism $\tilde{\iota}_{\mathfrak{p}}$ is defined as follows. We shall see below that the complex $R\Gamma_f(F_{\mathfrak{p}},V^i_l)$ is semi-simple at $0$; in other words, the identity map (on $(V_l^i)^{I_{\mathfrak{p}}}$ for $\p \nmid l$ and $D_{cris,\p}(V_l^i)$ for $\p\mid l$) induces an isomorphism $H^0_f(F_{\mathfrak{p}},V^i_l)\stackrel{\sim}{\rightarrow }H^1_f(F_{\mathfrak{p}},V^i_l)$  hence a trivialization
$$\tilde{\iota}_{\mathfrak{p}}:\mydet_{\bq_l}R\Gamma_f(F_{\mathfrak{p}},V^i_l)\simeq \mydet_{\bq_l}H^0_f(F_{\mathfrak{p}},V^i_l)\otimes_{\bq_l} \mydet^{-1}_{\bq_l}H^1_f(F_{\mathfrak{p}},V^i_l)\simeq \bq_l.$$
Then we define 
\begin{equation}\label{tilde-vartheta-l}
\tilde{\vartheta}^i_l: \Delta_f(h^i(X))_{\bq_l}\stackrel{\sim}{\rightarrow}
\mydet_{\bq_l}R\Gamma_c(U_{\et},V^i_l)\otimes\bigotimes_{\mathfrak{p}\in Z}\mydet_{\bq_l}R\Gamma_f(F_{\mathfrak{p}},V_l^i)
\stackrel{\sim}{\rightarrow} \mydet_{\bq_l}R\Gamma_c(U_{\et},V^i_l)
\end{equation}
for $0\leq i\leq 2d-2$, where the first isomorphism is (\ref{twodet}) and the second is induced by the $\tilde{\iota}_{\mathfrak{p}}$'s. By (\ref{treslong}) we have  $\tilde{\vartheta}^i_l=\prod_{\mathfrak{p}\in Z} P^*_{\mathfrak{p}}(h^i(X),1)\cdot \vartheta_l^i$ and we need to define
\begin{equation}\label{tilde-vartheta-infty}
\tilde{\vartheta}^i_\infty:=\prod_{\mathfrak{p}\in Z}\log(N(\mathfrak{p}))^{r_{i,\mathfrak{p}}}\vartheta_\infty^i.
\end{equation}
In our situation, we have 
$r_{i,\p}=0$ for $i>0$ (for weight reasons) and $r_{0,\p}=1$ (because $V_l^0=\mathbb{Q}_l$ with trivial $G_F$-action, see below) for any $\p\in Z$. The Tamagawa
number conjecture becomes
\begin{conj}\label{conjTNCtrunc}
There is an identity of
free rank one $\bz_l$-submodules of
$\mydet_{\bq_l}R\Gamma_c(U_{\et},V^i_l)$
\begin{equation}\label{tam1}
\bz_l\cdot\tilde{\vartheta}^i_l\circ\tilde{\vartheta}^i_\infty(L^*_U(h^i(X),0)^{-1})=
\mydet_{\bz_l}R\Gamma_c(U_{\et},T^i_l).
\end{equation}
\end{conj}

Now we observe that $R\Gamma_f(F_{\mathfrak{p}},V^i_l)$ is indeed semi-simple at $0$, and explain the compatibility between $\tilde{\iota}_{\mathfrak{p}}$ and the canonical trivialization of  $(\mathrm{det}_{\mathbb{Z}}R\Gamma(\mathcal{X}_{\mathfrak{p},W},\mathbb{Z}))_{\mathbb{Q}_l}$. Let $\p\in Z$. By Proposition \ref{prop-matthias} we have
\begin{equation}\label{shortiso}
\bigoplus_{i\geq0}R\Gamma_f(F_{\mathfrak{p}},V^i_l)[-i]\simeq R\Gamma(\mathcal{X}_{\mathfrak{p},et},\bq_l)\simeq R\Gamma(\mathcal{X}_{\mathfrak{p},W},\mathbb{Z})_{\mathbb{Q}_l}.
\end{equation}
But $H^i(\mathcal{X}_{\mathfrak{p},W},\mathbb{Z})_{\mathbb{Q}_l}=\mathbb{Q}_l$ for $i=0,1$ (since $\X_{\p}$ is connected) and $H^i(\mathcal{X}_{\mathfrak{p},W},\mathbb{Z})_{\mathbb{Q}_l}=0$ otherwise (by \cite{Lichtenbaum-finite-field} Theorem 7.4). It follows that $R\Gamma_f(F_{\mathfrak{p}},V^i_l)$ is acyclic for $i>0$, hence semi-simple at $0$. For $i=0$, we have $V^0_l=\bq_l$ with trivial $G_{F}$-action (since $\X_F$ is geometrically connected), hence $(V_l^0)^{I_{\p}}=\mathbb{Q}_l$ for $\p\nmid l$ and $D_{cris,\p}(V_l^0)=(F_{\p})_0$ for $\p\mid l$. In both cases, $R\Gamma_f(F_{\mathfrak{p}},V^0_l)$ is semi-simple at $0$. Moreover, (\ref{treslong}) is given by (\cite{Burns-Flach-98} Lemma 1). Indeed, for $i>0$ and $\p\mid l$, one has $$P_{\p}(V_l^i,1)=P_{l}((V_l^i)',1)=\mydet_{\bq_l}(1-\phi\mid D_{cris,\p}(V_l^i))$$
where $(V_l^i)':=\mathrm{Ind}_{F_{\p}/\bq_l}(V_l^i)$ is the $l$-adic representation of $G_{\bq_l}$ induced by $V_l^i$ (seen as a representation of $G_{F_{\p}}$). The case $\p\nmid l$ is similar and the case $i=0$ is obvious.

Note also that $H^i_f(F_{\mathfrak{p}},V^0_l)\simeq H^i(\mathcal{X}_{\mathfrak{p},W},\mathbb{Z})_{\mathbb{Q}_l}$ for $i=0,1$. Under this identification, the isomorphism $H^0_f(F_{\mathfrak{p}},V^0_l)\stackrel{\sim}{\rightarrow }H^1_f(F_{\mathfrak{p}},V^0_l)$ given by semi-simplicity of $R\Gamma_f(F_{\mathfrak{p}},V^0_l)$ corresponds to the map $H^0(\mathcal{X}_{\mathfrak{p},W},\mathbb{Z})_{\mathbb{Q}_l}\stackrel{\cup e}{\rightarrow} 
H^1(\mathcal{X}_{\mathfrak{p},W},\mathbb{Z})_{\mathbb{Q}_l}$ given by cup-product with the fundamental class $e\in H^1(\mathcal{X}_{\mathfrak{p},W},\mathbb{Z})$. Recall that $e$ is the pull-back of the homomorphism in $H^1(\mathrm{Spec}(\mathbb{F}_{\p})_W,\mathbb{Z})=\mathrm{Hom}(W_{\mathbb{F}_{\p}},\mathbb{Z})$ which sends $\mathrm{Fr}_{\p}$ to $1$. It follows that we have a commutative square of isomorphisms
\begin{equation}\begin{CD}
\bigotimes_{i\geq0}\mydet_{\bq_l}^{(-1)^i}R\Gamma_f(F_{\mathfrak{p}},V^i_l)
 @>\bigotimes_{i}\tilde{\iota}_{p}^{(-1)^i}>>
\bq_l \\
%\Vert@. @VVV\\
@VVV \Vert@.\\
%@VVV @VVV\\
\mydet_{\bq_l}R\Gamma(\mathcal{X}_{\mathfrak{p},W},\mathbb{Z})_{\mathbb{Q}_l}@>\mu_{\X_{\p},\bq_l}^{-1}>>
\mathbb{Q}_l
\end{CD}\label{dia00}
\end{equation}
where the left vertical isomorphism is induced by (\ref{shortiso}), and the lower horizontal isomorphism is induced by $\mu_{\X_{\p}}:\mathbb{Q}\stackrel{\sim}{\rightarrow } \mydet_{\bq}R\Gamma(\mathcal{X}_{\mathfrak{p},W},\mathbb{Z})_{\bq}$ which is in turn induced by the exact sequence (see \cite{Lichtenbaum-finite-field} Theorem 7.4)
\begin{equation}\label{licht-exactsequ}
...\stackrel{\cup e}{\rightarrow} H^i(\mathcal{X}_{\mathfrak{p},W},\mathbb{Z})_{\mathbb{Q}}\stackrel{\cup e}{\rightarrow} 
H^{i+1}(\mathcal{X}_{\mathfrak{p},W},\mathbb{Z})_{\mathbb{Q}}\stackrel{\cup e}{\rightarrow} ...
\end{equation}

\subsubsection{}
Let $\mathcal{X}/\mathcal{O}_F$ be a smooth projective scheme satisfying the assumptions of the introduction of this section \ref{sect-TNC}. Let $U\subseteq\textrm{Spec}(\mathcal{O}_F)$ be an open
subscheme on which the prime number $l$ is invertible. Consider the morphism $$\mbox{R}\Gamma_W(\overline{\mathcal{X}},\mathbb{Z})
\longrightarrow\bigoplus_{\mathfrak{p}\notin U}\mbox{R}\Gamma(\mathcal{X}_{\mathfrak{p},W},\mathbb{Z})$$
given by Proposition \ref{prop-functoriality-flat-p} and Proposition \ref{prop-iinfty}, where $\mathcal{X}_{\mathfrak{p},W}$ is the Weil-\'etale topos of $\mathcal{X}_{\mathfrak{p}}$. Here the primes $\p\notin U$ include archimedean primes, see Notation \ref{onenotat}. We define $\mbox{R}\Gamma_{W,c}(\mathcal{X}_U,\mathbb{Z})$ such that the triangle
\begin{equation}\label{prop-exact-triangles-with-support}
\mbox{R}\Gamma_{W,c}(\mathcal{X}_U,\mathbb{Z})\rightarrow \mbox{R}\Gamma_W(\overline{\mathcal{X}},\mathbb{Z})
\rightarrow\bigoplus_{\mathfrak{p}\notin U}\mbox{R}\Gamma(\mathcal{X}_{\mathfrak{p},W},\mathbb{Z})
\end{equation}
is exact. We do not show nor use that $\mbox{R}\Gamma_{W,c}(\mathcal{X}_U,\mathbb{Z})$ only depends on $\mathcal{X}_U$. In fact
$\mbox{R}\Gamma_{W,c}(\mathcal{X}_U,\mathbb{Z})$ is only defined up to a non-canonical isomorphism
but $\mbox{det}_{\mathbb{Z}}\mbox{R}\Gamma_{W,c}(\mathcal{X}_U,\mathbb{Z})$ is canonically defined and we have a canonical isomorphism
\begin{equation}\label{unedeplus}
\mbox{det}_{\mathbb{Z}}\mbox{R}\Gamma_{W,c}(\mathcal{X}_U,\mathbb{Z})\simeq \mbox{det}_{\mathbb{Z}}\mbox{R}\Gamma_{W,c}(\mathcal{X},\mathbb{Z})\otimes\mbox{det}^{-1}_{\mathbb{Z}}\mbox{R}\Gamma
(\mathcal{X}_{Z,W},\mathbb{Z})
\end{equation}
where $Z=\textrm{Spec}(\mathcal{O}_F)-U$. We define $$\mbox{R}\Gamma_{W,c}(\mathcal{X}_U,\tilde{\mathbb{R}}):=\mbox{R}\Gamma(\overline{\mathcal{X}}_W,j_!\tilde{\mathbb{R}})$$
where $j:\mathcal{X}_{U,W}\rightarrow \overline{\mathcal{X}}_W$ is the obvious open embedding of topoi (i.e. induced by $\X_{U,et}\rightarrow \overline{\X}_{et}$). The triangle
$$\mbox{R}\Gamma_{W,c}(\mathcal{X}_U,\tilde{\mathbb{R}})\rightarrow\mbox{R}\Gamma_{W}(\overline{\mathcal{X}},\tilde{\mathbb{R}})
\rightarrow\bigoplus_{\mathfrak{p}\notin U}\mbox{R}\Gamma(\mathcal{X}_{\mathfrak{p},W},\tilde{\mathbb{R}})$$
is exact, where the map on the right hand side is induced by the closed embedding 
$\coprod_{\mathfrak{p}\notin U}\mathcal{X}_{\mathfrak{p},W}\rightarrow\overline{\mathcal{X}}_W$
(which is the closed complement of $j$). We obtain a canonical isomorphism
\begin{equation}\label{2deplus}
\mbox{det}_{\mathbb{R}}\mbox{R}\Gamma_{W,c}(\mathcal{X}_U,\tilde{\mathbb{R}})\simeq \mbox{det}_{\mathbb{R}}\mbox{R}\Gamma_{W,c}(\mathcal{X},\tilde{\mathbb{R}})\otimes
\mbox{det}^{-1}_{\mathbb{R}}\mbox{R}\Gamma_W(\mathcal{X}_{Z},\tilde{\mathbb{R}}).
\end{equation}
By Corollary \ref{cor-canmapondet}, (\ref{unedeplus}) and (\ref{2deplus}) we have a canonical isomorphism
\begin{equation}\label{3deplus}
\mbox{det}_{\mathbb{R}}\mbox{R}\Gamma_{W,c}(\mathcal{X}_U,\mathbb{Z})_{\mathbb{R}}\stackrel{\sim}{\longrightarrow}
\mbox{det}_{\mathbb{R}}\mbox{R}\Gamma_{W,c}(\mathcal{X}_U,\tilde{\mathbb{R}})
\end{equation}
and we obtain
\begin{equation}\label{4deplus}
\lambda_{\X_U}:\mathbb{R}\stackrel{\sim}{\longrightarrow}
\mbox{det}_{\mathbb{R}}\mbox{R}\Gamma_{W,c}(\mathcal{X}_U,\tilde{\mathbb{R}})\stackrel{\sim}{\longrightarrow}
\mbox{det}_{\mathbb{R}}\mbox{R}\Gamma_{W,c}(\mathcal{X}_U,\mathbb{Z})_{\mathbb{R}}
\end{equation}
such that the following square of isomorphisms commutes:
\begin{equation}\begin{CD}
\mathbb{R}\otimes\mathbb{R} @>\lambda_{\X}\otimes \lambda_{\X_Z}^{-1}>>
\mbox{det}_{\mathbb{R}}\mbox{R}\Gamma_{W,c}(\mathcal{X},\mathbb{Z})_{\mathbb{R}} \otimes \mbox{det}^{-1}_{\mathbb{R}}\mbox{R}\Gamma(\mathcal{X}_{Z,W},\mathbb{Z})_{\mathbb{R}}\\
@VVV @VVV\\
\mathbb{R}@>\lambda_{\X_U}>>
\mbox{det}_{\mathbb{R}}\mbox{R}\Gamma_{W,c}(\mathcal{X}_U,\mathbb{Z})_{\mathbb{R}}
\end{CD}\label{dia3}\end{equation}
Here we identify $\lambda_{\X_Z}$ with its dual, the left vertical map is induced by the product map and the right vertical map is induced by (\ref{unedeplus}).

\begin{thm} Let $\mathcal{X}/\mathcal{O}_F$ be a smooth projective scheme over $\mathcal{O}_F$. Assume that $\X$ is connected of dimension $d$ and that $\X$ satisfies ${\bf L}(\mathcal{X}_{et},d)_{\geq0}$ and ${\bf B}(\mathcal{X},d)$. Assume moreover that  $H^1_f(F,H^i(\X_{\overline{F},\et},\bq_l))=0$ for all $i$, and that $\zeta(\X,s)$ has a meromorphic continuation to $s=0$. Then the Tamagawa number conjecture (Conjecture
\ref{tam}) for the motive
$\bigoplus_{i=0}^{2d-2}h^i(X)[-i]$
and all $l$ is equivalent to statement f) of Conjecture \ref{Lichtenbaum Conjecture} for $\X$.
\label{tamcompare}\end{thm}

\begin{proof}
We consider an open subscheme $U\subseteq\Spec(\mathcal{O}_F)$ on which $l$ is invertible and we let $Z$ be the closed complement of $U$. By \cite{Lichtenbaum-finite-field} $\X_{\p}$ satisfies Conjecture \ref{conjLichtenbaumIntro} for any $\p\in Z$. Hence the factorization
$\zeta(\X,s)=\zeta(\X_U,s)\cdot \zeta(\X_Z,s)$
together with (\ref{unedeplus}), (\ref{2deplus}) and (\ref{dia3}) show that Conjecture \ref{Lichtenbaum Conjecture} f) for
$\X$ is equivalent to Conjecture \ref{Lichtenbaum Conjecture} f) for $\X_U$.

By Proposition \ref{lapropcanonique} and by definition of $\lambda_{\X}$, we have a canonical isomorphism
\begin{equation}\label{isotriv}
\mydet_\bq R\Gamma_{W,c}(\X,\bz)_\bq\simeq \bigotimes_{i=0}^{2d-2}\Delta_f(h^i(X))^{(-1)^i}
\end{equation}
which is compatible (in the obvious sense) with $\lambda_{\X}$ and  $\otimes_i(\vartheta_\infty^i)^{(-1)^i}$. In particular, Conjecture \ref{Lichtenbaum Conjecture} f) implies Conjecture
\ref{beil} for $\bigoplus_{i=0}^{2d-2}h^i(X)[-i]$.
Moreover, for any prime $\p\in Z$,
cup-product with the canonical class $e\in H^1(\X_{\p,W},\bz)$
yields a trivialization (induced by (\ref{licht-exactsequ}))
\begin{equation}\label{triv-char-p}
\mu_{\X_{\p}}:\mathbb{Q}\cong\mydet_\bq R\Gamma(\X_{\p,W},\bz)_\bq.
\end{equation}
Then (\ref{unedeplus}), (\ref{isotriv}) and (\ref{triv-char-p}) yield an isomorphism
\begin{align*}
\vartheta_W:\mydet_\bq R\Gamma_{W,c}(\X_{U},\bz)_\bq\cong & \mydet_\bq
R\Gamma_{W,c}(\X,\bz)_\bq\otimes \bigotimes_{\p\in
Z}\mydet_\bq^{-1} R\Gamma(\X_{\p,W},\bz)_\bq\\
\cong & \bigotimes_{i=0}^{2d-2}\Delta_f(h^i(X))^{(-1)^i}.
\end{align*}
We claim that the following diagram of isomorphisms is commutative:
\begin{equation}\begin{CD} \br @>\lambda_{\X_U} >> \mydet_\br
R\Gamma_{W,c}(\X_{U},\bz)_\br\\
\Vert@. @VV \vartheta_{W,\br} V\\
\br @>\otimes_i(\tilde{\vartheta}_\infty^i)^{(-1)^i} >>
\bigotimes_{i=0}^{2d-2}\Delta_f(h^i(X))_\br^{(-1)^i}
\end{CD}\label{dia0R}\end{equation}
where $\lambda_{\X_U}$ (respectively $\tilde{\vartheta}_\infty^i$) is defined in (\ref{4deplus}) (respectively in (\ref{tilde-vartheta-infty})). Similarly,  for any prime $l$ invertible on $U$ we have a
commutative diagram of isomorphisms
\begin{equation}\begin{CD}
\mydet_{\bq_l} R\Gamma_{W,c}(\X_{U},\bz)_{\bq_l} @>>>
\mydet_{\bq_l} R\Gamma_c(\X_{U,\et},\bq_l) \\
@VV \vartheta_{W,\bq_l} V @VVV\\
\bigotimes_{i=0}^{2d-2}\Delta_f(h^i(X))_{\bq_l}^{(-1)^i}@>\otimes_i(\tilde{\vartheta}_l^i)^{(-1)^i}>>
\bigotimes_{i=0}^{2d-2}\mydet_{\bq_l}^{(-1)^i}R\Gamma_c(U_\et,V_l^i)
\end{CD}\label{dia4}\end{equation}
where the top horizontal isomorphism is induced by the isomorphism
\begin{equation}\label{vaserviraussi}
\mydet_{\bz} R\Gamma_{W,c}(\X_{U},\bz)\otimes_\bz\bz_l\cong \mydet_{\bz_l} R\Gamma_c(\X_{U,\et},\bz_l) 
\end{equation}
given by (\ref{unedeplus}) and Corollary \ref{cor-comp-l-adic}, while the right vertical isomorphism is
induced by the isomorphism
\begin{equation}\label{vaservir}
\mydet_{\bz_l} R\Gamma_c(\X_{U,\et},\bz_l)\cong \mydet_{\bz_l} R\Gamma_c(U_\et,R\pi_*\bz_l)\cong \bigotimes_{i=0}^{2d-2}\mydet_{\bz_l}^{(-1)^i}
R\Gamma_c(U_\et,T_l^i)
\end{equation}
where $T_l^i:=R^i\pi_*\bz_l$. Hence the $\bz_l$-lattice of $\bigotimes_{i=0}^{2d-2}\Delta_f(h^i(X))_{\bq_l}^{(-1)^i}$ given by the images of $\mydet_{\bz_l} R\Gamma_{W,c}(\X_{U},\bz)_{\bz_l}$ and $\bigotimes_{i=0}^{2d-2}\mydet_{\bz_l}^{(-1)^i}
R\Gamma_c(U_\et,T_l^i)$ coincide. This shows that Conjecture \ref{Lichtenbaum Conjecture} f) implies Conjecture
\ref{conjTNCtrunc} (hence Conjecture \ref{tam}) for $\bigoplus_{i=0}^{2d-2} h^i(\X_F)[-i]$ and all $l$. Conversely, Conjecture \ref{conjTNCtrunc} for $\bigoplus_{i=0}^{2d-2} h^i(\X_F)[-i]$ and all $l$ implies the $l$-primary part of Conjecture \ref{Lichtenbaum Conjecture} f) for $\X[1/l]$ and all $l$, hence the $l$-primary part of Conjecture \ref{Lichtenbaum Conjecture} f) for $\X$ and all $l$, hence Conjecture \ref{Lichtenbaum Conjecture} f) for $\X$. Here by the $l$-primary part of Conjecture \ref{Lichtenbaum Conjecture} f) for $\X$ we mean an identity
\[\mathbb{Z}_{(l)}\cdot\lambda_{\X}(\zeta^*(\mathcal{X},0)^{-1})=\mbox{det}_{\mathbb{Z}_{(l)}}\mbox{\emph{R}}\Gamma_{W,c}(\mathcal{X},\mathbb{Z})_{\mathbb{Z}_{(l)}}\]
where $\mathbb{Z}_{(l)}$ is the localization of $\mathbb{Z}$ at the prime ideal $l\mathbb{Z}$.

It remains to check that the squares (\ref{dia0R}) and (\ref{dia4}) are indeed commutative. Consider the following diagram.
{\tiny{
\[\begin{CD}
\mydet_{\bq_l} R\Gamma_{W,c}(\X_{U},\bz)_{\bq_l}\otimes\bq_l @>1\otimes \mu_{\X_Z,\bq_l}>> \mydet_{\bq_l} R\Gamma_{W,c}(\X_{U},\bz)_{\bq_l}\otimes\mydet_{\bq_l} R\Gamma(\X_{Z,W},\bz)_{\bq_l} @>b>> \mydet_{\bq_l} R\Gamma_{W,c}(\X,\bz)_{\bq_l}
 \\
@VVa\otimes 1 V @VVa\otimes fV @VVdV\\
\bigotimes_{i}\mydet_{\bq_l}^{(-1)^i}R\Gamma_c(U_\et,V_l^i)\otimes\bq_l@>1\otimes \bigotimes_{\p,i}\tilde{\iota}_{\p}^{(-1)^{i+1}}>> 
\bigotimes_{i}\mydet_{\bq_l}R\Gamma_c(U_{\et},V_l^i)\otimes\bigotimes_{\p,i} \mydet_{\bq_l}R\Gamma_f(F_{\p},V_l^i)
@>c>> \bigotimes_{i}\Delta_f(h^i(X))_{\bq_l}^{(-1)^i}
\end{CD}\]}}
Here we identify Weil-\'etale cohomology tensor $\bq_l$ with $l$-adic \'etale cohomology. Then the left vertical isomorphism $a$ is induced by (\ref{vaservir}), the map $f$ in the central vertical isomorphism is given by (\ref{shortiso}),  and $b$, $c$ and $d$ are induced by (\ref{unedeplus}), (\ref{twodet}) and (\ref{isotriv}) respectively. The commutativity of the left square follows from the commutativity of (\ref{dia00}). The commutativity of the right square follows from Proposition \ref{prop-matthias}. By definition of $\vartheta_{W}$ we have $d\circ b\circ (1\otimes \mu_{\X_Z,\bq_l})=\vartheta_{W,\bq_l}$. By definition of $\tilde{\vartheta}_l^i$ we have $\otimes_i(\tilde{\vartheta}_l^i)^{(-1)^i}=(c\circ (1\otimes \bigotimes_{\p,i}\tilde{\iota}_{\p}^{(-1)^{i+1}})^{-1}$. The commutativity of (\ref{dia4}) follows immediately.

The square (\ref{dia0R}) can be decomposed as follows:
{\tiny{
\[\begin{CD}
\br \otimes \br @>\lambda_{\X}\otimes \lambda_{\X_Z}^{-1}>> \mydet_{\br} R\Gamma_{W,c}(\X,\bz)_{\br}\otimes\mydet^{-1}_{\br} R\Gamma(\X_{Z,W},\bz)_{\br} @>b>> \mydet_{\br} R\Gamma_{W,c}(\X_U,\bz)_{\br}
 \\
@VV 1\otimes 1 V @VV\eta\otimes \mu_{\X_{Z},\br}V @VV\vartheta_{W,\br}V\\
\br \otimes \br@>(\otimes_i(\vartheta_\infty^i)^{(-1)^i})\otimes(\prod_{\p} \mathrm{log}(N(\p)))>> (\bigotimes_{i}\Delta_f(h^i(X))_{\br}^{(-1)^i})\otimes \hspace{0.5cm}\br\hspace{2cm}
@>Id>> \bigotimes_{i}\Delta_f(h^i(X))_{\br}^{(-1)^i}
\end{CD}\]
}}
Indeed, under the canonical identification $\br\otimes\br=\br$, the composition of the top horizontal maps  (respectively of the lower horizontal maps) is $\lambda_{\X_U}$ (respectively $\otimes_i(\tilde{\vartheta}_\infty^i)^{(-1)^i}$). Here $Id$ is the obvious identification, $\eta$ is induced by (\ref{isotriv}) and $b$ is induced by (\ref{unedeplus}). The right hand side square is commutative by definition of $\vartheta_{W}$. The left hand side square is the tensor product of two squares; it is therefore enough to check commutativity of both factors separately. The commutativity of the square corresponding to the first factor follows from the fact that (\ref{isotriv}) is compatible with $\lambda_{\X}$ and $\otimes_i(\vartheta_\infty^i)^{(-1)^i}$. The commutativity of the square corresponding to the second factor boils down to the identity
\begin{equation}\label{mulambda}
\lambda^{-1}_{\X_{\p}}=\mathrm{log}(N(\p))\cdot \mu^{-1}_{\X_{\p},\br}:\mydet_{\br} R\Gamma(\X_{\p,W},\bz)_{\br}\rightarrow \br.
\end{equation}
Identity (\ref{mulambda}) follows from the fact that $\mu_{\X_{\p},\br}$ is obtained by cup-product with the class $e\in H^1(W_{\mathbb{F}_\p},\mathbb{R})=\mathrm{Hom}(W_{\mathbb{F}_\p},\br)$ sending $\mathrm{Fr}_{\p}\in W_{\mathbb{F}_\p}$ to $1$, while $\lambda_{\X_{\p}}$ is obtained by cup-product with the class $\theta\in H^1(W_{\mathbb{F}_\p},\mathbb{R})$ sending $\mathrm{Fr}_{\p}$ to $\mathrm{log}(N(\p))$. This is an easy computation in view of the fact that both complexes $[...\rightarrow H^i(\X_{\p,W},\tilde{\mathbb{R}})\rightarrow  H^i(\X_{\p,W},\tilde{\mathbb{R}})\rightarrow...]$ are isomorphic to $[\br\rightarrow \mathrm{Hom}(W_{\mathbb{F}_\p},\br)]$ put in degrees $0,1$, because $\X_{\p}/\mathbb{F}_{\p}$ is geometrically connected. Identity (\ref{mulambda}) is of course compatible with $Z^*(\X_{\p},1)=\mathrm{log}(N(\p))\cdot\zeta^*(\X_{\p},0)$. This concludes the proof of the theorem.
\end{proof}

\subsection{Proven cases}\label{subsect-proven}

Let $\mathbb{F}_q$ be a finite field and let $A(\mathbb{F}_q)$ be the class of smooth projective varieties over $\mathbb{F}_q$
defined in Section \ref{subsection-assumtion-F_q}.
\begin{thm}\label{thm-char-p}
Let $\mathcal{X}$ be a smooth projective variety over the finite field $\mathbb{F}_q$. If $\mathcal{X}$ lies in $A(\mathbb{F}_q)$ then Conjecture  \ref{Lichtenbaum Conjecture} holds for $\mathcal{X}$.
\end{thm}
\begin{proof}
The variety $\mathcal{X}$ lies in $A(\mathbb{F}_q)$ hence
$\textbf{L}(\mathcal{X}_{et},d)\Leftrightarrow \textbf{L}(\mathcal{X}_{W},d)$ holds (see Proposition \ref{prop-equiconjectures-char-p}). In view of Theorem \ref{thm-comparison-char-p} the result follows from \cite{Lichtenbaum-finite-field} Theorem 7.4.
\end{proof}

\begin{thm} If $\mathcal{X}=\mbox{\emph{Spec}}(\mathcal{O}_F)$ is the spectrum of a number ring, then Conjecture \ref{Lichtenbaum Conjecture}  holds for $\mathcal{X}$.
\end{thm}
\begin{proof}
The result follows from the explicit computation of the Weil-\'etale cohomology in this case (see Section \ref{sect-numberrings}) and from the analytic class number formula $$\mbox{ord}_{s=0}\zeta(\mathcal{X},s)=\sharp\mathcal{X}_{\infty}-1\mbox{ and } \zeta_F^*(0)=-hR/w$$
where $h$ (respectively $w$) is the order of $Cl(F)$ (respectively of $\mu_F$) and $R$ is the regulator of the number field $F$.
\end{proof}

\begin{thm}\label{cor-big}
Let $\X$ be a smooth projective scheme over the number ring $\mathcal{O}_{\X}(\X)=\mathcal{O}_F$, where $F$ is an abelian number field.  Assume that $\X_F$ admits a smooth cellular decomposition (see Definition \ref{def-cellular}) and that $\X_{\p}\in\mathcal{L}(\mathbb{Z})$ for any finite prime $\p$ of $F$. Then Conjecture \ref{Lichtenbaum Conjecture} holds for $\X$.
\end{thm}
\begin{proof} The scheme $\mathcal{X}$ is connected and we set $d=\textrm{dim}(\X)$. The scheme $\X$ satisfies Conjecture ${\bf L}(\mathcal{X}_{et},d)$ (resp. Conjecture ${\bf B}(\X,d)$) by Proposition \ref{prop-cellular-L} (resp. by Proposition \ref{prop-knowncase-B}). Moreover, $\mathcal{X}_F$ is a cellular variety hence $h(\mathcal{X}_F)$ is of the form $\bigoplus_{k} h^0(F)(r_k)$ (in the category of Chow motives, see \cite{Brosnan05} Theorem 3.1). Hence $L(h^i(\X_F),s)$ is a product of shifts of the Dedekind zeta function $\zeta_F(s)$; in particular $L(h^i(\X_F),s)$ satisfies meromorphic continuation and the functional equation. By Theorem \ref{thm-beilinson} and \cite{Flach-moi} Theorem 1.1, Statement e) of Conjecture \ref{Lichtenbaum Conjecture} for $\mathcal{X}$ follows. Moreover, the Galois representations $H^i(\X_{\overline{F}},\bq_l)$ is a sum of $\bq_l(r)$ (with $r\leq0$), hence satisfy Conjecture \ref{bk} (see \cite{Benois-Nguyen-Quang-Do-02} Lemme 4.3.1). Finally, Conjecture \ref{tam} is known for $h^i(\X_F)$ (which is a sum of $h^0(F)(r)$) since $F/\bq$ is abelian \cite{Flach04}. Hence Statement f) of Conjecture \ref{Lichtenbaum Conjecture} for $\X$ follows from Theorem \ref{tamcompare}.
\end{proof}

The simplest non-trivial example of a scheme satisfying the assumptions of the previous theorem is $\mathbb{P}^n_{\mathcal{O}_F}$ where $F/\mathbb{Q}$ is abelian. The result then gives a cohomological interpretation of $\zeta_F^*(n)$ for $n\leq0$ (see Section \ref{K-theory}).

\subsection{Examples}

\subsubsection{Number rings}\label{sect-numberrings}

Let $\mathcal{O}_F$ be a number ring and let $\mathcal{X}=\mbox{Spec}(\mathcal{O}_F)$. Note that ${\bf{L}}(\mathcal{X}_{et},\mathbb{Z}(1))$ holds (see Theorem \ref{thm-comparison-mot-K}). The cohomology $H^*(\overline{\X}_{et},\bz)$ is computed in (\cite{On the WE} Proposition 6.6). By Proposition \ref{finitelygenerated-cohomology}
 one has $$H^i_{W}(\overline{\mathcal{X}},\mathbb{Z})=0 \mbox{ for $i<0$ and $i>3$}.$$
By  Proposition \ref{finitelygenerated-cohomology} again one has
$H^0_{W}(\overline{\mathcal{X}},\mathbb{Z})=\mathbb{Z}$, $H^1_{W}(\overline{\mathcal{X}},\mathbb{Z})=0$, an exact sequence
\begin{equation}\label{unesuiteexacte}
0\rightarrow H^2(\overline{\mathcal{X}}_{et},\mathbb{Z})_{codiv}\rightarrow H_W^2(\overline{\mathcal{X}},\mathbb{Z})
\rightarrow \mbox{Hom}(H^1(\overline{\mathcal{X}}_{et},\mathbb{Z}(1)),\mathbb{Z})\rightarrow 0
\end{equation}
and an isomorphism
$$H_W^3(\overline{\mathcal{X}},\mathbb{Z})\simeq H^3(\overline{\mathcal{X}}_{et},\mathbb{Z})_{codiv}=\textrm{Hom}(\mathcal{O}_F^{\times},\bq/\bz)_{codiv}=\mu_F^D.$$
since $H^0(\overline{\mathcal{X}}_{et},\mathbb{Z}(1))=0$.
The sequence (\ref{unesuiteexacte}) reads as follows (\cite{On the WE} Proposition 6.6.)
$$0\rightarrow Cl(F)^D\rightarrow H_W^2(\overline{\mathcal{X}},\mathbb{Z})
\rightarrow \mbox{Hom}(\mathcal{O}_F^{\times},\mathbb{Z})\rightarrow 0
$$
where $Cl(F)$ is the class group of $F$, $\mathcal{O}_F^{\times}$ is the unit group and $\mu_F:=(\mathcal{O}_F^{\times})_{tor}$. The map
$$H_{W,c}^i(\mathcal{X},\mathbb{Z})\otimes\mathbb{R}\stackrel{\sim}{\rightarrow} H_{W,c}^i(\mathcal{X},\tilde{\mathbb{R}})$$
is trivial for $i\neq1,2$, the obvious isomorphism for $i=1$ and the inverse of the dual of the classical regulator map
$$H_{W,c}^2(\mathcal{X},\mathbb{Z})\otimes\mathbb{R}
\simeq\mbox{Hom}(\mathcal{O}_F^{\times},\mathbb{R})\rightarrow (\prod_{\mathcal{X}_{\infty}}\mathbb{R})/\mathbb{R}$$
for $i=2$. The acyclic complex
$(H_{W,c}^*(\mathcal{X},\tilde{\mathbb{R}}),\cup\theta)$ is reduced to the identity map
$$H_{W,c}^1(\mathcal{X},\tilde{\mathbb{R}})=(\prod_{\mathcal{X}_{\infty}}\mathbb{R})/\mathbb{R}\stackrel{Id}{\longrightarrow}  (\prod_{\mathcal{X}_{\infty}}\mathbb{R})/\mathbb{R}=H_{W,c}^2(\mathcal{X},\tilde{\mathbb{R}}).$$
We obtain (see \cite{Lichtenbaum} Section 7)
$$\lambda_{\X}^{-1}(\mbox{det}_{\mathbb{Z}}\mbox{R}\Gamma_{W,c}(\mathcal{X},\mathbb{Z}))=(w/hR)\cdot\bz.$$

\subsubsection{Projective spaces over number rings}\label{K-theory}

Let $\mathcal{O}_F$ be the ring of integers in a totally imaginary number field  $F$ and let $n\geq1$. We set $\mathcal{X}=\mathbb{P}^{n}_{\mathcal{O}_F}$ and $d=\mathrm{dim}(\X)=n+1$. It follows easily from Proposition \ref{prop-cellular-L} (or directly from (\ref{proj-bund}) below and Theorem \ref{thm-comparison-mot-K}) that ${\bf L}(\X_{et},d)$ holds.
\begin{prop}\label{prop-K-theory}
For $2\leq i\leq2d+1$, we have an exact sequence 
$$0\rightarrow (K_{i-2}(\mathcal{O}_{F})_{tor})^D\rightarrow H_W^i(\overline{\mathcal{X}},\mathbb{Z})\rightarrow
\mathrm{Hom}_{\mathbb{Z}}(K_{i-1}(\mathcal{O}_{F}),\mathbb{Z})\rightarrow 0.$$
\end{prop}
\begin{proof}
The localization sequence (see \cite{Geisser-Duality} Corollary 7.2(a))
$$...\rightarrow H^{i-2}(\mathbb{P}^{n-1}_{\mathcal{O}_F,et},\mathbb{Z}(d-1))\rightarrow
H^{i}(\mathbb{P}^{n}_{\mathcal{O}_F,et},\mathbb{Z}(d))\rightarrow H^{i}(\mathbb{A}^n_{\mathcal{O}_F,et},\mathbb{Z}(d))\rightarrow...$$
is split by $H^{i}(\mathbb{A}^{n}_{\mathcal{O}_F,et},\mathbb{Z}(d))\stackrel{\sim}{\leftarrow}H^{i}(\mathrm{Spec}(\mathcal{O}_F)_{et},\mathbb{Z}(d))\rightarrow H^{i}(\mathbb{P}^{n}_{\mathcal{O}_F,et},\mathbb{Z}(d))$ (see Lemma \ref{lem-affine-bundle-formula}). By induction, this yields the projective bundle formula
\begin{equation}\label{proj-bund}
H^i(\overline{\mathcal{X}}_{et},\mathbb{Z}(d))\simeq H^i(\mathcal{X}_{et},\mathbb{Z}(d))\simeq \bigoplus_{j=0}^{n}H^{i-2j}(\mathrm{Spec}(\mathcal{O}_F)_{et},\mathbb{Z}(d-j)).
\end{equation}
If $0\leq j\leq n-1$, then $H^{i-2j}(\mathrm{Spec}(\mathcal{O}_F)_{et},\mathbb{Z}(d-j))=0$ for $i-2j\neq 1,2$ (by Lemma \ref{prop-fordegenerate}) and 
$$K_{2d-i}(\mathcal{O}_F)=K_{2(d-j)-(i-2j)}(\mathcal{O}_F)\simeq H^{i-2j}(\mathrm{Spec}(\mathcal{O}_F),\mathbb{Z}(d-j))\simeq H^{i-2j}(\mathrm{Spec}(\mathcal{O}_F)_{et},\mathbb{Z}(d-j))$$ for $i-2j=1,2$ by Theorem \ref{thm-comparison-mot-K}(6) and (\ref{Beilinson-Lichtenbaum}). 

We obtain
$H^{i}(\overline{\mathcal{X}}_{et},\mathbb{Z}(d))\simeq K_{2d-i}(\mathcal{O}_F)$ for $1\leq i\leq  2d-1$  and $H^{2d}(\overline{\mathcal{X}}_{et},\mathbb{Z}(d))\simeq Cl(F)=K_0(\mathcal{O}_F)_{tor}$. The result then follows from the exact sequence (see Proposition \ref{finitelygenerated-cohomology})
$$0\rightarrow H^i(\overline{\mathcal{X}}_{et},\mathbb{Z})_{codiv} \rightarrow H_W^i(\overline{\mathcal{X}},\mathbb{Z})\rightarrow \mbox{Hom}_{\mathbb{Z}}(H^{2d+2-(i+1)}(\overline{\mathcal{X}}_{et},\mathbb{Z}(d))_{\geq0},\mathbb{Z})\rightarrow 0$$
and from $H^i(\overline{\mathcal{X}}_{et},\mathbb{Z})_{codiv}\simeq (H^{2d+2-i}(\overline{\mathcal{X}}_{et},\mathbb{Z}(d))_{\geq0,tor})^D$ for $i\geq1$  (see Lemma \ref{prop-duality-Z-coefs}). 

\end{proof}

Recall from \cite{Borel74} that $K_i(\mathcal{O}_F)$ is finitely generated for any $i\geq0$ and finite for $i\neq0$ even. By Proposition \ref{prop-K-theory} and Proposition \ref{finitelygenerated-cohomology}, $H_W^*(\overline{\mathcal{X}},\mathbb{Z})$ is given by the following identifications and exact sequences:
\begin{eqnarray*}
H_W^i(\overline{\mathcal{X}},\mathbb{Z})&=&\mathbb{Z}\mbox{ for }i=0,\\
&=&0\mbox{ for }i=1,\\
0\rightarrow Cl(F)^D\rightarrow H_W^2(\overline{\mathcal{X}},\mathbb{Z})&\rightarrow&\mbox{Hom}(\mathcal{O}_F^{\times},\mathbb{Z})\rightarrow 0,\\
&=&\mu_{F}^D\mbox{ for }i=3,\\
0\rightarrow K_2(\mathcal{O}_F)^D\rightarrow H_W^4(\overline{\mathcal{X}},\mathbb{Z})&\rightarrow&\mbox{Hom}(K_3(\mathcal{O}_F),\mathbb{Z})\rightarrow 0,\\
&=&(K_3(\mathcal{O}_F)_{tor})^D\mbox{ for }i=5,\\
&...&\\
0\rightarrow K_{2n}(\mathcal{O}_F)^D\rightarrow H_W^{2n+2}(\overline{\mathcal{X}},\mathbb{Z})&\rightarrow&\mbox{Hom}(K_{2n+1}(\mathcal{O}_F),\mathbb{Z})\rightarrow 0,\\
&=&(K_{2n+1}(\mathcal{O}_F)_{tor})^D\mbox{ for }i=2n+3,\\
&=&0\mbox{ for }i>2n+3.
\end{eqnarray*}
The long exact sequence
$$...\rightarrow H_{W,c}^i(\mathcal{X},\mathbb{Z})\otimes\mathbb{R}\rightarrow H_{W}^i(\overline{\mathcal{X}},\mathbb{Z})\otimes\mathbb{R} \rightarrow H^i(\mathcal{X}_{\infty},\mathbb{Z})\otimes\mathbb{R}\rightarrow ...$$
and Proposition \ref{propiinftydecompos} show that, for $i\geq2$, $H_{W,c}^i(\mathcal{X},\mathbb{Z})\otimes\mathbb{R}$ is canonically isomorphic to either $H_{W}^i(\overline{\mathcal{X}},\mathbb{Z})\otimes\mathbb{R}$ or $H^{i-1}(\mathcal{X}_{\infty},\mathbb{Z})\otimes\mathbb{R}$ depending on the parity of $i$. Hence the acyclic  complex $$...\stackrel{\cup\theta}{\longrightarrow} H_{W,c}^i(\mathcal{X},\mathbb{Z})\otimes\mathbb{R}\stackrel{\cup\theta}{\longrightarrow} H_{W,c}^{i+1}(\mathcal{X},\mathbb{Z})\otimes\mathbb{R}\stackrel{\cup\theta}{\longrightarrow}...$$
is canonically isomorphic to
$$0\longrightarrow(\prod_{F\hookrightarrow\mathbb{C}}\mathbb{R})^{G_{\br}}/\mathbb{R}\stackrel{R_0^*}{\longrightarrow}\mbox{Hom}(\mathcal{O}_F^{\times},\mathbb{R})
\stackrel{0}{\longrightarrow}(\prod_{F\hookrightarrow\mathbb{C}}(2i\pi)^{-1}\mathbb{R})^{G_{\br}}\stackrel{R_1^*}{\longrightarrow}
\mbox{Hom}(K_3(\mathcal{O}_F),\mathbb{R})
\stackrel{0}{\longrightarrow}...$$ $$...\stackrel{0}{\longrightarrow}(\prod_{F\hookrightarrow\mathbb{C}}(2i\pi)^{-n}\mathbb{R})^{G_{\br}}\stackrel{R_n^*}{\longrightarrow}
\mbox{Hom}(K_{2n+1}(\mathcal{O}_F),\mathbb{R})
\stackrel{0}{\longrightarrow}0$$
where the isomorphisms $R_m^*$ are dual to the regulator maps.
It follows (see \cite{Lichtenbaum} Section 7) that $\lambda_{\X}^{-1}$ maps $\mbox{det}_{\mathbb{Z}}\mbox{R}\Gamma_{W,c}(\mathcal{X},\mathbb{Z})$ to
$$(w/hR_0)\cdot(\sharp K_{3}(\mathcal{O}_F)_{tor}/\sharp K_{2}(\mathcal{O}_F)\cdot R_1)\cdot...\cdot( \sharp  K_{2n+1}(\mathcal{O}_F)_{tor}/\sharp K_{2n}(\mathcal{O}_F)\cdot R_{n})\cdot\bz$$
where $R_m$ is the Beilinson regulator (we follow the indexing of \cite{Lichtenbaum73}). In view of $$\zeta^*(\mathbb{P}^n_{\mathcal{O}_F},0)=\zeta^*_F(0)\cdot\zeta^*_F(-1)\cdot ... \cdot \zeta^*_F(-n)$$
we see that Conjecture \ref{Lichtenbaum Conjecture} for $\X=\mathbb{P}^n_{\mathcal{O}_F}$ gives a cohomological reformulation of the classical version of Lichtenbaum's conjecture (see \cite{Lichtenbaum73} 4.2). Note that Weil-\'etale cohomology  gives (conjecturally) the right $2$-torsion for any number field, i.e. possibly with some real places (see \cite{Lichtenbaum73} 2.6). For example, if $F/\bq$ is abelian, then Conjecture \ref{Lichtenbaum Conjecture} holds for $\X=\mathbb{P}^n_{\mathcal{O}_F}$ and any $n\geq0$ by Theorem \ref{cor-big}.

\section{Appendix}
The goal of this section is to establish simple cases of Conjectures $\textbf{L}(\mathcal{X}_{et},d)$ and $\textbf{B}(\mathcal{X},d)$ which are used in Section \ref{subsect-proven}.

\subsection{Motivic cohomology of number rings} Let $F$ be a number field and let $\mathcal{O}_F$ be its ring of integers. In order to ease the notations, in this subsection we denote by $H^i(\mathcal{O}_F,\mathbb{Z}(n)):=H^p(\mathrm{Spec}(\mathcal{O}_F)_{Zar},\mathbb{Z}(n))$ and $H^i_{et}(\mathcal{O}_F,\mathbb{Z}(n)):=H^p(\mathrm{Spec}(\mathcal{O}_F)_{et},\mathbb{Z}(n))$ Zariski and \'etale hypercohomology of the cycle complex $\mathbb{Z}(n)$ over $\mathrm{Spec}(\mathcal{O}_F)$. We consider the spectral sequence constructed by Levine (see \cite{Levine-localization} Spectral Sequence (8.8))
\begin{equation}\label{ss}
E_2^{p,q}=H^p(\mathcal{O}_F,\mathbb{Z}(-q/2))\Rightarrow K_{-p-q}(\mathcal{O}_F).
\end{equation}
The aim of this subsection is to prove Theorem \ref{thm-comparison-mot-K}. It is well-known (see for example \cite{Kolster-Sands08} Proposition 2.1) that this result follows one way or another from the Bloch-Kato conjecture (relating Milnor K-theory to Galois cohomology), which is proven in \cite{Voevodsky11}. However, we could not find a proof of Theorem \ref{thm-comparison-mot-K} in the literature.
\begin{thm}\label{thm-comparison-mot-K} The following assertions are true.
\begin{enumerate}
\item For $i\leq2$, $H_{et}^i(\mathcal{O}_F,\mathbb{Z}(1))$ is finitely generated. 
\item For any $n\geq0$ and any $i\in\mathbb{Z}$, $H^i(\mathcal{O}_F,\mathbb{Z}(n))$ is finitely generated. 
\item For any $n\geq2$ and any $i\in\mathbb{Z}$, $H^i_{et}(\mathcal{O}_F,\mathbb{Z}(n))$ is finitely generated. 

\item The edge morphisms from the spectral sequence (\ref{ss}) yield maps
$$c_{i,n}:K_{2n-i}(\mathcal{O}_F)\longrightarrow H^i(\mathcal{O}_F,\mathbb{Z}(n))$$
for $n\geq2$ and $i=1,2$. 
\item The kernel and the cokernel of $c_{i,n}$ are both finite and $2$-primary torsion. 
\item If  $F$ is totally imaginary, then the maps $c_{i,n}$ are isomorphisms. 
\end{enumerate}

\end{thm}

\begin{proof} The proof of Theorem \ref{thm-comparison-mot-K} requires the following lemmas.

\begin{lem}\label{lemZ-Q/Z}
Let $n\geq2$. If $i\neq 1,2$, then 
$$H^i_{et}(\mathcal{O}_F,\mathbb{Z}(n))\simeq H^{i-1}_{et}(\mathcal{O}_F,
\mathbb{Q}/\mathbb{Z}(n)).$$
\end{lem}
\begin{proof}
Applying the exact functor $-\otimes_{\mathbb{Z}}\mathbb{Q}$ to the spectral sequence (\ref{ss}), we obtain a spectral sequence with $\mathbb{Q}$-coefficients. By \cite{Levine-motivic-and-K-theory} Theorem 11.7, (\ref{ss}) degenerates with $\mathbb{Q}$-coefficients and yields
$$H^i(\mathcal{O}_F,\mathbb{Q}(n))\simeq K_{2n-i}(\mathcal{O}_F)^{(n)}_{\mathbb{Q}}.$$
By Borel's theorem \cite{Borel74}, $K_{2n-i}(\mathcal{O}_F)_{\mathbb{Q}}=0$ for $i$ even, $i\neq 2n$. Moreover, one has $K_{2n-i}(\mathcal{O}_F)_{\mathbb{Q}}^{(n)}=0$ for $i$ odd, $i\neq1$ (see \cite{Weibel-Handbook} Theorem 47). Since $K_{0}(\mathcal{O}_F)_{\mathbb{Q}}=K_{0}(\mathcal{O}_F)_{\mathbb{Q}}^{(0)}$, this
gives $H^i(\mathcal{O}_F,\mathbb{Q}(n))=0$ for $i$ even, except for $(i,n) = (0,0)$, and 
$H^i(\mathcal{O}_F,\mathbb{Q}(n))=0$ for $i$ odd, except for $i=1$.
Lemma \ref{lemZ-Q/Z} then follows from the long exact sequence
$$...\rightarrow H_{et}^i(\mathcal{O}_F,\mathbb{Z}(n))\rightarrow H_{et}^i(\mathcal{O}_F,\mathbb{Q}(n))\rightarrow H^i_{et}(\mathcal{O}_F,\mathbb{Q}/\mathbb{Z}(n))\rightarrow...$$
and from the isomorphism $H^i(\mathcal{O}_F,\mathbb{Q}(n))\simeq H_{et}^i(\mathcal{O}_F,\mathbb{Q}(n))$ (see \cite{Geisser-MotvicCohDed} Proposition 3.6).
\end{proof}

For a prime number $p$, we write $j_p:\mathrm{Spec}(\mathcal{O}_F[1/p])\rightarrow \mathrm{Spec}(\mathcal{O}_F)$ for the obvious open embedding, $j_{p,*}$ for the direct image functor with respect to the \'etale topology and $Rj_{p,*}$ for its total derived functor.

\begin{lem}\label{n>1}
For $n\geq2$ one has an isomorphism in the derived category of \'etale sheaves on $\mathrm{Spec}(\mathcal{O}_F)$:
$$\mathbb{Q}/\mathbb{Z}(n)\simeq \bigoplus_p \underrightarrow{\mathrm{lim}}\, Rj_{p,*}\mu_{p^{r}}^{\otimes n}$$
where $\mu_{p^r}$ is the \'etale sheaf of $p^r$-th roots of unity and the direct sum (respectively the colimit) is taken over all prime numbers (respectively over $r$).

\end{lem}
\begin{proof}
It is enough to showing that $\mathbb{Z}/p^{r}\mathbb{Z}(n)=Rj_{p,*}\mu_{p^{r}}^{\otimes n}$ for any prime number $p$. Using \cite{Geisser-Duality} Corollary 7.2, \cite{Geisser-Levine00} Theorem 8.5 and \cite{Geisser-MotvicCohDed} Theorem 1.2(4), we obtain an exact triangle
$$ i_{p,*}(\nu_r^{n-1})[-(n-1)-2]\rightarrow\mathbb{Z}/p^{r}\mathbb{Z}(n)\rightarrow Rj_{p,*}(\mu_{p^{r}}^{\otimes n})$$
where $\nu_r^{n}=\nu_{p^r}^{n}=W\Omega^n_{r,log}$ is the logarithmic de Rham-Witt sheaf, and $i_p$ is the closed immersion of the points lying over $p$.  But $\nu_r^{n-1}$ is trivial on the finite field $\mathbb{F}_{\mathfrak{p}}:=\mathcal{O}_F/\p$ (for $\mathfrak{p}\mid p$) because $n-1\geq 1$. The result then follows from $\mathbb{Q}/\mathbb{Z}(n)\simeq \bigoplus_p \underrightarrow{\mathrm{lim}}\,\mathbb{Z}/p^{r}\mathbb{Z}(n)$.
\end{proof}

\begin{lem}\label{prop-fordegenerate} The following assertions are true.
\begin{itemize}
\item If $n\geq1$ and $i\leq 0$, then we have $H_{et}^i(\mathcal{O}_F,\mathbb{Z}(n))=0$. 
\item If $n\geq2$ and $i\geq3$, then $H^i_{et}(\mathcal{O}_F,\mathbb{Z}(n))$ is finite and $2$-torsion. 
\item Assume that $F$ is totally imaginary. If $n\geq2$ and $i\geq3$, then $H_{et}^i(\mathcal{O}_F,\mathbb{Z}(n))=0$. 
\end{itemize}
\end{lem}
\begin{proof}
In view of $\mathbb{Z}(1)\simeq\mathbb{G}_m[-1]$ where $\mathbb{G}_m$ is the multiplicative group (see \cite{Geisser-Duality} Lemma 7.4), the fact that $H^i_{et}(\mathcal{O}_F,\mathbb{Z}(n))=0$ for $n\geq1$ and $i\leq 0$ follows immediately from Lemma \ref{lemZ-Q/Z} and Lemma \ref{n>1}.

We assume now that $n\geq2$. By \cite{Soule79} Theorem 5, the group 
$H_{et}^2(\mathcal{O}_F[1/p],\underrightarrow{\mathrm{lim}}\, \mu_{p^{r}}^{\otimes n})$ is of finite exponent  for any $p\neq 2$ (the colimit is taken over $r$). But $H_{et}^3(\mathcal{O}_F[1/p],\mu_{p^{r}}^{\otimes n})=0$ for $p\neq 2$ by Artin-Verdier duality (see \cite{Milne-duality} Corollary 3.3), hence $H_{et}^2(\mathcal{O}_F[1/p],\underrightarrow{\mathrm{lim}}\, \mu_{p^{r}}^{\otimes n})$ is divisible. This yields $H_{et}^2(\mathcal{O}_F[1/p],\underrightarrow{\mathrm{lim}}\, \mu_{p^{r}}^{\otimes n})=0$ for $p\neq 2$, since this group is both divisible and of finite exponent. Moreover, we have
$$H_{et}^i(\mathcal{O}_F[1/p],\underrightarrow{\mathrm{lim}}\, \mu_{p^{r}}^{\otimes n})=0$$
for any $i\geq 3$ and $p\neq2$, again by Artin-Verdier duality. By \cite{Rognes-Weibel00} Corollary 4.4 and \cite{Rognes-Weibel00} Proposition 4.6, for any $i\geq2$, we have  
$H_{et}^i(\mathcal{O}_F[1/2],\underrightarrow{\mathrm{lim}}\, \mu_{2^{r}}^{\otimes n})=(\mathbb{Z}/2\mathbb{Z})^{r_1}$ for $n-i$ odd and
$H^i_{et}(\mathcal{O}_F[1/2],\underrightarrow{\mathrm{lim}}\, \mu_{2^{r}}^{\otimes n})=0$ for  $n-i$ even, where $r_1$ is the number of real places of $F$.

Using $H_{et}^i(\mathcal{O}_F,\mathbb{Q}/\mathbb{Z}(n))\simeq \bigoplus_p \underrightarrow{\mathrm{lim}}\,H_{et}^i(\mathcal{O}_F[1/p],\mu_{p^{r}}^{\otimes n})$ (see Lemma \ref{n>1}), it then follows that $H_{et}^i(\mathcal{O}_F,\mathbb{Q}/\mathbb{Z}(n))$ is finite $2$-torsion for $i\geq 2$, and $H_{et}^i(\mathcal{O}_F,\mathbb{Q}/\mathbb{Z}(n))=0$ for $F$ totally imaginary and $i\geq2$. We obtain the result using Lemma \ref{lemZ-Q/Z}.

\end{proof}

\hspace{-0.5cm}\emph{Proof of Theorem \ref{thm-comparison-mot-K}.} The finite generation of $H^i_{et}(\mathcal{O}_F,\mathbb{Z}(1))$ for $i\leq2$ follows from $\mathbb{Z}(1)\simeq\mathbb{G}_m[-1]$ (see \cite{Geisser-Duality} Lemma 7.4). Indeed, $H^i_{et}(\mathcal{O}_F,\mathbb{G}_m)=0$ for $i<0$, $H^0_{et}(\mathcal{O}_F,\mathbb{G}_m)=\mathcal{O}_F^{\times}$ is  finitely generated and the class group $H^1_{et}(\mathcal{O}_F,\mathbb{G}_m)=Cl(F)$ is finite.

By \cite{Voevodsky11} Theorem 6.16 and by \cite{Geisser-MotvicCohDed} Theorem 1.2(2), the map
\begin{equation}\label{Beilinson-Lichtenbaum}
H^i(\mathcal{O}_F,\mathbb{Z}(n))\stackrel{\sim}{\longrightarrow} H_{et}^i(\mathcal{O}_F,\mathbb{Z}(n)),
\end{equation}
induced by the morphism from the \'etale site to the Zariski site, is an isomorphism for $i\leq n+1$. Note that, for dimension reasons, we have $H^i(\mathcal{O}_F,\mathbb{Z}(n))=0$ for $i>n+1$. By Lemma \ref{prop-fordegenerate}, we have
$H^i(\mathcal{O}_F,\mathbb{Z}(n))=0$ for $n\geq1$ and $i\leq0$. It follows that the edge morphisms from the spectral sequence (\ref{ss}) give maps
$$c_{i,n}:K_{2n-i}(\mathcal{O}_F)\longrightarrow H^i(\mathcal{O}_F,\mathbb{Z}(n))$$
for $n\geq2$ and $i=1,2$. This yields assertion (4) of Theorem \ref{thm-comparison-mot-K}.

Let us prove assertion (2) of Theorem \ref{thm-comparison-mot-K}. The codomain of any non-trivial differential $d_{r}^{p,q}$ of the spectral sequence (\ref{ss}) at the $E_r$-page for $r\geq2$ is a finite $2$-torsion abelian group. Moreover, for $(p,q)$ fixed, the differential $d_{r}^{p,q}$ is zero for $r>>0$ big enough. Since $E_{\infty}^{p,q}$ is a sub-quotient of some $K$-group, it is finitely generated by \cite{Borel74}. It follows that $E_2^{p,q}$ is finitely generated for any $p$ and any $q$. Hence $H^i(\mathcal{O}_F,\mathbb{Z}(n))$ is finitely generated for any $i$ and any $n$.

If $F$ is totally imaginary, the spectral sequence (\ref{ss}) degenerates by Lemma \ref{prop-fordegenerate} (any differential at the $E_2$-page has either trivial domain or trivial codomain), and the maps $c_{i,n}$ are isomorphisms. This proves assertion (6) of Theorem \ref{thm-comparison-mot-K}. For an arbitrary number field $F$, the spectral sequence (\ref{ss}) degenerates with $\mathbb{Z}[1/2]$-coefficients  by Lemma \ref{prop-fordegenerate}, and yields isomorphisms
$$c_{i,n}\otimes\mathbb{Z}[1/2]: K_{2n-i}(\mathcal{O}_F)\otimes\mathbb{Z}[1/2]\stackrel{\sim}{\longrightarrow} H^i(\mathcal{O}_F,\mathbb{Z}(n))\otimes\mathbb{Z}[1/2]$$
for $n\geq2$ and $i=1,2$. This proves assertion (5) of Theorem \ref{thm-comparison-mot-K}.

Finally, $H^i_{et}(\mathcal{O}_F,\mathbb{Z}(n))$ is finitely generated for $n\geq2$ and any $i\in\mathbb{Z}$. Indeed, (\ref{Beilinson-Lichtenbaum}) is an isomorphism for $i\leq n+1$ and $H^i_{et}(X,\mathbb{Z}(n))$ is finite for $i\geq n+2\geq4$ by Lemma \ref{prop-fordegenerate}. This gives assertion (3) of Theorem \ref{thm-comparison-mot-K}.

\end{proof}

\begin{rem}
Using arguments of Levine \cite{Levine-motivic-and-K-theory}, one can determine most of the differentials of the spectral sequence (\ref{ss}) (in particular (\ref{ss}) degenerates at $E_4$), and obtain precise information on the kernel and the cokernel of the map $c_{i,n}$ (see \cite{Levine-motivic-and-K-theory} Theorem 14.10). 
\end{rem}

\subsection{Smooth projective varieties over finite fields}\label{subsection-assumtion-F_q}

For a smooth projective scheme  $Y$ over a finite field $\mathbb{F}_q$, we consider the cohomology $H^i(Y_W,\mathbb{Z}(n))$ of the Weil-\'etale topos $Y_W$ with coefficients in Bloch's cycle complex (see \cite{Lichtenbaum-finite-field} and \cite{Geisser-Weiletale}). The following conjecture is due to Lichtenbaum and Geisser.
\begin{conj}
$\textbf{\emph{L}}(Y_W,n)$ For any $i\in\bz$, $H^i(Y_W,\mathbb{Z}(n))$ is finitely generated.
\end{conj}

Following \cite{Soul84} and \cite{Geisser-Weiletale}, we consider the full subcategory $A(\mathbb{F}_q)$ of the category of smooth projective varieties over $\mathbb{F}_{q}$ generated by products of curves and the following operations:

(1) If $X$ and $Y$ are in $A(\mathbb{F}_q)$ then $X\coprod Y$ is in $A(\mathbb{F}_q)$.

(2) If $Y$ is in $A(\mathbb{F}_q)$ and there are morphisms $c:X\rightarrow Y$ and $c':Y\rightarrow X$ in the category of Chow motives, such that $c'\circ c:X\rightarrow X$ is multiplication by a constant, then $X$ is in $A(\mathbb{F}_q)$.

(3) If $\mathbb{F}_{q^m}/\mathbb{F}_{q}$ is a finite extension and $X\times_{\mathbb{F}_{q}}\mathbb{F}_{q^m}$ is in $A(\mathbb{F}_{q^m})$, then $X$ is in $A(\mathbb{F}_{q})$.

(4) If $Y$ is a closed subscheme of $X$ with $X$ and $Y$ in $A(\mathbb{F}_{q})$, then the blow-up $X'$ of $X$ along $Y$ is in $A(\mathbb{F}_{q})$.

\begin{prop}\label{prop-equiconjectures-char-p}
Let $Y$ be a connected smooth projective scheme over a finite field $\mathbb{F}_q$ of dimension $d$. The following assertions are true.
\begin{itemize}
\item We have
$\textbf{\emph{L}}(Y_W,d)\Leftrightarrow \textbf{\emph{L}}(Y_{et},d)$. 
\item If $Y$ belongs to $A(\mathbb{F}_{q})$ then $\textbf{\emph{L}}(Y_{et},d)$ holds.
\item If $Y$ belongs to $A(\mathbb{F}_{q})$ and $n>d$, then $H^i(Y_{et},\mathbb{Z}(n))$ is finitely generated for any $i\in\mathbb{Z}$.
\end{itemize}
\end{prop}
\begin{proof}
By \cite{Geisser-Weiletale} Theorem 7.1, there is an exact sequence (for any $n$)
\begin{equation}\label{geisser-sequence}
...\rightarrow H^i(Y_{et},\mathbb{Z}(n))\rightarrow H^i(Y_{W},\mathbb{Z}(n))\rightarrow H^{i-1}(Y_{et},\mathbb{Q}(n))\rightarrow H^{i+1}(Y_{et},\mathbb{Z}(n))\rightarrow...
\end{equation}
which yields isomorphisms
\begin{equation}\label{geisser-splitting}
H^i(Y_W,\mathbb{Z}(n))\otimes\mathbb{Q}\simeq H^i(Y_W,\mathbb{Q}(n))\simeq H^i(Y_{et},\mathbb{Q}(n))\oplus H^{i-1}(Y_{et},\mathbb{Q}(n)).
\end{equation}
Assume now that Conjecture $\textbf{L}(Y_{W},d)$ holds. By \cite{Geisser-Weiletale} Theorem 8.4, it follows from $\textbf{L}(Y_{W},d)$ that we have an isomorphism
$$H^i (Y_W,\mathbb{Z}(d))\otimes\mathbb{Z}_l\simeq H^i_{cont}(Y,\mathbb{Z}_l(d))$$
for any prime number $l$ and any $i$. But for $i\neq 2d, 2d+1$, $H^i_{cont}(Y,\mathbb{Z}_l(d))$ is finite for any $l$ and zero for almost all $l$ (see \cite{Kahn-equivalence-rationnelle} Proof of Corollaire 3.8 for references). Hence $H^i(Y_W,\mathbb{Z}(d))$ is finite for $i\neq 2d,2d+1$. Then (\ref{geisser-splitting}) gives $H^i(Y_{et},\mathbb{Q}(d))=0$ for $i<2d$, hence $H^i(Y_{et},\mathbb{Q}(d))\simeq H^i(Y,\mathbb{Q}(d))=0$ for $i\neq2d$. The exact sequence (\ref{geisser-sequence}) then shows that $H^i(Y_{et},\mathbb{Z}(d))\rightarrow H^i(Y_{W},\mathbb{Z}(d))$ is injective for $i\leq 2d+1$.
Hence $H^i(Y_{et},\mathbb{Z}(d))$ is finitely generated for $i\leq 2d+1$. This yields $\textbf{L}(Y_{W},d)\Rightarrow \textbf{L}(Y_{et},d)$. Conversely, we have $\textbf{L}(Y_{et},d)\Rightarrow \textbf{L}(Y_{W},d)$ by Theorem \ref{thm-comparison-char-p}.

Consider now a variety $Y\in A(\mathbb{F}_{q})$ of pure dimension $d$. The fact that $\textbf{L}(Y_W,d)$ holds is given by  \cite{Geisser-Weiletale} Theorem 9.5. Let $n>d$. It follows from the proofs of \cite{Geisser-Weiletale} Theorems 9.4 and 9.5 that $H^i(Y,\mathbb{Q}(n))=0$ for any $i<2n$. Moreover $H^i(Y,\mathbb{Q}(n))=0$ for any $i\geq 2n> n+d$. The last claim of the proposition then follows from (\ref{geisser-sequence}) since $H^i(Y_{W},\mathbb{Z}(n))$ is known to be finitely generated for any $i\in\mathbb{Z}$ by \cite{Geisser-Weiletale} Theorem 9.5.

\end{proof}

\subsection{The class $\mathcal{L}(\mathbb{Z})$}

In this section we follow Geisser's notation (see \cite{Geisser-Duality} Section 2) for the cycle complex $\mathbb{Z}^c(n)$. If $\mathcal{X}$ is a $d$-dimensional connected scheme which is proper over $\mathbb{Z}$ then we have 
\begin{equation}\label{Zc-vers-Z}
\mathbb{Z}^c(n)=\mathbb{Z}(d-n)[2d].
\end{equation}
\begin{defn} \label{conj-clef-c}
Let $\mathcal{X}$ be a separated scheme of finite type over $\Spec(\mathbb{Z})$. We say that $\mathcal{X}$ satisfies \emph{$\textbf{L}^c(\mathcal{X}_{et})$} if one has:
\begin{itemize}
\item $H^i(\mathcal{X}_{et},\mathbb{Z}^c(0))$ is finitely generated for any $i\leq 0$;
\item $H^i(\mathcal{X}_{et},\mathbb{Z}^c(n))$ is finitely generated for any $i\in\mathbb{Z}$ and any $n<0$.
\end{itemize}
\end{defn}
Note that if $\mathcal{X}$ is a regular scheme connected of dimension $d$ which is proper over $\mathbb{Z}$, then $\textbf{L}^c(\mathcal{X}_{et})\Rightarrow \textbf{L}(\mathcal{X}_{et},d)$ (see (\ref{Zc-vers-Z}) above). We define below a class of (simple) arithmetic schemes satisfying Property $\textbf{L}^c(\mathcal{X}_{et})$. Let $SFT(\mathbb{Z})$ be the category of separated schemes of finite type over $\Spec(\mathbb{Z})$.

\begin{defn}\label{def-C(Z)}
We denote by $\mathcal{L}(\mathbb{Z})$ the class of schemes of $SFT(\bz)$ generated by the following objects:
\begin{itemize}
\item the empty scheme $\emptyset$;
\item varieties $Y\in A(\mathbb{F}_{q})$ for any finite field $\mathbb{F}_{q}$;
\item spectra of number rings $\Spec(\mathcal{O}_F)$;
\end{itemize}
and the following operations:
\begin{itemize}
\item $(\mathcal{L}0)$ Let $Z\hookrightarrow X$ be a closed immersion with open complement $U$ such that $Z$ is regular and proper. If two object of $(Z,X,U)$ belong to $\mathcal{L}(\bz)$ then so does the third.
\item $(\mathcal{L}1)$ Let $Z\hookrightarrow X$ be a closed immersion with open complement $U\in \mathcal{L}(\mathbb{Z})$. Then $X\in \mathcal{L}(\mathbb{Z})$ if and only if $Z\in \mathcal{L}(\mathbb{Z})$.
\item $(\mathcal{L}2)$ We have $X_i\in \mathcal{L}(\mathbb{Z})$ for $0\leq i\leq p$ if and only if $\coprod_{0\leq i\leq p}X_i\in \mathcal{L}(\mathbb{Z})$.
\item $(\mathcal{L}3)$ If $V\rightarrow U$ is an affine bundle and $U$ belongs to $\mathcal{L}(\mathbb{Z})$, then so does $V$.
\item $(\mathcal{L}4)$ Let $\{U_i\rightarrow X,\,i\in I\}$ be a finite surjective family of \'etale morphisms. If $U_{i_0,...,i_p}$  belongs to $\mathcal{L}(\mathbb{Z})$ for any  $(i_0,...,i_p)\in I^{p+1}$ and any $p\geq0$, then so does $X$.
\end{itemize}
\end{defn}
In the statement of $(\mathcal{L}4)$, we write $U_{i_0,...,i_p}:=U_{i_0}\times_X... \times_X U_{i_p}$ as usual. In practice, we use $(\mathcal{L}4)$ for a finite \'etale Galois cover $U\rightarrow X$, in which case it is enough to check that $U\in \mathcal{L}(\bz)$. 

\begin{prop}\label{prop-LZimpliesL}
Any object $\mathcal{X}$ in the class $\mathcal{L}(\mathbb{Z})$ satisfies \emph{$\textbf{L}^c(\mathcal{X}_{et})$}.
\end{prop}
We say that the property ${\bf L}^c$ is stable under operation $(\mathcal{L}i)$, if any scheme $X\in SFT(\bz)$ constructed out of schemes $X_{\alpha}$ satisfying $\textbf{L}^c(X_{\alpha,et})$ by operation $(\mathcal{L}i)$ also satisfies $\textbf{L}^c(X_{et})$.
\begin{proof}
By Theorem \ref{thm-comparison-mot-K} (see also (\ref{Zc-vers-Z})), any number ring $\Spec(\mathcal{O}_F)$ satisfies $\textbf{L}^c(\Spec(\mathcal{O}_F)_{et})$. By Proposition \ref{prop-equiconjectures-char-p}, any variety $Y\in A(\mathbb{F}_{q})$ satisfies $\textbf{L}^c(Y_{et})$. It remains to check that the property ${\bf L}^c$ is stable under operations $(\mathcal{L}i)$ for $i=0,...,4$. By purity (see \cite{Geisser-Duality} Corollary 7.2), we have a long exact sequence
$$...\rightarrow H^{i}(Z_{et},\mathbb{Z}^c(n))\rightarrow H^{i}(X_{et},\mathbb{Z}^c(n))\rightarrow H^{i}(U_{et},\mathbb{Z}^c(n))\rightarrow H^{i+1}(Z_{et},\mathbb{Z}^c(n))\rightarrow...$$
for any open--closed decomposition $U\hookrightarrow X\hookleftarrow Z$ and any $n\leq 0$. Moreover, if $Z$ is regular proper, then $H^{1}(Z_{et},\mathbb{Z}^c(0))$ is finitely generated. Indeed $H^{1}(Z_{et},\mathbb{Z}^c(0))=0$ if $Z(\br)=\emptyset$ and $H^{1}(Z_{et},\mathbb{Z}^c(0))$ is a finite dimensional $\bz/2\bz$-vector space otherwise (see Lemma \ref{firstLemma}).  It follows that the property ${\bf L}^c$ is stable under operations $(\mathcal{L}0)$ and $(\mathcal{L}1)$.

The fact that ${\bf L}^c$ is stable under operation $(\mathcal{L}2)$ is obvious since cohomology respects finite direct sums. It is stable under operation $(\mathcal{L}3)$ by Lemma \ref{lem-affine-bundle-formula}.

Let $\{U_i\rightarrow X,\,i\in I\}$ be a finite \'etale covering family.  We write $X_p=\coprod_{(i_0,...,i_p)\in I^{p+1}} U_{i_0,...,i_p}$ for $p\geq0$. The Cartan-Leray spectral sequence
\begin{equation}\label{onespectsequ}
E_1^{p,q}=H^q(X_{p,et},\mathbb{Z}^c(n))\Longrightarrow H^{p+q}(X_{et},\mathbb{Z}^c(n))
\end{equation}
converges by Lemma \ref{lem-vanishing}. Indeed Lemma \ref{lem-vanishing} implies that $H^q(X_{p,et},\mathbb{Z}^c(n))$ is a $\bq$-vector space for $q<-2\cdot \textrm{dim}(X)$, which must be trivial since $H^q(X_{p,et},\mathbb{Z}^c(n))$ is assumed to be a finitely generated abelian group. The spectral sequence (\ref{onespectsequ}) then shows that ${\bf L}^c$ is stable under operation $(\mathcal{L}4)$.

\end{proof}

\begin{lem}\label{lem-affine-bundle-formula} Let $X$ be separated of finite type over $\Spec(\bz)$ and let
 $f:\mathbb{A}^{r}_{X,et}\rightarrow X_{et}$ be the natural map. Then one has
$Rf_*\mathbb{Z}^c(n)\simeq\mathbb{Z}^c(n-r)[2r]$ for any $n\leq0$.
\end{lem}
\begin{proof}
Since $\mathbb{Z}^c(n)$ satisfies \'etale cohomological descent for $n\leq0$ (see \cite{Geisser-Duality} Theorem 7.1), one is reduced to show the analogous statement for the Zariski topology. Using \cite{Geisser-MotvicCohDed} Corollary 3.4, the result follows from the homotopy formula $Rp_*\mathbb{Z}^c(n)\simeq\mathbb{Z}^c(n-1)[2]$, where $p:\mathbb{A}^{1}_{Y,Zar}\rightarrow Y_{Zar}$ is the natural map and $Y$ is defined over a field (see \cite{Geisser-MotvicCohDed} Corollary 3.5).
\end{proof}
One defines $\bz/m\bz^c(n)=\bz^c(n)\otimes \bz/m\bz\simeq \bz^c(n)\otimes^L \bz/m\bz$ (since $\bz^c(n)$ is a complex of flat sheaves) and $\bq/\bz^c(n)=\underrightarrow{\textrm{lim}}\,\bz/m\bz^c(n)$.
Note that in the situation of Lemma \ref{lem-affine-bundle-formula} one has
\begin{equation}\label{affine-forQ/Z}
Rf_*\bq/\mathbb{Z}^c(n)\simeq\bq/\mathbb{Z}^c(n-r)[2r].
\end{equation}
Indeed, applying Lemma \ref{lem-affine-bundle-formula}, the functor $-\otimes^L\bz/m\bz$ and passing to the limit, we obtain (\ref{affine-forQ/Z}).
Here we use the fact that $Rf_*$ commutes with filtered inductive limits (since $X$ is separated of finite type over $\Spec(\bz)$).

\begin{lem}\label{lem-vanishing}
Let $X$ be separated of finite type over $\Spec(\bz)$ and let $n\leq0$. Then
$$H^i(X_{et},\mathbb{Q}/\mathbb{Z}^c(n))=0$$
for $i< - 2\cdot \textrm{\emph{dim}}(X)$.
\end{lem}
\begin{proof}
Replacing $X$ with $\mathbb{A}^{-n}_X$,  we may assume $n=0$ by (\ref{affine-forQ/Z}). There exists a finite \'etale covering family $\{V_i\rightarrow\Spec(\bz)\}$ such that $V_i(\mathbb{R})=\emptyset$, hence a finite \'etale covering
family $\{U_i\rightarrow X\}$ such that $U_i$ is defined over $\Spec(\mathcal{O}_{K_i})$ for some totally imaginary number field $K_i$. If the result is known for the $U_{i_0,...,i_p}$'s,
then it follows for $X$ by the spectral sequence
$$E_1^{p,q}=H^q(X_{p,et},\bq/\mathbb{Z}^c(n))\Longrightarrow H^{p+q}(X_{et},\bq/\mathbb{Z}^c(n))$$
where $X_p=\coprod_{(i_0,...,i_p)\in I^{p+1}} U_{i_0,...,i_p}$ and $U_{i_0,...,i_p}=U_{i_0}\times_X...\times_X U_{i_p}$ is of dimension $\leq \textrm{dim}(X)$. So we may assume that $X$ is defined over $\Spec(\mathcal{O}_{K})$ for some totally imaginary number field $K$. We have
an isomorphism of finite groups (\cite{Geisser-Duality} Theorem 7.8)
\begin{equation}\label{dualityweird}
H^{1-i}(X_{et},\bz/m\bz^c(0))=H^i_c(X_{et},\bz/m\bz)^D
\end{equation}
where the right hand side is the Pontryagin dual of $H^i_c(X_{et},\bz/m\bz)$ (here $H^i_c(X_{et},-)$ denotes usual \'etale cohomology with compact support). Take a Nagata compactification of $X$ over $\Spec(\mathcal{O}_{K})$, i.e. an open immersion $j:X\hookrightarrow X'$ with dense image such that $X'$ is proper over $\Spec(\mathcal{O}_{K})$. Note that $X'(\br)=\emptyset$ hence $X'$ is of $l$-cohomological dimension $2\cdot\textrm{dim}(X)+1$ for any prime number $l$ (see \cite{SGA4} X Theorem 6.2). We obtain
$$H^i_c(X_{et},\bz/m\bz)=H^i(X'_{et},j_!\bz/m\bz)=0\textrm{ for } i>2\cdot\textrm{dim}(X)+1$$
hence $H^{i}(X_{et},\bz/m\bz^c(0))=0$ for $i<-2\cdot \textrm{dim}(X)$ by (\ref{dualityweird}). Now the result follows from
$$H^{i}(X_{et},\bq/\bz^c(0))=\underrightarrow{\textrm{lim}}\,H^{i}(X_{et},\bz/m\bz^c(0))$$
which is valid since the \'etale site of $X$ is notherian ($X$ is separated of finite type).
\end{proof}

The class $\mathcal{L}(\mathbb{Z})$ contains singular schemes. For example, it easy to see that any proper curve (possibly singular) over a finite field lies in $\mathcal{L}(\mathbb{Z})$.

\subsection{Geometrically cellular schemes}

\begin{defn}\label{def-cellular}
Let $k$ be a field and let $Y$ be a scheme separated and of finite type over $k$. We say that the $k$-scheme $Y$ has a \emph{cellular decomposition} if there exists a filtration of $Y$ by reduced closed subschemes
\begin{equation}\label{cel-decom}
Y^{\mathrm{red}}=Y_N\supseteq Y_{N-1}\supseteq ...\supseteq Y_{-1}=\emptyset
\end{equation}
such that $Y_i\setminus Y_{i-1}\simeq\mathbb{A}_{k}^{a_i}$ is $k$-isomorphic to an affine space over $k$. We say that \emph{(\ref{cel-decom})} is a \emph{smooth cellular decomposition} if $Y_i$ is moreover smooth over $k$ for any $i\geq 0$.

The $k$-scheme $Y$ is \emph{geometrically cellular} if
$Y\otimes_k\bar{k}$ has a cellular decomposition, where $\bar{k}$ is a separable closure of $k$. 

An $S$-scheme $\mathcal{X}\rightarrow S$ separated and of finite type is \emph{geometrically cellular} if the fiber $\mathcal{X}_s$ is geometrically cellular for any $s\in S$.

\end{defn}

One can easily show that a $k$-scheme $Y$ is geometrically cellular if and only if there exists a finite Galois extension $k'/k$ such that $Y\otimes_kk'$ is cellular. It follows from the proof of Proposition \ref{prop-cellular-L} below that any geometrically cellular scheme over a finite field belongs to $\mathcal{L}(\mathbb{Z})$. More generally, any geometrically cellular scheme over a number ring belongs to $\mathcal{L}(\mathbb{Z})$.

\begin{prop}\label{prop-cellular-L}
Let $\mathcal{X}\rightarrow \Spec(\mathcal{O}_F)$ be flat, separated and of finite type over a number ring $\mathcal{O}_F$, such that $\mathcal{X}_F$ is geometrically cellular. The following conditions are equivalent.
\begin{itemize}
\item For any finite prime $\mathfrak{p}$ of $F$, $\mathcal{X}_{\p}\in\mathcal{L}(\mathbb{Z})$.
\item $\X\in\mathcal{L}(\mathbb{Z})$.
\end{itemize}
 \end{prop}
Here we set $\X_F:=\X\otimes_{\mathcal{O}_F}F$ and $\X_{\p}:=\X\otimes_{\mathcal{O}_F}(\mathcal{O}_F/\p)$.
\begin{proof}
By assumption there exists a finite Galois extension $K/F$ such that $\mathcal{X}_K$ is cellular. We write
$$(\mathcal{X}\otimes_{\mathcal{O}_{F}}K)^{\mathrm{red}}=Y_N\supseteq Y_{N-1}\supseteq ...\supseteq Y_{-1}=\emptyset$$
such that $Y_i\setminus Y_{i-1}\simeq\mathbb{A}_{K}^{a_i}$ and we consider the closure $\overline{Y_i}$ of $Y_i$ in $\mathcal{X}\otimes_{\mathcal{O}_{F}}\mathcal{O}_K$, where $\overline{Y_i}$ is endowed with its structure of reduced closed subscheme. We obtain an isomorphism $$(\overline{Y_i}\setminus \overline{Y_{i-1}})\otimes_{\mathcal{O}_{K}}K\simeq Y_i\setminus Y_{i-1}\simeq \mathbb{A}_{K}^{a_i}$$
and it follows that there exist an open $\textrm{Spec}(\mathcal{O}_{K,S_i})\subseteq\textrm{Spec}(\mathcal{O}_{K})$ and an isomorphism over $\textrm{Spec}(\mathcal{O}_{K,S_i})$
$$(\overline{Y_i}\setminus \overline{Y_{i-1}})\otimes_{\mathcal{O}_{K}}\mathcal{O}_{K,S_i}\simeq \mathbb{A}_{\mathcal{O}_{K,S_i}}^{r_i}.$$
Now we take a finite set $S\supseteq\cup S_i$ big enough so that $\textrm{Spec}(\mathcal{O}_{K,S})\rightarrow \textrm{Spec}(\mathcal{O}_{F,S})$ is an \'etale Galois cover of group $G$. Here $S$ also denotes its image in $\textrm{Spec}(\mathcal{O}_F)$. Then we set $\mathcal{Y}_i=\overline{Y_{i}}\otimes_{\mathcal{O}_{K}}\mathcal{O}_{K,S}$ and we obtain a filtration by (reduced) closed subschemes
$$(\mathcal{X}\otimes_{\mathcal{O}_{F}}\mathcal{O}_{K,S})^{\mathrm{red}}=\mathcal{Y}_N\supseteq \mathcal{Y}_{N-1}\supseteq ...\supseteq \mathcal{Y}_{-1}=\emptyset$$
such that $\mathcal{Y}_i\setminus \mathcal{Y}_{i-1}\simeq\mathbb{A}_{\mathcal{O}_{K,S}}^{r_i}$. But $\Spec(\mathcal{O}_{K,S})$ belongs to $\mathcal{L}(\mathbb{Z})$ by $(\mathcal{L}0)$, hence so does $\mathbb{A}_{\mathcal{O}_{K,S}}^{r_i}$ by $(\mathcal{L}3)$, hence so does $\mathcal{X}\otimes_{\mathcal{O}_{F}}\mathcal{O}_{K,S}$ by induction and $(\mathcal{L}1)$. Notice that a scheme $X\in SFT(\mathbb{Z})$ belongs to $\mathcal{L}(\mathbb{Z})$ if and only if $X^{\mathrm{red}}$ does: this follows from $(\mathcal{L}1)$ since $\emptyset\in \mathcal{L}(\mathbb{Z})$.

Moreover
$$\mathcal{X}\otimes_{\mathcal{O}_{F}}\mathcal{O}_{K,S}\longrightarrow \mathcal{X}\otimes_{\mathcal{O}_{F}}\mathcal{O}_{F,S}$$
is a finite \'etale Galois cover, hence $\mathcal{X}\otimes_{\mathcal{O}_{F}}\mathcal{O}_{F,S}$ lies in $\mathcal{L}(\mathbb{Z})$ by $(\mathcal{L}4)$. Therefore, we have 
$$\forall\p\in S,\,\,\X_{\p}\in\mathcal{L}(\mathbb{Z})\Longleftrightarrow\X_S=\coprod_{\p\in S}\X_{\p}\in\mathcal{L}(\mathbb{Z})\Longleftrightarrow \X\in\mathcal{L}(\mathbb{Z})$$
by $(\mathcal{L}2)$ and $(\mathcal{L}1)$. One may enlarge the finite set $S$ so as to include any given $\p$.
\end{proof}

In order to obtain a more functorial formulation of Conjecture \ref{confBFintro} (which implies $\textbf{B}(\mathcal{X},d)$) we use results and notations of \cite{Holm-Scho11}. We denote by $\widehat{H}^n(\X,\br(p))$  Arakelov motivic cohomology with real coefficients in the sense of \cite{Holm-Scho11} Remark 4.7. These groups are defined following the construction of \cite{Holm-Scho11} using the spectrum $H_{\mathrm{B}}\otimes\br$ instead of $H_{\mathrm{B}}$ so that there is an exact sequence (see \cite{Holm-Scho11} Theorem 4.5 (ii))
$$...\rightarrow \widehat{H}^n(\X,\br(p))\rightarrow H^n(\X,\bq(p))_{\br}\rightarrow H_{\mathcal{D}}^n(X_{/\mathbb{R}},\br(p))\rightarrow...$$
at least for $\X$ an l.c.i. scheme over a number ring in the sense of (\cite{Holm-Scho11} Definition 2.3), where $X=\X_{\bq}$. Here we use the identification $H^n(\X,\bq(p))_{\br}\simeq K_{2p-n}(\X)^{(p)}_{\br}$ for $\mathcal{X}$ regular (\cite{Levine-localization} Theorem 11.7). Let $\X$ be an l.c.i. scheme over a number ring. If $\mathcal{X}$ is moreover proper, regular, connected and $d$-dimensional then Conjecture \ref{confBFintro} for $\X$ is equivalent to
\begin{equation}\label{B-reform}
\widehat{H}^n(\X,\br(d))=0 \textrm{ for }n\neq 2d\textrm{  and }\widehat{H}^{2d}(\X,\br(d))=\br.
\end{equation}
If the (proper, regular, connected and $d$-dimensional) scheme $\mathcal{X}$ lies over a finite field, we say that $\X$  satisfies Conjecture \ref{confBFintro} if (\ref{B-reform}) holds, i.e. if one has
$$H^n(\X,\bq(d))=0 \textrm{ for }n\neq 2d\textrm{  and }H^{2d}(\X,\bq(d))=\bq.$$
Note that for $\mathcal{X}$ (proper, regular, connected and $d$-dimensional) over a finite field one has
\begin{equation}\label{L-implique-B}
\textbf{L}(\mathcal{X}_{et},d)\Rightarrow \mbox{Conjecture \ref{confBFintro} for $\X$}.
\end{equation}

\begin{prop}\label{prop-knowncase-B}
Let $\X$ be a smooth and projective scheme over the number ring $\mathcal{O}_{\X}(\X)=\mathcal{O}_F$. Assume that $\X\in\mathcal{L}(\mathbb{Z})$ and that $\X_{F}$ admits a smooth cellular decomposition. Then $\X$ satisfies Conjecture \ref{confBFintro}. In particular, $\textbf{\emph{B}}(\mathcal{X},d)$ holds, where $d=\mathrm{dim}(\X)$.
\end{prop}

\begin{proof}
By assumption there exist a filtration
$$\mathcal{X}_F=Y_N\varsupsetneq Y_{N-1}\varsupsetneq ...\varsupsetneq Y_{0}\varsupsetneq Y_{-1}=\emptyset$$
by smooth closed subschemes $Y_i$ and isomorphisms $Y_i\setminus Y_{i-1}\simeq\mathbb{A}^{a_i}_F$. As in the proof of Proposition \ref{prop-cellular-L}, one can show that there exist an open subscheme $U\subset\Spec(\mathcal{O}_F)$, a filtration
$$\mathcal{X}_U=\mathcal{Y}_N\varsupsetneq \mathcal{Y}_{N-1}\varsupsetneq ...\varsupsetneq \mathcal{Y}_{0}\varsupsetneq \mathcal{Y}_{-1}=\emptyset$$
and $U$-isomorphisms $\mathcal{Y}_i\setminus\mathcal{Y}_{i-1}\simeq\mathbb{A}^{a_i}_U$ such that the following holds. One has $\mathcal{Y}_i\otimes_UF\simeq Y_i$, the scheme $\mathcal{Y}_{i}$ is a closed subscheme of $\mathcal{Y}_{i+1}$, and $\mathcal{Y}_{i}$ is smooth over $U$.

By (\cite{Liu} Corollary 5.3.17) $\X_F$ is smooth and connected, hence irreducible. This gives $CH^0(\X_F)=\mathbb{Z}$. On the other hand, there is a direct sum decomposition $h(\mathcal{X}_F)=\bigoplus_{0\leq i\leq N} h(F)(-a_i)$ in the category of Chow motives (\cite{Brosnan05} Theorem 3.1). It follows that there exists a unique index $0\leq i_0\leq N$ such that $a_{i_0}=0$. But $\X_F$ is proper (over $F$), hence so is $Y_0\simeq \mathbb{A}^{a_0}_F$, hence $i_0=0$.

The (absolute) dimension yields a locally constant map $d_i:\mathcal{Y}_i\rightarrow \mathbb{Z}$. We define
$c_i:\mathcal{Y}_i\rightarrow \mathbb{Z}$ similarly: Let $y\in \mathcal{Y}_i$ and let $\mathcal{Y}_{i,y}$ (resp. $\mathcal{Y}_{i+1,y}$) be the connected component of $\mathcal{Y}_i$ (resp. of $\mathcal{Y}_{i+1}$) containing $y$. Then we define $c_i(y)=\textrm{codim}(\mathcal{Y}_{i,y},\mathcal{Y}_{i+1,y})$ and
$$\widehat{H}^{n-2c_{i}}(\mathcal{Y}_{i},\br(d_{i})):=\bigoplus_{y\in\pi_0(\mathcal{Y}_{i})}\widehat{H}^{n-2c_{i}(y)}(\mathcal{Y}_{i,y},\br(d_{i}(y)))$$
For any $0\leq i\leq N$ we have a long exact sequence (see \cite{Holm-Scho11} Theorem 4.16 (iii))
$$...\rightarrow \widehat{H}^{n-2c_{i-1}}(\mathcal{Y}_{i-1},\br(d_{i-1}))\rightarrow\widehat{H}^n(\mathcal{Y}_{i},\br(d_{i}))\rightarrow \widehat{H}^n(\mathcal{Y}_{i}\setminus \mathcal{Y}_{i-1},\br(d_i))\rightarrow...$$
Notice that the locally constant function $d_i$ is constant on $\mathcal{Y}_{i}\setminus \mathcal{Y}_{i-1}\simeq \mathbb{A}^{a_i}_U$ with constant value $a_i+1$. But for $i\geq1$, one has $a_i>0$ and
$$\widehat{H}^n(\mathcal{Y}_{i}\setminus \mathcal{Y}_{i-1},\br(d_i))\simeq
\widehat{H}^n(\mathbb{A}^{a_i}_U,\br(a_i+1))=\widehat{H}^n(U,\br(a_i+1))=0\mbox{ for all $n$}.$$
Indeed, the second equality is given by (\cite{Holm-Scho11} Theorem 4.16 (ii)) and the vanishing of $\widehat{H}^n(U,\br(a_i+1))$ for $a_i>0$ follows from the fact that $U$ is of the form $\Spec(\mathcal{O}_{F,S})$ since the Beilinson regulator
$$H^n(\Spec(\mathcal{O}_{F,S}),\bq(a_i+1))_{\br}\stackrel{\sim}{\longrightarrow} H^n(\Spec(F),\bq(a_i+1))_{\br}\stackrel{\sim}{\longrightarrow} H^n_{\mathcal{D}}(\Spec(F)_{/\br},\br(a_i+1))$$
is an isomorphism for $a_i>0$ and all $n$.
We obtain an identity
$$\widehat{H}^{n-2c_{i-1}}(\mathcal{Y}_{i-1},\br(d_{i-1}))\simeq\widehat{H}^n(\mathcal{Y}_{i},\br(d_i))\mbox{ for all $n$}.$$
Notice that $\mathcal{Y}_0\simeq V$ and that $d_0$ and $\sum_{i=0}^{i=N-1} c_i$ are both constant on $\mathcal{Y}_0$  with constant value $1$ and $d-1$ respectively. An induction on $i$ yields
\begin{equation}\label{ici}
\widehat{H}^n(\X_U,\br(d))\simeq \widehat{H}^{n-2d+2}(\mathcal{Y}_0,\br(1))\simeq \widehat{H}^{n-2d+2}(U,\br(1))\mbox{ for all $n$}.
\end{equation}
Using (\ref{ici}), the fact that $U$ is of the form $\Spec(\mathcal{O}_{F,S})$ (with $S\neq\emptyset$) and well known facts concerning the Dirichlet regulator, we obtain $\widehat{H}^n(\X_U,\br(d))=0$ for $n\neq 2d-1$ and an exact sequence
$$0\rightarrow \widehat{H}^{2d-1}(\X_U,\br(d))\rightarrow \Prod_{\p\in S}\br\rightarrow\br\rightarrow 0$$
where $S$ denotes the closed complement of $U$ in $\Spec(\mathcal{O}_F)$. For any finite prime $\p$ of $F$, $\mathcal{X}_{\p}$ is (geometrically) connected (see \cite{Liu} Corollary 5.3.17) of dimension $d-1$. By Propositions \ref{prop-cellular-L} and \ref{prop-LZimpliesL}, $\X_{\p}$ satisfies $\textbf{L}(\mathcal{X}_{\p,et},d-1)$ hence $\X_{\p}$ satisfies Conjecture \ref{confBFintro} by (\ref{L-implique-B}). We set $\X_S:=\coprod_S\X_{\p}$. The fact that $\widehat{H}^{n-2}(\X_S,\br(d-1))=0$ for $n\neq2d$ and the long exact sequence
$$...\rightarrow\widehat{H}^{n-2}(\X_S,\br(d-1))\rightarrow\widehat{H}^n(\X,\br(d))\rightarrow \widehat{H}^n(\X_U,\br(d))\rightarrow...$$
then give $\widehat{H}^n(\X,\br(d))=0$ for $n\neq2d,2d-1$. The result follows from the fact that there is a morphism of exact sequences
{\small{
\[ \xymatrix{
0\ar[r]&\widehat{H}^{2d-1}(\X,\br(d))\ar[d]\ar[r]&\widehat{H}^{2d-1}(\X_U,\br(d))\ar[d]^{=}\ar[r]&\widehat{H}^{2d-2}(\X_S,\br(d-1))\ar[d]^{\simeq} \ar[r]&\widehat{H}^{2d}(\X,\br(d)) \ar[r]\ar[d]&0   \\
0\ar[r]& 0\ar[r]&\widehat{H}^{2d-1}(\X_U,\br(d))\ar[r]&\Prod_{\p\in S}\br \ar[r]&\br \ar[r]&0
}
\]}}
\end{proof}

\vspace{0.5cm}

{\bf {Acknowledgments.}} I am very grateful to C. Deninger for comments and advices and to M. Flach for many discussions related to Weil-\'etale cohomology. I would also like to thank T. Chinburg,  T. Geisser and S. Lichtenbaum for comments and discussions. This work is based on ideas of S. Lichtenbaum. Finally, I would like thank the referees for many constructive suggestions and comments. During the preparation of this work, the author was supported by the Marie Curie fellowship 253346.


\begin{thebibliography}{48}
\bibitem{Benois-Nguyen-Quang-Do-02}{Benois, D.; Nguyen Quang Do, T.: \emph{Les nombres de Tamagawa locaux et la conjecture de Bloch et Kato pour les motifs $\mathbb{Q}(m)$ sur un corps ab\'elien.} Ann. Sci. \'Ecole Norm. Sup. (4) 35 (2002), no. 5, 641--672.}
\bibitem{Bienenfeld}{Bienenfeld, M.: \emph{An \'etale cohomology duality theorem for number fields with a real embedding.}  Trans. Amer. Math. Soc.  303 (1) (1987), 71--96.}
\bibitem{Bloch86}{Bloch, S.: \emph{Algebraic cycles and the Beilinson conjectures.} The Lefschetz centennial conference, Part I (Mexico City, 1984), 65--79, Contemp. Math., 58, Amer. Math. Soc., Providence, RI, 1986.}
\bibitem{Borel74}{Borel, A.: \emph{Stable real cohomology of arithmetic groups.} Ann. Sci. \'Ecole Norm. Sup. (4) 7 (1974), 235--272 (1975).}
\bibitem{Brosnan05}{Brosnan, P.: \emph{On motivic decompositions arising from the method of Bialynicki-Birula.} Invent. Math. 161 (2005), no. 1, 91--111.}
\bibitem{Burgos-Gil}{Burgos Gil, J. I.: \emph{Semipurity of tempered Deligne cohomology.}  Collect. Math.  59 (1) (2008), 79--102.}
\bibitem{Burgos-GoncharovRegulator}{Burgos Gil, J. I., Feliu,  E., Takeda, Y.: \emph{On Goncharov's Regulator and Higher Arithmetic Chow Groups.} To appear in  Int. Math. Res. Not. IMRN.}
\bibitem{Burns}{Burns, D.: \emph{Perfecting the nearly perfect.}  Pure Appl. Math. Q. 4 (4) (2008), 1041--1058.}
\bibitem{Burns-Flach-98}{Burns, D., Flach, M.: \emph{On Galois structure invariants associated to Tate motives.} Amer. J. Math. 120 (1998), no. 6, 1343--1397.}
\bibitem{Deligne74}{Deligne, P.: \emph{La conjecture de Weil. I.}
Inst. Hautes \'Etudes Sci. Publ. Math. No. 43 (1974), 273--307.}
\bibitem{Deninger08}{Deninger, C.: \emph{Analogies between analysis on foliated spaces and arithmetic geometry.} Groups and analysis, 174--190, London Math. Soc. Lecture Note Ser., 354, Cambridge Univ. Press, Cambridge, 2008.}
\bibitem{Flach04}{Flach, M.: \emph{The equivariant Tamagawa number conjecture: a survey.} With an appendix by C. Greither. Contemp. Math., 358, Stark's conjectures: recent work and new directions, 79--125, Amer. Math. Soc., Providence, RI, 2004.}
\bibitem{MatFlach}{Flach, M.: \emph{Cohomology of topological groups with applications to the Weil group.}
Compositio Math. 144 (3) (2008), 633--656.}
\bibitem{Flach-moi}{Flach, M., Morin, B.: \emph{On the Weil-\'etale topos of regular arithmetic schemes.} Doc. Math. 17 (2012) 313--399.}
\bibitem{Fontaine92}{Fontaine, J-M.: \emph{Valeurs sp\'eciales des fonctions $L$ des motifs.} S\'eminaire Bourbaki, Vol. 1991/92. Ast\'erisque No. 206 (1992), Exp. No. 751, 4, 205--249.}
\bibitem{Fontaine-Perrin-Riou}{Fontaine, J-M., Perrin-Riou, B.: \emph{Autour des conjectures de Bloch et Kato : cohomologie galoisienne et valeurs de fonctions L.} Motives (Seattle, WA, 1991), 599--706, Proc. Sympos. Pure Math., 55, Part 1, Amer. Math. Soc., Providence, RI, 1994.}
\bibitem{Geisser-MotvicCohDed}{Geisser, T.: \emph{Motivic cohomology over Dedekind rings.}  Math. Z. 248 (4) (2004), 773--794.}
\bibitem{Geisser-Weiletale}{Geisser, T.: \emph{Weil-\'etale cohomology over finite fields.}  Math. Ann.  330 (4) (2004), 665--692.}
\bibitem{Geisser-MotCoh-K-theory}{Geisser, T.: \emph{Motivic cohomology, K-theory and topological cyclic homology.}  Handbook of K-theory. Vol. 1,  193--234.}
\bibitem{Geisser-Duality}{Geisser, T.: \emph{Duality via cycle complexes.}  Ann. of Math. (2) 172 (2) (2010), 1095--1126.}
\bibitem{Geisser-Levine00}{Geisser, T., Levine, M.: \emph{The K-theory of fields in characteristic p.} Invent. Math. 139 (2000), no. 3, 459--493.}
\bibitem{Geisser-Levine01}{Geisser, T., Levine, M.: \emph{The Bloch-Kato conjecture and a theorem of Suslin-Voevodsky.} J. Reine Angew. Math. 530 (2001), 55--103.}
\bibitem{Goncharov}{Goncharov, A. B.: \emph{Polylogarithms, regulators, and Arakelov motivic complexes.}  J. Amer. Math. Soc. 18 (1) (2005), 1--60.}
\bibitem{SGA4}{Grothendieck, A., Artin, M., Verdier, J.L.: \emph{Th\'eorie des Topos et cohomologie \'etale des sch\'emas
(SGA4).} Lectures Notes in Math. 269, 270, 305, Springer, 1972.}
\bibitem{Holm-Scho11}{Holmstrom, A., Scholbach, J.:\emph{Arakelov motivic cohomology I}. Preprint 2011.}
\bibitem{Kahn-K-theory}{Kahn, B.: \emph{Algebraic K-theory, algebraic cycles and arithmetic geometry.} Handbook of K-theory. Vol. 1, 351--428, Springer, Berlin, 2005.}
\bibitem{Kahn-equivalence-rationnelle} {Kahn, B.: \emph{\'Equivalences rationnelle et num\'erique sur certaines vari\'et\'es de type ab\'elien sur un corps fini.}  Ann. Sci. \'ecole Norm. Sup. (4)  36 (6) (2003), 977--1002.}
\bibitem{Kato-Saito86}{Kato, K., Saito, S.: \emph{Global class field theory of arithmetic schemes.}
Contemp. Math., 55, 255--331, Amer. Math. Soc., Providence, RI, 1986.}

\bibitem{Knudsen-Mumford}{Knudsen, F., Mumford, D.: \emph{The projectivity of the moduli space of stable curves I: Preliminaries on `det' and `Div'.} Math. Scand. 39 (1976), 19--55.}
\bibitem{Kolster-Sands08}{Kolster, M., Sands, J.W.: \emph{Annihilation of motivic cohomology groups in cyclic 2-extensions.}
Ann. Sci. Math. Qu\'ebec 32 (2) (2008), 175--187.}
\bibitem{Levine-motivic-and-K-theory}{Levine, M.: \emph{K-theory and motivic cohomology of schemes.}} Preprint (1999). http://www.math.uiuc.edu/K-theory/336/
\bibitem{Levine-localization}{Levine, M.: \emph{Techniques of localization in the theory of algebraic cycles.} J. Algebraic Geom. 10 (2) (2001), 299--363.}
\bibitem{Lichtenbaum73}{Lichtenbaum, S.: \emph{Values of zeta-functions, \'etale cohomology, and algebraic K-theory.} Lecture Notes in Math., Vol. 342, 489--501, Springer, Berlin, 1973.}
\bibitem{Lichtenbaum-finite-field}{Lichtenbaum, S.: \emph{The Weil-\'etale topology on schemes over finite fields.}
Compositio Math.  141 (3) (2005), 689--702.}
\bibitem{Lichtenbaum}{Lichtenbaum, S.: \emph{The Weil-\'etale topology for number rings.} Ann. of Math. (2)  170 (2) (2009), 657--683.}
\bibitem{Liu}{Liu, Q. \emph{Algebraic geometry and arithmetic curves.} Oxford Graduate Texts in Mathematics, 6. Oxford Science Publications. Oxford University Press, Oxford, 2002.}
\bibitem{Milne-duality}{Milne, J.: \emph{Arithmetic duality theorems.} Perspectives in Mathematics 1, Academic Press, Inc., Boston, Mass., 1996.}
\bibitem{On the WE}{Morin, B.: \emph{On the Weil-\'etale cohomology of number fields.} Trans. Amer. Math. Soc. 363 (2011), 4877--4927.}
\bibitem{Fund-group-II}{Morin, B.: \emph{The Weil-\'etale fundamental group of a number field II.} Selecta Math. (N.S.) 17 (1) (2011), 67--137.}
\bibitem{Rognes-Weibel00}{Rognes, J.; Weibel, C.: \emph{Two-primary algebraic K-theory of rings of integers in number fields.} Appendix A by Manfred Kolster. J. Amer. Math. Soc. 13 (2000), no. 1, 1--54.}
\bibitem{Soule79}{Soul\'e, C.: \emph{K-th\'eorie des anneaux d'entiers de corps de nombres et cohomologie \'etale.} Invent. Math. 55 (1979), no. 3, 251--295.}
\bibitem{Soule}{Soul\'e, C.: \emph{K-th\'eorie et z\'eros aux points entiers de fonctions z\'eta.} Proceedings of the International Congress of Mathematicians, Vol. 1, 2 (Warsaw, 1983),  437--445.}
\bibitem{Soul84}{Soul\'e, C.: \emph{Groupes de Chow et K-th\'eorie de vari\'et\'es sur un corps fini.} Math. Ann. 268 (1984), no. 3, 317--345.}
\bibitem{Swan60}{Swan, R.G.: \emph{A new method in fixed point theory.} Comment. Math. Helv. 34 (1960), 1--16.}
\bibitem{Verdier96}{Verdier, J.L.: \emph{Des cat\'egories d\'eriv\'ees des cat\'egories ab\'eliennes.} Ast\'erisque No. 239 (1996).}
\bibitem{Voevodsky11}{Voevodsky, V.: \emph{On motivic cohomology with Z/l-coefficients.} Ann. of Math. (2) 174 (2011), no. 1, 401--438.}
\bibitem{Weibel-Handbook}{Weibel, C.: \emph{Algebraic K-theory of Rings of Integers in Local and Global Fields}. Handbook of K-theory. Vol. 1,  139--190.}
\bibitem{Wiesend07}{Wiesend, G.: \emph{Class field theory for arithmetic schemes.} Math. Z. 256 (2007), no. 4, 717--729.}
\end{thebibliography}
\end{document}